\documentclass[12pt,notitlepage]{amsart}
\usepackage{latexsym,amsfonts,amssymb,amsmath,amsthm}
\usepackage{color}
\usepackage[inner=1.0in,outer=1.0in,bottom=1.0in, top=1.0in]{geometry}

\newcommand{\mz}{\ensuremath{\mathbb Z}}
\newcommand{\mr}{\ensuremath{\mathbb R}}
\newcommand{\mh}{\ensuremath{\mathbb H}}

\newcommand{\mc}{\ensuremath{\mathbb C}}

\newcommand{\shortmod}{\ensuremath{\negthickspace \negthickspace \negthickspace \pmod}}

\newcommand{\half}{\ensuremath{ \frac{1}{2}}}

\newcommand{\intR}{\int_{-\infty}^{\infty}}

\newcommand{\thalf}{\tfrac12}
\newcommand{\sumstar}{\sideset{}{^*}\sum}

\newcommand{\e}[2]{e\left(\frac{#1}{#2}\right)}

\theoremstyle{plain}		
	\newtheorem{mytheo}{Theorem} [section]
	
	\newtheorem{myprop}[mytheo]{Proposition}
	
     \newtheorem{mylemma}[mytheo]{Lemma}
	
	\newtheorem{myconj}[mytheo]{Conjecture}

\theoremstyle{remark}

\numberwithin{equation}{section}
\begin{document}

% \author{Sheng-Chi Liu} 
% \address{Department of Mathematics \\
% 	  Texas A\&M University \\
% 	  College Station \\
% 	  TX 77843-3368 \\
% 		U.S.A.}
% \email{scliu@math.tamu.edu}
% 
\author{Matthew P. Young} 
\address{Department of Mathematics \\
	  Texas A\&M University \\
	  College Station \\
	  TX 77843-3368 \\
		U.S.A.}
\email{myoung@math.tamu.edu}
\thanks{This material is based upon work supported by the National Science Foundation under agreement No. DMS-1101261.  Any opinions, findings and conclusions or recommendations expressed in this material are those of the authors and do not necessarily reflect the views of the National Science Foundation.}

\begin{abstract}
We consider some analogs of the quantum unique ergodicity conjecture for geodesics, horocycles, or ``shrinking'' families of sets.  In particular, we prove the analog of the QUE conjecture for Eisenstein series restricted to the infinite geodesic connecting $0$ and $\infty$ inside the modular surface.

\end{abstract}
 \title{The quantum unique ergodicity conjecture for thin sets}
 \maketitle
 \section{Introduction}
The quantum unique ergodicity (QUE) conjecture of Rudnick and Sarnak \cite{RudnickSarnak} is an equidistribution statement for Hecke-Maass forms of large Laplace eigenvalue on the modular surface $\Gamma \backslash \mh$, $\Gamma = PSL_2(\mz)$.  It says that if $U(z)$ is a Hecke-Maass form 
and $\phi$ is a fixed, smooth and compactly-supported function on $\Gamma \backslash \mathbb{H}$, then
\begin{equation}
\label{eq:QUE}
 \int_{\Gamma \backslash \mathbb{H}} |U(z)|^2 \phi(z) \frac{3}{\pi} \frac{dx dy}{y^2} \rightarrow %\Big(\int_{\Gamma \backslash \mathbb{H}} |U(z)|^2  \frac{3}{\pi} \frac{dx dy}{y^2} \Big) \Big( 
\int_{\Gamma \backslash \mathbb{H}} \phi(z)  \frac{3}{\pi} \frac{dx dy}{y^2} %\Big)
.
\end{equation}
as the Laplace eigenvalue of $U$ tends to infinity, provided $U$ is  normalized with probability measure
\begin{equation}
\label{eq:L2normalization}
 \int_{\Gamma \backslash \mathbb{H}} |U(z)|^2 \frac{3}{\pi} \frac{dx dy}{y^2} =1.
\end{equation}
Lindenstrauss \cite{Lindenstrauss} recently proved the QUE conjecture in the compact setting (and proved it in the non-compact case except for the possibility of ``escape of mass'' at the cusp), and Soundararajan \cite{SoundQUE} completed the proof in this non-compact case.  The mass equidistribution conjecture \cite{RudnickSarnak} is the analog of the QUE conjecture where $U(z) = y^{k/2} f(z)$ for $f$ a weight $k$ holomorphic Hecke cusp form, with $k \rightarrow \infty$.  It was proven by Holowinsky and Soundararajan \cite{HolowinskySound}.

In this paper, we investigate the possibility of equidistribution of $U(z)$ restriced to some ``thin'' sets, such as a geodesic, a horocycle, or a shrinking family of discs.  Since the rate of convergence in \eqref{eq:QUE} is either not known (in the Maass case) or rather slow (in the holomorphic case, where the error term is a small power of $\log{k}$), it seems unlikely using current technology to prove QUE for the restrictions of Hecke cusp forms.  Instead, one of our goals has been to find natural and (presumably) reliable conditions that imply QUE in these extreme cases.  For Maass forms, questions of this type were raised and studied numerically by Hejhal and Rackner \cite{HejhalRackner}.  Very recently, Ghosh, Reznikov, and Sarnak \cite{GRS} have proven strong upper and lower bounds for the $L^2$ norm of Hecke-Maass cusp forms restricted to geodesics or horocycles, with the application of proving (on the Lindel\"{o}f hypothesis) that the number of nodal domains goes to infinity with the Laplace eigenvalue.  

Quantum ergodicity states that $\eqref{eq:QUE}$ holds for a density one subsequence of $U$'s \cite{ZelditchQE} \cite{ZelditchEisenstein}, and has been extended in many different directions.  The analog of quantum ergodicity  for restricted eigenfunctions has also recently been studied, for instance, see \cite{TothZelditchI} \cite{TothZelditchII}. 

One particularly interesting example, having extra structure, is the vertical geodesic connecting $0$ and $i \infty$.  
\begin{myconj}
\label{conj:geodesicQUE}
 Suppose that $\psi: \mr^{+} \rightarrow \mr$ is a smooth, compactly-supported function.  Then
\begin{equation}
\label{eq:geodesicQUE}
 \lim_{j \rightarrow \infty} \int_0^{\infty} |u_j(iy)|^2 \psi(y) \frac{dy}{y} =  \int_0^{\infty} 2\psi(y) \frac{dy}{y},
\end{equation}
where $u_j$ runs over the even Hecke-Maass forms, normalized by \eqref{eq:L2normalization}.  Furthermore,
\begin{equation}
\label{eq:geodesicMassEquidistribution}
 \lim_{k \rightarrow \infty} \int_0^{\infty} y^k |f(iy)|^2 \psi(y) \frac{dy}{y} =  \int_0^{\infty} \psi(y) \frac{dy}{y},
\end{equation}
where $f(z)$ runs over weight $k$ holomorphic Hecke cusp forms, $L^2$-normalized with probability measure.
\end{myconj}
Remarks.  Conjecture \ref{conj:geodesicQUE} says that QUE should hold for functions restricted to the geodesic joining $0$ and $i \infty$.  Notice that even Maass forms are predicted to be twice as big on this geodesic as on the fundamental domain.
%We should emphasize that this $2$ apparently arises because we are comparing normalizations on two different spaces.  
One might naturally speculate that the factor of $2$ in \eqref{eq:geodesicQUE} arises from the fact that the odd Maass forms all vanish along the geodesic and so the even Maass forms have to be twice as large (on average) to account for this disparity.  One can see this type of behavior from the Selberg pre-trace formula where $\sum_{j} |u_j(z)|^2 h(t_j) \sim \sum_{j} h(t_j)$ for $z \in \mathbb{H}$ a fixed non-elliptic point, and for certain classes of weight functions $h$; cf. p.179 of \cite{IwaniecSpectral}.  The point is that this spectral sum includes both the even and odd Maass forms, but if $u_j(z) = 0$ for all the odd forms, say, then $|u_j(z)|^2$ has to be $2$ on average over the even forms.
In the course of the derivation, the only difference between \eqref{eq:geodesicQUE} and \eqref{eq:geodesicMassEquidistribution} is that the Fourier expansion of a Maass form has Fourier coefficients at both negative and positive integers.  We do not have a conjecture for the precise size of the error terms in these two asymptotics, but we expect a power saving based on the integral moment conjectures of Conrey, Farmer, Keating, Rubinstein, and Snaith \cite{CFKRS}.  It would be interesting to prove an Omega-type result for the error terms.

Conjecture \ref{conj:geodesicQUE} is apparently more difficult than the usual QUE conjecture.  Geometrically, since the domain is $1$-dimensional (hence, of measure $0$), it is not obviously implied by \eqref{eq:QUE}.  
Actually, in the Maass case, \eqref{eq:geodesicQUE} is tantalizingly close to the usual QUE.  The reason is that after applying harmonic analysis on the positive reals, \eqref{eq:geodesicQUE} becomes related to the second moment of the $L$-function associated to $U$; see \eqref{eq:secondmomentformulaforIU} below for the exact formula.  The problem then becomes related to solving the shifted convolution problem $\sum_{n \approx t_j} \lambda(n) \lambda(n+m)$ with $m \ll t_j^{\varepsilon}$, and with a smooth weight function.  Ghosh, Reznikov, and Sarnak explain (see \cite{GRS} Appendix A) that QUE only implies bounds on shifted convolution sums with a certain class of weight functions, basically those arising from incomplete Poincare series, and this class is not rich enough for many purposes.  The holomorphic case \eqref{eq:geodesicMassEquidistribution} is quite different from the Maass case.  In this case, Luo and Sarnak showed that the mass equidistribution theorem does imply a bound on the shifted convolution sum with an arbitrary fixed smooth weight function \cite{LuoSarnakQUE}, but unfortunately \eqref{eq:geodesicMassEquidistribution} reduces to a much more difficult shifted convolution sum, one roughly of the form $S=\sum_{m \approx \sqrt{k}} \sum_{n \approx k} \lambda(n) \lambda(n+m)$, and it is required to show $S = o(k)$.

%It turns out that \eqref{eq:geodesicQUE} and \eqref{eq:geodesicMassEquidistribution} are of very different levels of difficulty.  
See \cite{GRS} for upper and lower bounds on integrals of the form \eqref{eq:geodesicQUE}; %which also utilizes the QUE theorem; 
these bounds are close to the right order of magnitude.  
For the holomorphic case, see \cite{BKY} for an upper bound on \eqref{eq:geodesicMassEquidistribution} of the size $k^{1/4 + \varepsilon}$ while a bound of the form $k^{\varepsilon}$ would prove that $L(f, 1/2) \ll k^{1/4 + \varepsilon}$ which would be stronger than any known subconvexity bound for any $L$-function (here the conductor is $k^2$).

It is interesting to compare the tools needed to derive the usual QUE conjecture with the geodesic version.  We briefly sketch 
a standard way to approach the proof of \eqref{eq:QUE}; see Section \ref{section:tripleproductapproach} for a more in-depth discussion. Using the spectral decomposition, write $\langle |U|^2, \phi \rangle = \langle |U|^2, \frac{3}{\pi} \rangle \langle 1, \phi \rangle + \sum_{j} \langle |U|^2, u_j \rangle \langle u_j, \phi \rangle + \dots$ with the dots indicating the continuous spectrum.  The constant eigenfunction provides a main term.  An integration by parts argument shows that the spectral coefficients satisfy the bound $|\langle u_j, \phi \rangle| \leq C(A) (1/4 + t_j^2)^{-A}$ where $t_j$ is the spectral parameter of $u_j$, and $A > 0$ arbitrary, so that the 
spectral sum can be effectively truncated almost instantly.  Watson's formula \cite{Watson} relates $|\langle |U|^2, u_j \rangle |^2$ to the central value of a triple product $L$-function.  The convexity bound for the $L$-function then barely fails to prove \eqref{eq:QUE} and a subconvexity bound would succeed.  The Lindel\"{o}f hypothesis is known to give an 
optimal error term by work of Luo and Sarnak \cite{LuoSarnakQUE} \cite{LuoSarnakQuantumVariance}.  It is a common theme, nicely illustrated in this example, that subconvexity and equidistribution are often closely 
related (almost equivalent).  %However, thin equidistribution of the type studied in this paper is not known to be directly related to subconvexity.  Instead, we require asymptotics for mean values of $L$-functions or robust shifted convolution sum estimates.  

As in the sketch of the proof of QUE outlined above, one is naturally inclined to apply the spectral decomposition to approach Conjecture \ref{conj:geodesicQUE}, but this is not effective.
The problem can already be anticipated in \eqref{eq:geodesicQUE} where the constant eigenfunction only gives half the main term so one must expect to extract a main term from the spectral sum\footnote{For odd Maass forms, the constant term gives a main term which must be cancelled by the spectral sum.}.  The other issue is that the spectral sum is very ``long'' (the 
spectral coefficients do not immediately rapidly decay anymore) and there must be cancellation in the sum over the spectral coefficients.  The {\em squares} of the spectral coefficients are given as triple product $L$-functions, and the 
change in sign of these coefficients is difficult to detect.  Rather than using the spectral decomposition, one may use Parseval (for $\mathbb{R}^+$) to relate the geodesic integral to a second moment of $L$-functions.  Lindel\"{o}f here gives an upper bound of the right order of magnitude but the asymptotic is a more subtle issue and is not known to follow from GRH.  
See \cite{HarperZeta} following \cite{SoundZeta} for the best results on 
upper bounds on moments assuming GRH.  The five authors'  conjecture \cite{CFKRS} predicts an asymptotic for this second moment of $L$-functions.  
We noticed that the main term in this asymptotic could then be expressed as $\langle |U|^2, H(z,\psi) \rangle$ for some nice $\Gamma$-invariant function $H$ (almost an incomplete Eisenstein series-- see \eqref{eq:Hformula} for the exact formula).  Finally, we use the 
QUE/mass equidistribution theorems to relate this inner product to $\langle 1, H(z,\psi) \rangle$ which then leads immediately to \eqref{eq:geodesicQUE} and \eqref{eq:geodesicMassEquidistribution}.  

Our main (unconditional) result in this paper is the following
\begin{mytheo}
\label{thm:geodesicEisenstein}
Suppose $\psi$ is a smooth, compactly-supported function on $\mathbb{R}^+$.  We have
\begin{equation}
\frac{\pi}{3 \log(1/4 +T^2)} \int_0^{\infty}  |E(iy, 1/2 + iT)|^2 \psi(y) \frac{dy}{y} \sim  \int_0^{\infty} 2 \psi(y) \frac{dy}{y},
\end{equation}
as $T \rightarrow \infty$.
\end{mytheo}
Theorem \ref{thm:geodesicEisenstein} is inspired by an analogous result of Luo-Sarnak \cite{LuoSarnakQUE}  who showed that $\frac{3}{\pi} \log(1/4+T^2)$ is the average value of $|E(z, 1/2 + iT)|^2$ restricted to any fixed compact Jordan measurable subset of $\Gamma \backslash \mh$ having positive measure (here we have corrected the constant in place of $\frac{48}{\pi} \log{T}$ as stated by \cite{LuoSarnakQUE}; for a correct statement, see \cite{HejhalRackner}, (7.9) or \cite{SpinuL4norm}, (1.1)).  Then we interpet Theorem \ref{thm:geodesicEisenstein} as saying the Eisenstein series is twice as big (on average) on any segment of the geodesic as on the rest of the fundamental domain, just as is predicted by Conjecture \ref{conj:geodesicQUE} for even Maass forms. In fact, we appeal to the Luo-Sarnak result just as in the derivation of Conjecture \ref{conj:geodesicQUE} we appeal to the QUE theorem.  
Moreover, we prove a more precise version of the asymptotic with a power saving in the error term; see Theorem \ref{thm:geodesicEisenstein2} for this result.
Since we have a power saving, this indicates that we could probably estimate some variations such as by having $\psi$ vary with $T$ in some way--see Proposition \ref{prop:thinEisensteinQUE} for an example of what is meant here.
See also \cite{Koyama} for the level aspect of QUE for Eisenstein series.

The Eisenstein series is often an interesting test case for more advanced problems with Maass forms.  For instance, Spinu \cite{SpinuL4norm} bounded the $L^4$ norm of the Eisenstein series when restricted to a compact set. 

Our proof of Theorem \ref{thm:geodesicEisenstein} requires the full spectral theory of automorphic forms.  As mentioned earlier in the introduction, it is not effective to directly apply the spectral decomposition to $|E(iy, 1/2 + iT)|^2$ and then integrate over $y$, because the spectral sum is very ``long.''  Instead, we use Parseval to relate it to the fourth moment of the Riemann zeta function which of course has been extensively studied over many years, e.g. see \cite{Ingham26} \cite{Ramachandra} \cite{HB79} \cite{IwaniecFourthMoment} \cite{Motohashi} for some notable results.  However, the present case is somewhat unusual in that it samples the zeta function at widely-separated points.  Here it is roughly of the form
\begin{equation}
 \int_{|t| \leq T} (1 + |t-T|)^{-1/2} (1 + |t+T|)^{-1/2} |\zeta(1/2 + it + iT) \zeta(1/2 + it - iT)|^2 dt.
\end{equation}
Furthermore, there are ranges in the integral where the conductor drops, e.g. for $T-2\Delta \leq t \leq T-\Delta$ with $\Delta = o(T)$, in which case the weight becomes larger.  Luckily, the measure of the set where the conductor drops is also relatively small, so these effects partially negate each other.  In fact, we can use short-interval bounds on the fourth moment of zeta (e.g., see \eqref{eq:IwaniecFourthMoment} below) to dispense with the conductor-dropping ranges, so the main issue is to understand the range $|t| \leq .99 T$.
In a related direction, Bettin \cite{Bettin} considered large shifts in the second moment of zeta.  See \cite{Chandee} for strong bounds on shifted moments assuming GRH.

Our basic approach is to use an approximate functional equation for the product of zeta functions, leading to the problem of asymptotically evaluating the shifted divisor sum $\sum_{n \approx T} \tau_{iT}(n) \tau_{iT}(n+m)$, with a smooth weight.  For this we use spectral theory in the guise of the Kuznetsov formula.  It is an important point that we only use the spectral tools on the favorable part of the integral with $|t| \leq .99T$.  It might be illuminating to relate this favorable portion of the $t$-integral directly to an inner product of $|E(z, 1/2 + iT)|^2$ with Poincare series.  See \cite{GoodSquareMean} \cite{GoodConvolution} \cite{EHS} for some clues in this direction.  One curious feature of the proof is that after applying the spectral theory to the shifted divisor sum, its estimation comes down to subconvexity for the Riemann zeta function and Hecke-Maass $L$-functions.  We are nevertheless reluctant to say that the geodesic QUE theorem for Eisenstein series is equivalent to subconvexity, because in fact we require the full spectral machinery to even get to this point, and because these tools themselves lead to strong subconvexity bounds.  Furthermore, to treat the ranges with $\Delta = o(T)$ requires a short interval mean value bound for the zeta function which itself implies subconvexity (but not vice-versa).

Another interesting quantity along the lines of Conjecture \ref{conj:geodesicQUE} is the
 $L^2$-norm of $u_j(iy)$ along the geodesic, that is, the case $\psi(y) = 1$ for all $y$ (a problem considered by \cite{GRS}).  Although this problem should be easier than the geodesic integral weighted by $\psi$ (because it is a longer average), 
Conjecture \ref{conj:geodesicQUE} breaks down in this case.  In this direction, we have
\begin{myconj}
\label{conj:L2geodesic}
 For $\alpha$ fixed, $|\text{Re}(\alpha)| < 1/2$, and $u_j$ an even Hecke-Maass cusp form, we have
\begin{equation}
\label{eq:geodesicMaassalphafixed}
\int_0^{\infty} |u_j(iy)|^2 y^\alpha \frac{dy}{y} \sim 2 \int_{\Gamma \backslash \mh} |u_j(z)|^2 (E(z, 1 + \alpha) + E(z,1-\alpha)) \frac{dx dy}{y^2},
\end{equation}
as $j \rightarrow \infty$.  Similarly, 
\begin{equation}
\label{eq:geodesicHolomorphicalphafixed}
 \int_0^{\infty} y^k |f(iy)|^2 y^{\alpha} \frac{dy}{y} \sim \int_{\Gamma \backslash \mh} y^k |f(z)|^2 (E(z, 1 + \alpha) + E(z,1-\alpha)) \frac{dx dy}{y^2},
\end{equation}
as $k \rightarrow \infty$.
\end{myconj}
The right hand side of \eqref{eq:geodesicMaassalphafixed} has a removable singularity at $\alpha = 0$.  When $\alpha = 0$, $E(z, 1+\alpha) + E(z, 1-\alpha)$ becomes the constant term %\footnote{Do not confuse this constant term with the more common usage of the constant term in the Fourier expansion of $E(z,s)$ in terms of $z$} 
of $2E(z,s)$ in the Laurent expansion around $s=1$ 
which can then be expressed in terms of $\log ( \sqrt{y} |\eta(z)|^2 )$;
see (22.69) of \cite{IK}.  See also Remark 6.2 of \cite{GRS} for an alternative formulation of the $\alpha = 0$ case.  However, they did not provide a derivation of their argument so we briefly include one below with Conjecture \ref{conj:L2geodesicalpha0}.

% One can show unconditionally (without deep tools like QUE) that $\int_0^{\infty} |u_j(iy)|^2 \frac{dy}{y} \gg  (L(1, \mathrm{sym}^2 u_j))^{-1}$, so that a proof of Conjecture \ref{conj:L2geodesic} would show $L(1, \mathrm{sym}^2 u_j) \gg t_j^{-\varepsilon}$ (or even $\gg (\log t_j)^{-1}$), which is a famous result of Hoffstein and Lockhart \cite{HL} which has a proof relying on the automorphy of the symmetric-square lift.  A similar remark holds for holomorphic forms of large weight.

It is also natural to consider horocycle integrals.  To this end, we have
\begin{myconj}
\label{conj:horocycle}
Let $U(z)$ be either $y^{k/2} f(z)$ or $u_j(z)$, and let $\psi: \mz \backslash \mr: \rightarrow \mr$ be a smooth function.  Then
%Suppose $U(z)$ is a Hecke-Maass cusp form.  
\begin{equation}
 \int_0^{1} \psi(x) |U(z)|^2 dx \sim \int_0^{1} \psi(x) dx
\end{equation}
for fixed $y > 0$, as the weight/eigenvalue of $U$ becomes large.  
% In case $U(z) = E(z, 1/2 + iT)$ then we have ??
% \begin{equation}
%  \int_0^{1} |E(z, 1/2 + iT)|^2 dx \sim 1.
% \end{equation}
\end{myconj}
In \cite{SarnakReznikov} (see points 3 and 4) it is conjectured that the $2n$-th moment of $U(z)$ on the horocycle converges to the $2n$-th moment of the Gaussian. 
For $\psi = 1$ and $U(z) = u_j(z)$, Theorem 1.1 of \cite{GRS} gives unconditional upper and lower bounds for the horocycle integral that are within $T^{\varepsilon}$ of the conjectured asymptotic.  A minor variation of Conjecture \ref{conj:horocycle} is given by Hejhal and Rackner \cite{HejhalRackner} (6.12).  Conjecture \ref{conj:horocycle} really consists of two parts which have different flavors.  The first step is to say that $\int_0^{1} \psi(x) |U(z)|^2 dx = (\int_0^{1} \psi(x) dx)(\int_0^{1} |U(z)|^2 dx) + \mathcal{S}$ where $\mathcal{S}$ is a kind of shifted convolution sum.  Cancellation in such a sum would indicate that $\mathcal{S} = O(T^{-\delta})$.  We shall make some conjectures on these shifted convolution sums in Section \ref{section:shiftedconvolutionsums}.  If we normalize $U$ so that $\int_0^{1} |U(z)|^2 dx = 1$ (rather than with \eqref{eq:L2normalization}) then at this point one obtains a natural notion of equidistribution along the horocycle.  Next we argue that if $U$ is normalized by \eqref{eq:L2normalization} then $\int_0^{1} |U(z)|^2 dx \sim 1$; some heuristic reasoning along these lines is given in Section \ref{section:horocycle}.

A different way to study the QUE conjecture for thin sets is to restrict $|U|^2$ to (say) a small disc with fixed center but with radius that shrinks at some rate with $U$.  These questions are examined in Section \ref{section:shrinking}.  We were partially motivated to study this question based on an analogous problem for Heegner points of discriminant $D$ and level $q$ (prime) which was considered in \cite{LMY}.  A simple question asked in \cite{LMY} is, given $q$, how large does $D$ have to be to guarantee that a Heegner point of discriminant $D$ and level $q$ lies in $\omega SL_2(\mz) \backslash \mh$ for any coset $\omega SL_2(\mz) \in \Gamma_0(q) \backslash SL_2(\mz)$?  Obviously the number of Heegner points ($\approx |D|^{1/2}$) must exceed the index of $\Gamma_0(q)$ in $SL_2(\mz)$ ($\approx q$), so that $q \ll D^{1/2-\delta}$ is a necessary condition.  On the other hand, the Lindel\"{o}f hypothesis for certain Rankin-Selberg $L$-functions implies that $q \ll D^{1/2-\delta}$ is sufficient, which is then best-possible (up to $|D|^{\varepsilon}$).
 
To this end, let $\phi: SL_2(\mz) \backslash \mh \rightarrow \mr$ be smooth and compactly supported, and consider the inner product $\langle |U|^2, \phi \rangle $ where we shall keep track of the dependence on $\phi$, as it may vary with $U$ in some way.  One natural question is to understand how fast the support of $\phi$ can shrink (which increases the size of the derivatives of $\phi$), and still expect QUE to hold.  An obvious limitation to this is that a Maass form $U$ typically oscillates at frequency $T$, and is hence roughly constant at distances less than the de Broglie wavelength of size $\approx 1/T$.  See \cite{HejhalRackner}, Sections 3 and 5.1 for elaboration here.
So if $\phi$ has support on a disc of radius $T^{-1-\delta}$ and center $z_0$, then we expect that $\langle U^2, \phi \rangle \approx U(z_0)^2 \langle 1, \phi \rangle$, and clearly QUE would not hold on this scale.

There are two standard approaches to QUE, one being Watson's formula and bounds for triple product $L$-functions, and the other being Poincare series and bounds for shifted convolution sums.  Unlike in the Heegner point case described above, assuming the Lindel\"{o}f hypothesis (for triple product $L$-functions) does not  give a result valid on discs of radius $T^{-1+\delta}$ (the smallest size discs upon which one might expect QUE to hold).  See Proposition \ref{prop:QUEshrinkingWatsonapproach} for this result.
However, the Poincare series approach seems to be better-suited for this problem, and via this method, we have
\begin{myprop}
\label{prop:thinQUE}
Suppose that a family of functions $\phi$ satisfy
\begin{equation}
\label{eq:phiderivativesCartesianintroduction}
 \frac{\partial^{k+l}}{\partial x^k \partial y^l} \phi(x + iy) \ll_{k,l} A^k B^l,
\end{equation}
for some $A, B \geq 1$ (with implied constants independent of the family of $\phi$'s)
and that each $\phi$ in the family has support contained in a fixed compact set $K$.  Let $U$ be a Hecke-Maass cusp form.  Suppose the Lindel\"{o}f hypothesis holds for $L(\mathrm{sym}^2  U, s)$, and assume Conjecture \ref{conj:SCH2} holds (this is a bound for a double sum of shifted convolution sums).  Then we have
\begin{equation}
 \langle U^2, \phi \rangle = \langle 1, \phi \rangle + O(\|\phi\|_1 T^{-1/2 + \varepsilon} (A^{1/2} + B^{1/2})).
\end{equation}
\end{myprop}
Thus Proposition \ref{prop:thinQUE}, with $A = B = T^{1-\delta}$, indicates that QUE should hold on any small scale larger than $T^{-1+\delta}$.  

For the Eisenstein series case, Luo and Sarnak \cite{LuoSarnakQUE} showed that 
\begin{equation}
\label{eq:QUEEisenstein}
\langle |E(z, 1/2 + iT)|^2, \phi \rangle \sim \tfrac{3}{\pi} \log(1/4 +T^2) \langle 1, \phi \rangle 
\end{equation}
for \emph{fixed} $\phi$.  
We shall show in Section \ref{section:QUEEIsensteinsubsection} the following (unconditional)
\begin{myprop}
\label{prop:thinEisensteinQUE}
 Suppose that a family of functions $\phi$ satisfy \eqref{eq:phiderivativesCartesianintroduction} for some $A =B \leq T^{1-\delta}$, and that each $\phi$ in the family has support contained in a fixed compact set.  Then
\begin{equation}
\label{eq:thinQUEEisenstein}
 \langle |E(z, 1/2 + iT)|^2, \phi \rangle = \log(1/4 + T^2) \langle \phi, \tfrac{3}{\pi}  \rangle  + O(A^{1/2} T^{-1/6 + \varepsilon} \| \phi \|_2) + O\Big(\frac{\log{T}}{\log \log{T}} \|\phi\|_1\Big).
\end{equation}
\end{myprop}
As a corollary, we deduce that the Eisenstein series $E(z, 1/2 + iT)$ cannot be exceptionally small (nor large) on a disc of radius $\gg T^{-1/9 + \delta}$ (the calculation for this exponent $1/9$ is that for $\phi$ approximating such a disc, $\langle 1, \phi \rangle \asymp A^{-2}$, while $\|\phi\|_2 \asymp A^{-1}$, and so one requires $A^{-2} \gg A^{-1/2} T^{-1/6 + \varepsilon}$).  
We also give a more precise version of the main term here with a power saving; see \eqref{eq:thinQUEEisenstein2} below.  

Since our paper is partially conjectural, we have organized it so that the conjectural results appear only in Sections \ref{section:derivation}, \ref{section:SCSandHorocycle}, \ref{section:tripleproductapproach}, and \ref{section:PoincareQUE}.  Within those sections, it should be clear from context whether a given formula is heuristically valid, conditional on some specific unproved hypothesis (e.g., GRH), or is valid unconditionally.  

\textbf{Acknowledgements.}  I think Roman Holowinsky, Sheng-Chi Liu, Riad Masri, and Steve Zelditch for discussions on this work.

\section{Notation and standard results}
\subsection{Fourier expansions and the standard $L$-functions}
Let $U(z)$ be one of the three functions $u_j(z)$, $y^{k/2} f(z)$, $E(z, 1/2 + iT)$.  We recall the definition
% Let $\Gamma = PSL_2(\mz)$, $\Gamma_{\infty} = \{ \begin{pmatrix} 1 & b \\ 0 & 1 \end{pmatrix}: b \in \mz \}$.  Let
\begin{equation}
 E(z,s) = \sum_{\gamma \in \Gamma_{\infty} \backslash \Gamma} \text{Im}(\gamma z)^s,
\end{equation}
where $\Gamma = PSL_2(\mz)$ and $\Gamma_{\infty}$ is the stabilizer of $\infty$.
Then $\text{Res}_{s=1} E(z,s) = \frac{3}{\pi}$.
Each $U(z)$ has the Fourier expansion
\begin{equation}
\label{eq:UFourier}
 U(z) = c_0(y) + \rho(1) \sum_{n \neq 0} \frac{\lambda(n)}{\sqrt{|n|}} e(nx) V(2 \pi |n| y),
\end{equation}
with additional notation as follows.  Here $c_0(y) = c_0(y,s)$ is the constant term in the Fourier expansion which is nonzero only in the Eisenstein case for which $c_0(y,s) = y^{s} + \varphi(s) y^{1-s}$ with $\varphi(s) = \frac{\theta(1-s)}{\theta(s)}$ and $\theta(s) = \pi^{-s} \Gamma(s) \zeta(2s)$.  Note $|\varphi(1/2 + iT)| =1$.
Here $\lambda(n)$ are Hecke eigenvalues which on the Ramanujan conjecture are bounded in absolute value by the divisor function $d(n)$, and
\begin{equation}
 \label{eq:Vdef}
V(y) = \begin{cases}
	  V_T(y) := \sqrt{y} K_{iT}(y), \qquad &\text{Maass and Eisenstein cases} \\
	  V_k(y) := y^{k/2} \exp(-y), \qquad &\text{Holomorphic case},
       \end{cases}
\end{equation}
where $T$ is the spectral parameter of $u_j$. In the Maass and Eisenstein cases, $\lambda(-n) = \lambda(n)$ since $u_j$ is assumed to be even, while in the holomorphic case $\lambda(n) = 0$ for $n < 0$.  In case $U$ is the Eisenstein series, $\lambda(n) = \tau_{iT}(n) := \sum_{ab = |n|} (a/b)^{iT}$, and
with the above normalization of the constant term, we have
\begin{equation}
 \label{eq:rho1def}
\theta(1/2 + iT) \rho(1) = (2/\pi)^{1/2}, \qquad \theta(s) = \pi^{-s}\Gamma(s) \zeta(2s).
\end{equation}
Then by Stirling's formula and standard bounds for the Riemann zeta function, we deduce
\begin{equation}
\label{eq:rho1squared}
 |\rho(1)|^2 = \frac{2}{\pi} \frac{ \cosh(\pi T)}{ |\zeta(1+2iT)|^2} = T^{o(1)} \exp(\pi T).
\end{equation}
%Our experience indicates that it is pleasant to fix the normalization of $U$ only at the very end and up to that point give formulas that are valid for any normalization of $U$.

Suppose initially that $\text{Re}(s) > \half$.  Then we define
\begin{equation}
\label{eq:completedLFunction}
 \mathcal{L}(1/2 + s) = \int_0^{\infty} (U(iy) - c_0(y )) y^s \frac{dy}{y}.
\end{equation}
This integral converges absolutely since $U(iy) - c_0(y,1/2+iT) \ll_T  y^{-1/2}$ for $y \ll 1$ (it is $O(\exp(-y))$ for $y \gg 1$).  If $U$ is cuspidal then the integral converges absolutely for all $s \in \mc$.
For $\text{Re}(s) > \half$, we can reverse the order of integration and summation to show
\begin{equation}
\label{eq:completedLFunction2}
 \mathcal{L}(1/2 + s) = (1 + \lambda(-1)) \rho(1) L(1/2 + s,U) \gamma_V(1/2+ s) ,
\end{equation}
where we write
\begin{equation}
 L(s, U) = \sum_{n \geq 1} \frac{\lambda(n)}{n^s}, \qquad \gamma_V(1/2+s) = \int_0^{\infty} V(2 \pi y) y^s \frac{dy}{y}. 
\end{equation}
It is sometimes useful to express $\gamma_V$ in terms of gamma functions.  From  \cite{GR} (6.561.16), 
\begin{equation}
\label{eq:gammaVformula}
 \gamma_{V_T}(1/2 + s) = 2^{-3/2} \pi^{-s} \Gamma\Big(\frac{\thalf + s + iT}{2}\Big) \Gamma\Big(\frac{\thalf + s - iT}{2}\Big).
\end{equation}
For the holomorphic case, we have by a direct calculation 
\begin{equation}
 \gamma_{V_k}(1/2 + s) = (2 \pi )^{-s} \Gamma(\tfrac{k}{2} + s).
\end{equation}
When $U(z) = E(z, 1/2 + iT) = E_T$, %$\lambda(n) = \tau_{iT}(n)$, and so 
a short calculation shows
\begin{equation}
\label{eq:EisensteinSeriesLfunction}
 L(s, E_T) = \zeta(s + i T) \zeta(s-iT).
\end{equation}

\subsection{Rankin-Selberg integrals}
\begin{mylemma}
\label{lemma:Mellin}
 Suppose $F$ and $G$ are smooth functions on $\mr^{+}$ with rapid decay at $0$ and $\infty$.  Then for any $\alpha, \beta \in \mc$, we have
\begin{equation}
 \int_0^{\infty} F(y)y^{\alpha} G(y) y^{\beta} \frac{dy}{y} = \frac{1}{2 \pi i} \int_{(c)} \widetilde{F}(s+\alpha) \widetilde{G}(\beta -s) ds,
\end{equation}
where $\widetilde{W}(s)$ denotes the Mellin transform of a function $W$:
\begin{equation}
\widetilde{W}(s) = \int_0^{\infty} W(y) y^s \frac{dy}{y}, \qquad W(y) = \frac{1}{2 \pi i} \int_{(c)} \widetilde{W}(s) y^{-s} ds.
\end{equation}
\end{mylemma}
Lemma \ref{lemma:Mellin} is a special case of Mellin convolution (easily proved in the smooth setting by Mellin inversion) so we omit the proof.

Since we have the Mellin pair
\begin{equation}
 \gamma_V(1/2+s) = \int_0^{\infty} V(2\pi y) y^s \frac{dy}{y}, \qquad V(2\pi y) = \frac{1}{2 \pi i} \int_{(c)} \gamma_V(1/2+s) y^{-s} ds,
\end{equation}
by Lemma \ref{lemma:Mellin}, we have for $\text{Re}(s) > -1$,
\begin{equation}
\label{eq:gammaVsquaredDef}
\gamma_{V^2}(1+s) := \int_0^{\infty} V(2\pi y)^2 y^s \frac{dy}{y} = \frac{1}{2\pi i} \int_{(0)} \gamma_{V}(1/2 + s + v) \gamma_{V}(1/2-v) dv.
\end{equation}
Technically, the conditions of Lemma \ref{lemma:Mellin} are not met, but the formula still holds in this range.  It is also posible to find a closed formula for $\gamma_{V^2}$.
By (6.576.4) of \cite{GR} we have for $V_T(y) = \sqrt{y} K_{iT}(y)$ that
\begin{equation}
\label{eq:gammaVsquaredBessel}
 \gamma_{V_T^2}(1+s) = 2^{-2} \pi^{-s} \frac{\Gamma(\frac{1+s+2iT}{2}) \Gamma(\frac{1+s}{2})^2\Gamma(\frac{1+s-2iT}{2})}{\Gamma(1+s)}.
\end{equation}
If $V_k(y) = y^{k/2} \exp(-y)$, then by direct calculation
\begin{equation}
\label{eq:gammaVsquaredExponential}
 \gamma_{V_k^2}(1+s) = 2^{-k} (4\pi )^{-s} \Gamma(k+s).
\end{equation}
In many cases throughout this paper we have preferred to use the definition \eqref{eq:gammaVsquaredDef} over the evaluations \eqref{eq:gammaVsquaredBessel}--\eqref{eq:gammaVsquaredExponential}, in part because it allows us to treat the holomorphic and Maass cases simultaneously, but also because the exact evaluations are not necessary and it is easier to work directly with the definition.

The Dirichlet series analog of the above definitions is
\begin{equation}
\label{eq:Zdef}
 Z(s, U) = \sum_{n=1}^{\infty} \frac{\lambda(n)^2}{n^s},
\end{equation}
which converges absolutely for $\text{Re}(s) > 1$ by Rankin-Selberg theory (see below).
For $U = u_j$ or $y^{k/2} f(z)$, we have
\begin{equation}
\label{eq:Zformulacusp}
 Z(s, U)  = \frac{\zeta(s)}{\zeta(2s)} L(\mathrm{sym}^2 U, s),
\end{equation}
while if $U = E(z, 1/2 + iT) = E_T$, then
\begin{equation}
\label{eq:ZformulaEisenstein}
Z(s, E_T) = \frac{\zeta^2(s) \zeta(s-2iT) \zeta(s+2iT)}{\zeta(2s)}.
\end{equation}

\begin{mylemma}
\label{lemma:RankinSelbergIntegral}
 Suppose that $U(z)$ is either $u_j(z)$ or $y^{k/2} f(z)$.  Then for $\text{Re}(s) > 0$, we have
\begin{equation}
\label{eq:RankinSelbergIntegral}
 \int_{\Gamma \backslash \mh} |U(z)|^2 E(z, 1+s) \frac{dx dy}{y^2} 
=  (1 + \lambda(-1)^2) |\rho(1)|^2 Z(1+s, U) \gamma_{V^2}(1+s).
%= (1 + \lambda(-1)^2) |\rho(1)|^2 \Big(\sum_{n \geq 1} \frac{\lambda(n)^2}{n^{1+s}} \Big) \frac{1}{2 \pi i} \int_{(0)} \gamma_V(1/2+s+v) \gamma_V(1/2-v) dv.
\end{equation}
The formula extends to hold for all $s \in \mc$ by meromorphic continuation.
\end{mylemma}
This is the standard unfolding argument but since we use it repeatedly we give the full proof.

\begin{proof}[Proof of Lemma \ref{lemma:RankinSelbergIntegral}]
For $\text{Re}(s) > 1$, we have
\begin{equation}
 \langle |U|^2, E(\cdot, \overline{s}) \rangle := \int_{\Gamma \backslash \mh} |U(z)|^2 E(z,s) \frac{dx dy}{y^2} =  \sum_{\gamma \in \Gamma_{\infty} \backslash \Gamma} \int_{\Gamma \backslash \mh} |U(z)|^2 (\text{Im}(\gamma z))^s \frac{dx dy}{y^2},
\end{equation}
which by unfolding gives
\begin{equation}
\langle |U|^2, E(\cdot, \overline{s}) \rangle =\int_0^{1} \int_0^{\infty} |U(z)|^2 y^s \frac{dx dy}{y^2}.
\end{equation}
Inserting the Fourier expansion of $U$ and performing the $x$-integral, we obtain
\begin{multline}
\label{eq:EisensteinUnfoldingCalculation}
\langle |U|^2, E(\cdot, \overline{s}) \rangle=  |\rho(1)|^2 \sum_{n \neq 0} \frac{\lambda(n)^2}{|n|} \int_0^{\infty} |V(2 \pi |n| y)|^2 y^{s-1} \frac{dy}{y} 
\\
= (1 + \lambda(-1)^2) |\rho(1)|^2 \Big(\sum_{n \geq 1} \frac{\lambda(n)^2}{n^{s}}\Big) \int_0^{\infty} |V(2\pi y)|^2 y^{s-1} \frac{dy}{y},
\end{multline}
which by \eqref{eq:gammaVsquaredDef} and \eqref{eq:Zdef}, finishes the proof.
\end{proof}
When $U$ is the Eisenstein series then \eqref{eq:RankinSelbergIntegral} does not converge.  Instead, we consider the inner product against the incomplete Eisenstein series as do Luo-Sarnak in \cite{LuoSarnakQUE}, but we briefly re-derive the formula to have a self-contained exposition.  Suppose that $h: \mathbb{R}^+ \rightarrow \mathbb{R}$ is smooth and compactly-supported.  By definition,
\begin{equation}
 E(z, h) = \sum_{\gamma \in \Gamma_{\infty} \backslash \Gamma} h(\text{Im}(\gamma z)),
\end{equation}
and Mellin inversion gives for $c > 1$,
\begin{equation}
\label{eq:incompleteEisensteincontourintegralintermsofEzs}
 E(z,h) = \frac{1}{2\pi i} \int_{(c)} E(z,s) \widetilde{h}(-s) ds.
\end{equation}
%Our formulas are most pleasant for the function $H(y) = y h(y)$.
The formula (valid for $U$ Eisenstein series or a cusp form) states
\begin{multline}
\label{eq:UsquaredinnerproductincompleteEisenstein}
 \langle |U|^2, E(z,h) \rangle = \int_0^{\infty} |c_0(y)|^2 h(y) \frac{dy}{y^2} + 
\\
(1 + \lambda(-1)^2) \frac{|\rho(1)|^2}{2\pi i} \int_{(\varepsilon)} \widetilde{h}(-1-s) Z(1+s, U) \gamma_{V^2}(1+s) ds.
\end{multline}
\begin{proof}
Unfolding as in the proof of Lemma \ref{lemma:RankinSelbergIntegral}, we obtain
\begin{equation}
 \langle |U|^2, E(z,h) \rangle = \int_0^{\infty} \int_0^{1} |U(z)|^2 h(y) \frac{dx dy}{y^2}.
\end{equation}
Inserting the Fourier expansion of $U$, we obtain
\begin{multline}
\label{eq:unfoldingincompleteEisenstein}
 \langle |U|^2, E(z,h) \rangle = \int_0^{\infty} |c_0(y)|^2 h(y) \frac{dy}{y^2} 
\\
+ (1 + \lambda(-1)^2) |\rho(1)|^2 \sum_{n \geq 1} \frac{\lambda(n)^2}{n} \int_0^{\infty} |V(2\pi n y)|^2 h(y) \frac{dy}{y^2}.
\end{multline}
Using Mellin inversion on $h$, we obtain that the non-constant terms in \eqref{eq:unfoldingincompleteEisenstein} equal
\begin{equation}
\label{eq:nonconstanttermsusingMellininversion}
 (1 + \lambda(-1)^2 )|\rho(1)|^2 \frac{1}{2\pi i} \int_{(\varepsilon)} \widetilde{h}(-1-s)  \Big(\sum_{ n \geq 1} \frac{\lambda(n)^2}{n^{1+s}} \Big) \int_0^{\infty} |V(2\pi y)|^2  y^{s} \frac{dy}{y} ds.
\end{equation}
Again using \eqref{eq:gammaVsquaredDef} and \eqref{eq:Zdef}, we obtain \eqref{eq:UsquaredinnerproductincompleteEisenstein}.
\end{proof}

\section{Derivation of Conjectures \ref{conj:geodesicQUE} and \ref{conj:L2geodesic}}
\label{section:derivation}
In this section, we suppose $U$ is either $u_j(z)$ or $y^{k/2} f(z)$.  By Lemma \ref{lemma:Mellin} and \eqref{eq:completedLFunction}, 
\begin{equation}
I(U,\alpha):= \int_0^{\infty} |U(iy)|^2 y^{\alpha} \frac{dy}{y} = \frac{1}{2\pi i} \int_{(0)} \mathcal{L}(1/2 + s + \alpha) \mathcal{L}(1/2-s) ds.
\end{equation}
By \eqref{eq:completedLFunction2}, we have
\begin{equation}
\label{eq:secondmomentformulaforIU}
 I(U,\alpha) =  (1+ \lambda(-1))^2  \frac{|\rho(1)|^2}{2 \pi i} \int_{(c)} L( \thalf + \alpha + s, U) L(\thalf  - s, U) \gamma_V(\thalf + \alpha + s) \gamma_V(\thalf - s) ds.
\end{equation}
This is effectively a shifted second moment of the $L$-function with a weight depending on $U$.  The weight is very natural, being the Archimedean part of the completed $L$-function.
Although such moments for $U$ fixed have been well-studied, the case of $U$ varying is more difficult, as mentioned in the introduction.  Therefore we appeal to the recipe for computing the main term of a moment of $L$-functions due to \cite{CFKRS}.  The following formula should be understood in a formal sense:
\begin{multline}
I(U,\alpha) \sim 
(1+ \lambda(-1))^2 |\rho(1)|^2   
\\
\Big(\sum_{m,n \geq 1} \frac{\lambda(m) \lambda(n)}{m^{1/2+\alpha} n^{1/2}} \frac{1}{2 \pi}  \intR \Big(\frac{n}{m}\Big)^{it} \gamma_V(1/2 + \alpha + it) \gamma_V(1/2  -it) dt + \dots\Big),
\end{multline}
where the dots indicate an identical term with $\alpha$ replaced by $-\alpha$, consistent with $I(U,\alpha) = I(U, -\alpha)$.  This is the first step of the \cite{CFKRS} recipe which calls for us to write a formal approximate functional equation for each $L$-function, and to discard terms where the ``root number'' is oscillating with the family.
We view the $t$-integral as selecting only the diagonal term $m=n$, leading to
\begin{align}
\nonumber
I(U,\alpha) \sim (1+ \lambda(-1))^2 |\rho(1)|^2 \Big(\sum_{n \geq 1} \frac{\lambda(n)^2}{n^{1+\alpha}} \frac{1}{2\pi i} \int_{(0)} \gamma_V(1/2 + \alpha + s) \gamma_V(1/2 -s) ds + \dots \Big)
\\
\label{eq:I2alphadiagonalterm}
 = (1+ \lambda(-1))^2 |\rho(1)|^2 \big( Z(1+\alpha, U) \gamma_{V^2}(1+\alpha) + Z(1-\alpha, U) \gamma_{V^2}(1-\alpha) \big).
\end{align}
Conveniently there was no need to invoke the evaluations \eqref{eq:gammaVsquaredBessel}--\eqref{eq:gammaVsquaredExponential}.
Using \eqref{eq:RankinSelbergIntegral}, we obtain the conjectured asymptotic
% That is, we have
% \begin{equation}
% \label{eq:I2withKBessel}
%  I_2(\alpha+\beta) \sim 4 |\rho(1)|^2 (\sum_{n \geq 1} \frac{|\lambda(n)|^2}{n^{1+\alpha+\beta}} \int_0^{\infty} |K_{iT}(2\pi y)|^2 y^{1 + \alpha+\beta} \frac{dy}{y} + \dots),
% \end{equation}
\begin{equation}
\label{eq:I2withEisenstein}
 I(U,\alpha) \sim (1 + \lambda(-1)) \int_{\Gamma \backslash \mathbb{H}} |U(z)|^2 (E(z, 1 + \alpha) +  E(z, 1 - \alpha) )\frac{dx dy}{y^2}.
\end{equation}
This is a compact way to express the two formulas in Conjecture \ref{conj:L2geodesic}, since $\lambda(-1) = 1$ for $U= u_j$ an even Maass form, and $\lambda(-1) = 0$ for $U = y^{k/2} f(z)$ coming from a holomorphic form.

Next we derive Conjecture \ref{conj:geodesicQUE} from Conjecture \ref{conj:L2geodesic}, supposing that the implied error term in Conjecture \ref{conj:L2geodesic} has some uniformity in $\alpha$.
Let $\psi(y)$ be a smooth function with compact support on $\mathbb{R}^+$, and 
define
\begin{equation}
 I(U, \psi) = \int_0^{\infty} |U(iy)|^2 \psi(y) \frac{dy}{y}.
\end{equation}
We may suppose that $\psi$ is even, i.e., $\psi(y^{-1}) = \psi(y)$, because in general one can decompose $\psi$ as the sum of an even and an odd function, and $I(U, \psi)$ vanishes if $\psi$ is odd.
Thus $\widetilde{\psi}(s)$ is entire and satisfies $\widetilde{\psi}(s) = \widetilde{\psi}(-s)$.

By Mellin inversion, and using \eqref{eq:I2withEisenstein}, we have
\begin{multline}
 I(U, \psi) = \frac{1}{2\pi i} \int_{(0)} \widetilde{\psi}(-s) I(U,s) ds 
\\
\sim (1 + \lambda(-1)) \frac{1}{2\pi i} \int_{(0)} \widetilde{\psi}(-s)  \int_{\Gamma \backslash \mathbb{H}} |U(z)|^2 (E(z, 1 + s) +  E(z, 1 - s) )\frac{dx dy}{y^2} ds.
\end{multline}
Here is where we assumed the error term in \eqref{eq:I2withEisenstein} holds with some uniformity.

Reversing the order of integration, and with the definition
\begin{equation}
\label{eq:Hdef}
H(z, \psi) = \frac{1}{2\pi i} \int_{(0)} \widetilde{\psi}(-s) (E(z, 1 + s) +  E(z, 1 - s) ) ds,
\end{equation}
we have the conjecture
\begin{equation}
\label{eq:IUpsiwithH}
 I(U, \psi) \sim (1 + \lambda(-1)) \int_{\Gamma \backslash \mathbb{H}} |U(z)|^2 
H(z, \psi)
\frac{dx dy}{y^2}.
\end{equation}
Finally we show unconditionally, using the QUE theorem, that the right hand side of \eqref{eq:IUpsiwithH} agrees with the statement of Conjecture \ref{conj:geodesicQUE}. 

Observe that $H(z, \psi)$ is $SL_2(\mz)$-invariant.  In fact $H$ is related to an incomplete Eisenstein series, namely
\begin{equation}
\label{eq:Hformula}
 H(z, \psi) = -\frac{3}{\pi} \widetilde{\psi}(0) + 2 E(z, y \psi(y)),
\end{equation}
% where for a smooth compactly-supported function $\phi(y)$, we define the incomplete Eisenstein series
% \begin{equation}
%  E(z, \phi) = \sum_{\gamma \in \Gamma_{\infty} \backslash \Gamma} \phi(\text{Im}(\gamma z)).
% \end{equation}
We prove this now.
First, shift the contour in \eqref{eq:Hdef} to $\text{Re}(s) = \varepsilon$ and expand the integral as the sum of two integrals corresponding to the two terms $E(z,1+s)$ and $E(z,1-s)$.  Then by shifting the latter of these contours to the left, we have
\begin{equation}
 \frac{1}{2\pi i} \int_{(\varepsilon)} \widetilde{\psi}(-s)   E(z, 1 - s) ds = -\frac{3}{\pi} \widetilde{\psi}(0)  + \frac{1}{2\pi i} \int_{(-\varepsilon)} \widetilde{\psi}(-s)   E(z, 1 - s) ds,
\end{equation}
which by the change of variables $s \rightarrow -s$ and the evenness of $\widetilde{\psi}$ gives 
\begin{equation}
 H(z, \psi) =-\frac{3}{\pi} \widetilde{\psi}(0) +  \frac{2}{2 \pi i} \int_{(\varepsilon)} \widetilde{\psi}(-s) E(z,1+s) ds.
\end{equation}
Finally we observe from \eqref{eq:incompleteEisensteincontourintegralintermsofEzs} that
\begin{equation}
 \frac{1}{2 \pi i} \int_{(\varepsilon)} \widetilde{\psi}(-s) E(z,1+s) ds = E(z, y \psi(y)). 
\end{equation}

  %If we drop the condition \eqref{eq:L2normalization}, then the QUE conjecture \eqref{eq:QUE} can be expressed as
% \begin{equation}
% \label{eq:QUE2}
%  \int_{\Gamma \backslash \mathbb{H}} |U(z)|^2 \phi(z) \frac{3}{\pi} \frac{dx dy}{y^2} \rightarrow \Big(\int_{\Gamma \backslash \mathbb{H}} |U(z)|^2  \frac{3}{\pi} \frac{dx dy}{y^2} \Big) \Big( \int_{\Gamma \backslash \mathbb{H}} \phi(z)  \frac{3}{\pi} \frac{dx dy}{y^2} \Big).
% \end{equation}
Inserting \eqref{eq:Hformula} into \eqref{eq:IUpsiwithH}, applying \eqref{eq:QUE} to $I(U, \psi)$ (with $\psi$ fixed, the eigenvalue/weight of $U$ large), and recalling that $U$ is normalized by \eqref{eq:L2normalization},
we obtain
\begin{multline}
 I(U, \psi) \sim (1 + \lambda(-1)) \int_{\Gamma \backslash \mathbb{H}} \Big(-\frac{3}{\pi} \widetilde{\psi}(0) +  2E(z, y \psi(y)) \Big)  \frac{dx dy}{y^2} 
\\
= (1 + \lambda(-1)) \Big(-  \widetilde{\psi}(0) + 2 \int_0^1 \int_0^{\infty} y \psi(y)  \frac{dx dy}{y^2}\Big).
\end{multline}
This simplifies as
\begin{equation}
I(U, \psi) \sim (1 + \lambda(-1))  \int_0^{\infty} \psi(y) \frac{dy}{y},
\end{equation}
which again is a compact way to write both asymptotics in Conjecture \ref{conj:geodesicQUE}.

The case $\alpha = 0$ is particularly interesting, and Conjecture \ref{conj:L2geodesic} is perhaps not the best form in this case.
\begin{myconj}
\label{conj:L2geodesicalpha0} 
We have
\begin{equation}
 \int_0^{\infty} |u_j(iy)|^2 \frac{dy}{y} \sim 2 \log(1/4 + t_j^2),
\end{equation}
as $t_j \rightarrow \infty$.
\end{myconj}
One way to guess the form of this answer is to imagine that $|u_j(iy)|^2 \sim 2$ (on average) until it begins to decay quickly, which is for $y + y^{-1} \gg t_j$ (for this, see \eqref{eq:BesselKUB} and surrounding discussion below).

 We set $t_j = T$ and continue with \eqref{eq:I2alphadiagonalterm}.
% , which takes the form
%  \begin{equation}
% I(U,\alpha) \sim  (1 + \lambda(-1))^2 |\rho(1)|^2 (Z(1+\alpha) \gamma_{V_{T}^2}(1+\alpha) + Z(1-\alpha) \gamma_{V_{T}^2}(1-\alpha) ),
%  \end{equation}
%  where recall \eqref{eq:gammaVsquaredDef} and \eqref{eq:Zdef}.  
With this notation, Lemma \ref{lemma:RankinSelbergIntegral} shows
\begin{equation}
 \text{Res}_{\alpha = 0} (1 + \lambda(-1)) |\rho(1)|^2 Z(1+\alpha) \gamma_{V_T^2}(1 + \alpha) = 1.
\end{equation}
Let $Z(1+\alpha) = \frac{r_{-1}}{\alpha}(1 + r_0 \alpha + \dots)$.  By Theorem 5.17 of \cite{IK} (conditional on GRH and Ramanujan), we have $r_0 \ll \log \log T$.   By taking a Taylor expansion, we have the prediction
\begin{equation}
\int_0^{\infty} |u_j(iy)|^2 y^{\alpha} \frac{dy}{y} \sim (1 + \lambda(-1))(2 \frac{\gamma_{V_{T}^2}'}{\gamma_{V_T^2}}(1) + 2 r_0 + O(\alpha)).
\end{equation}
Using the explicit evaluation \eqref{eq:gammaVsquaredBessel} and Stirling's formula, we have
\begin{equation}
 \frac{\gamma_{V_{T}^2}'}{\gamma_{V_T^2}}(1) = \half \log(1/4 + T^2) + O(1).
\end{equation}
This quickly leads to Conjecture \ref{conj:L2geodesicalpha0}.

\section{Shifted convolution sums and horocycle integrals}
\label{section:SCSandHorocycle}
The shifted convolution problem is a well-known problem in analytic number theory having direct applications to subconvexity.  See Michel's Park City lecture notes \cite{MichelParkCity} for a good general introduction.  

\subsection{Conjectures}
\label{section:shiftedconvolutionsums}
We have found that for applications to QUE of restricted eigenfunctions, the following formulation of the shifted convolution problem is natural.
\begin{myconj}
\label{conj:SCH}
 Let $U(z)$ be a Hecke-Maass cusp form with Laplace eigenvalue $1/4 + T^2$.  Then if $0 < a \leq y \leq b$, and $m \neq 0$, we have
\begin{equation}
\label{eq:SCH}
|\rho(1)|^2 \sum_{n \in \mz} \lambda(n) \lambda(n + m) K_{iT}(2\pi |n|y) K_{iT}(2 \pi |m+n| y) \ll_{a,b,\tau, R} T^{-\tau+\varepsilon} |m|^R,
\end{equation}
for some fixed $\tau>0$ and some fixed $R \geq 0$.
\end{myconj}
An equivalent way to formulate \eqref{eq:SCH} is
\begin{equation}
\label{eq:SCHhorocyleform}
 \int_0^{1} |U(x+iy)|^2 e(-mx) dx \ll T^{-\tau+\varepsilon} |m|^{R},
\end{equation}
which we interpret as a bound on the $m$-th Fourier coefficient of $|U|^2$ along the horocycle at height $y$.

It may even be true that  $\tau = 1/2$ and $R=0$ are valid which amounts to ``square-root'' cancellation in $n$ and strong uniformity in $m$.  The estimate is true with $\tau = 0$ and $R = 0$ by Theorem 1.1 of \cite{GRS}.  Their proof proceeds by
foregoing any cancellation in the sum; this requires a close analysis of the uniform asymptotic behavior of the Bessel function for which see Lemma 3.1 of \cite{GRS}.  The basic point is that $\cosh(\pi T) K_{iT}(y)$ is exponentially small for $y > T + CT^{1/3}$, and 
\begin{equation}
\label{eq:BesselKUB}
 \cosh(\pi T) K_{iT}(y) \ll
\begin{cases}
 T^{-1/4} |T-y|^{-1/4}) \qquad &\text{for } y < T - CT^{1/3} \\
T^{-1/3} \qquad &\text{for } |y - T| \leq CT^{1/3}.
\end{cases}
\end{equation}
%In particular, for $y \leq T/2$, we have $\cosh(\pi T) K_{iT}(y) \ll T^{-1/2}$. 

In Conjecture \ref{conj:SCH} we did not impose a restriction on the size of $m$ because the Bessel function (times $\rho(1)^2$) is exponentially small for $|n| \gg T^{1+\varepsilon}$ or $|n+m| \gg T^{1+\varepsilon}$, so if $|m| \geq T^{1+\varepsilon}$ then Conjecture \ref{conj:SCH} is trivial.  This formulation of the shifted convolution problem is non-standard in the sense that the weight function is oscillatory and also that the automorphic form is varying.
Since the above shifted convolution hypothesis, if true, is very far from current technology, it seems natural to explore some of its consequences.

We will use
Conjecture \ref{conj:SCH} to aid in studying various QUE integrals.

Conjecture \ref{conj:SCH} implies a nontrivial bound on the sup-norm of the Maass form $U(z)$ restricted to a fixed compact subset of $\mathbb{H}$, which requires some explanation.  
To this end, we present a neat variant of van der Corput's method as follows.
\begin{myprop}
\label{prop:vdCvariant}
Suppose that $G(x)$ is a $\mz$-periodic function having a Fourier series supported on an interval containing $J$ integers.  Then for any $t \in \mr$, 
\begin{equation}
\label{eq:vdCvariant}
|G(t)|^2 \leq J \int_0^{1} |G(x)|^2 dx.
\end{equation}
\end{myprop}
\begin{proof}
For $A, B \in \mz$ and $B \geq 1$, let $D_{A, B}(x) := \sum_{h= A + 1}^{A + B} e(hx) = e(Ax) D_{0,B}(x)$, which is a shifted Dirichlet kernel.  Supposing $G(t) = \sum_{m \in \mz} b_m e(mt)$, 
it is an easy calculation to show $(G* D_{A,B})(t) = \sum_{m=A+1}^{A+B} b_m e(mt)$.  With $B=J$ and some choice of $A$, we have $G*D_{A,J} = G$.
Thus by Cauchy's inequality,
\begin{equation}
|G(t)|^2 = \Big|\int_0^{1} G(y) D_{A,J}(t-y) dy \Big|^2 \leq \Big(\int_0^{1} |G(x)|^2 dx\Big)  \Big(\int_0^{1} |D_{A,J}(x)|^2 dx\Big).
\end{equation}
It is easy to check (by Parseval) that $\int_0^{1} |D_{A,J}(x)|^2 dx = J$.
\end{proof}
To see the relation between Proposition \ref{prop:vdCvariant} and the usual van der Corput bound, suppose that $a_n$ is a sequence vanishing outside an interval containing $J$ integers, and let $F(x) = \sum_{n \in \mz} a_n e(nx)$.  For a positive integer $H$, let $D_H(x) = \sum_{h=1}^{H} e(hx)$, and set $G(x) = F(x) D_H( t-x)$.  Note that $G(x)$ has a Fourier series supported on an interval containing at most $J+H$ integers, $D_H(0) = H$, and so that by \eqref{eq:vdCvariant}, we have
\begin{equation}
H^2 |F(t)|^2 \leq (J+H) \int_0^{1} |F(x)|^2 |D_H(t-x)|^2 dx.
\end{equation}
One can see this as a form of amplification.
Writing out the definition of $D_H$, opening the square, and reversing the order of integration, we obtain
\begin{equation}
\label{eq:vdCvariantHorocyle}
 H^2 |F(t)|^2 \leq (J+H)  \sum_{k=1}^{H} \sum_{l=1}^{H} e(t(-k+l)) \int_0^{1} |F(x)|^2 e((k-l)x) dx.
\end{equation}
Setting $k-l = h$, and dividing through by $H^2$, we obtain
\begin{equation}
\label{eq:vdCvariantHorocycle2}
  |F(t)|^2 \leq \frac{(J+H)}{H}  \sum_{|h| < H} \Big(1 - \frac{|h|}{H}\Big)  e(-ht) \int_0^{1} |F(x)|^2 e(hx) dx.
\end{equation}
Taking $t=0$, and inserting the definition of $F$, we rewrite this as
\begin{equation}
 \Big| \sum_{n \in \mz} a_n \Big|^2 \leq \frac{(J+H)}{H}  \sum_{|h| < H} \Big(1 - \frac{|h|}{H}\Big) \sum_{n \in \mz} a_n \overline{a_{n+h}},
\end{equation}
which is the usual formulation of the van der Corput inequality.

If Conjecture \ref{conj:SCH} is true, then we claim
\begin{equation}
\label{eq:supnormofU}
|U(z)| \ll T^{\frac12 - \frac{\tau}{2(R+1)} + \varepsilon}.
\end{equation}
For reference, the trivial sup-norm bound is $|U(z)| \ll T^{1/2 + \varepsilon}$.
To derive \eqref{eq:supnormofU} from Conjecture \ref{conj:SCH}, we apply \eqref{eq:vdCvariantHorocycle2} to $F(x) = U(x+iy)$ (with $y$ in some fixed compact interval).  Actually we first truncate the Fourier expansion for $U(z)$ at $|n| \leq T^{1+\varepsilon}$ using the fact that the tail is very small coming from the exponential decay of the Bessel function in this range.  In this way, we obtain for $1 \leq H \leq T^{1+\varepsilon}$
\begin{equation}
 |U(x+iy)|^2 \ll \frac{T^{1+\varepsilon}}{H} \Big(\int_0^{1} |U(x+iy)|^2 dx + \sum_{0 < |h| < H} \Big|\int_0^{1} |U(x+iy)|^2 e(hx) dx \Big|\Big) + T^{-100},
\end{equation}
to which we appeal to \eqref{eq:SCHhorocyleform}, leading to
\begin{equation}
 |U(z)|^2 \ll \frac{T^{1+\varepsilon}}{H} + T^{1-\tau+\varepsilon} H^R.
\end{equation}
Thus choosing $H \asymp T^{\frac{ \tau}{R+1}}$, we obtain \eqref{eq:supnormofU}.  

Even with the optimal bound $\tau = 1/2$, $R=0$, \eqref{eq:supnormofU} leads to $|U(z)| \ll_{\varepsilon} T^{1/4 + \varepsilon}$ which is far from the conjectured bound $|U(z)| \ll_{} T^{\varepsilon}$ \cite{IwaniecSarnak} \cite{SarnakMorowetz} (for $z$ in a fixed compact set).  To get to this optimal bound, we would need to assume some additional cancellation in the shift $m$ appearing in \eqref{eq:SCH}.  To this end, we suppose a more flexible and stronger shifted convolution hypothesis:
\begin{myconj}
\label{conj:SCH2}
 Let $U(z)$ be a Hecke-Maass cusp form with Laplace eigenvalue $1/4 + T^2$.  Then if $0 < a \leq y \leq b$, $\alpha \in \mr$, and $1 \leq M \leq T^{1+\varepsilon}$, we have
\begin{equation}
\label{eq:SCH2}
|\rho(1)|^2 \sum_{1 \leq |m| \leq M} e(m\alpha) \sum_{n \in \mz} \lambda(n) \lambda(n + m) K_{iT}(2\pi |n|y) K_{iT}(2 \pi |m+n| y) \ll_{a,b, \varepsilon} T^{-1/2+\varepsilon} M^{1/2}.
\end{equation}
Equivalently, uniformly in  $\alpha \in \mr$,
\begin{equation}
\label{eq:SCH2horocycleform}
\sum_{1 \leq |m| \leq M}  \int_0^{1} |U(x+iy)|^2 e(m(\alpha-x)) dx \ll_{a,b, \varepsilon} T^{-1/2+\varepsilon} M^{1/2}.
\end{equation}
\end{myconj}
If the sup-norm bound $U(z) \ll T^{\varepsilon}$ holds (for a family of $U$'s restricted to $z$ in a fixed compact set), then \eqref{eq:SCH2horocycleform} holds when $M = T^{1+o(1)}$, since if $D_M(x) = \sum_{m=1}^{M} e(mx)$ is the Dirichlet kernel, then the left hand side of \eqref{eq:SCH2horocycleform} is $\ll T^{\varepsilon} \int_0^{1} |D_M(x)| dx \ll T^{\varepsilon}$.  This gives some evidence towards Conjecture \ref{conj:SCH2}.

In the opposite direction, Conjecture \ref{conj:SCH2} implies $U(z) \ll T^{\varepsilon}$ (uniformly on any fixed compact set).  This follows by applying Conjecture \ref{conj:SCH2} to \eqref{eq:vdCvariantHorocycle2} after partial summation, getting $|U(x+iy)|^2 \ll H^{-1} T^{1+\varepsilon} + H^{-1/2} T^{1/2 + \varepsilon}$, and the optimal choice is $H = T$.

So far we have concentrated on the Hecke-Maass case, but it is equally interesting to consider holomorphic cusp forms.  In this case, we have
by \eqref{eq:EisensteinUnfoldingCalculation} and \eqref{eq:gammaVsquaredExponential} that 
\begin{equation}
\label{eq:rhosquaredholomorphic}
 1 = \frac{|\rho(1)|^2}{2^k} \Gamma(k) \mathop{Res}_{s=1} Z(s, U), \quad \text{hence} \quad |\rho(1)|^2 = \frac{2^k k^{o(1)}}{\Gamma(k)}.
\end{equation}
Then a calculation with Stirling's formula shows
\begin{equation}
\label{eq:Vklocalization}
 \rho(1) V_k(y) = k^{1/4 + o(1)} \exp\Big(- \frac{(y-\frac{k}{2})^2}{k}\Big).
\end{equation}
Thus $\rho(1) V_k(2 \pi ny)$ is localized in $|4\pi n y -k| \leq C(k \log {k})^{1/2}$.  We make the following
\begin{myconj}
\label{conj:SCHholo}
 Let $U(z) = y^{k/2} f(z)$ with $f$ a holomorphic Hecke cusp form of weight $k$.  Then if $0 < a \leq y \leq b$, and $m \neq 0$, we have
\begin{equation}
\label{eq:SCHholo}
\rho(1)^2 \sum_{n \geq 1} \frac{\lambda(n) \lambda(n + m)}{\sqrt{n(n+m)}} V_k(2\pi ny) V_k(2 \pi (m+n) y) \ll_{a,b, \varepsilon} k^{-1/4+\varepsilon},
\end{equation}
equivalently,
\begin{equation}
 \int_0^{1} |U(x+iy)|^2 e(mx) dx \ll_{a,b,\varepsilon} k^{-1/4 + \varepsilon}.
\end{equation}
More generally, we conjecture that for $M \leq k^{1/2 + \varepsilon}$,
\begin{equation}
 \sum_{1 \leq |m| \leq M} e(-m\alpha) \int_0^{1} |U(x+iy)|^2 e(mx) dx \ll_{a,b,\varepsilon} k^{-1/4 + \varepsilon} M^{1/2}.
\end{equation}

\end{myconj}
There are effectively $O(k^{1/2 + o(1)})$ terms in the sum in \eqref{eq:SCHholo}, so this is predicting square-root cancellation in the sum over $n$.  One easily checks using \eqref{eq:Vklocalization} that \eqref{eq:SCHholo} is very small unless $m = O(k^{1/2+\varepsilon}/y)$.

\subsection{Horocycle integrals}
\label{section:horocycle}
In this section we discuss Conjecture \ref{conj:horocycle}, aided by the conjectures in Section \ref{section:shiftedconvolutionsums}.
Define, with $U$ either $u_j$ or $y^{k/2} f(z)$,
\begin{equation}
 I_H(U, \psi) = \int_0^{1} \psi(x) |U(z)|^2 dx.
\end{equation}
Inserting the Fourier expansion \eqref{eq:UFourier}, we obtain
\begin{equation}
\label{eq:IHUpsiformula}
 I_H(U, \psi) = |\rho(1)|^2 \sum_{m,n \neq 0} \frac{\lambda(m) \lambda(n)}{\sqrt{|mn|}} V(2\pi |m| y) V(2\pi |n| y) \widehat{\psi}(m-n),
\end{equation}
where $\widehat{\psi}(k) = \int_0^{1} \psi(x) e(-kx) dx$.  Thus
\begin{equation}
\label{eq:IHUpsi}
 I_H(U, \psi) = I_H(U, 1) \int_0^{1} \psi(x) dx + \mathcal{S},
\end{equation}
where Conjecture \ref{conj:SCH} or Conjecture \ref{conj:SCHholo} (depending on if $U$ comes from a Maass form or a holomorphic form) implies $\mathcal{S} = o(1)$ .

Next we examine $I_H(U, 1)$.  By \eqref{eq:IHUpsiformula} specialized to $\psi = 1$,
\begin{equation}
 I_H(U,\psi) = (1 + \lambda(-1)^2) |\rho(1)|^2 \sum_{n=1}^{\infty} \frac{\lambda(n)^2}{n} V(2\pi n y)^2.
\end{equation}
The inverse Mellin version of \eqref{eq:gammaVsquaredDef} gives $V(2\pi y)^2 = \frac{1}{2 \pi i} \int_{(\sigma)} \gamma_{V^2}(1+s) y^{-s} ds$, so
\begin{equation}
\label{eq:horocycleIntegralLfunction}
I_H(U, 1) =  (1 + \lambda(-1)^2) |\rho(1)|^2  \frac{1}{2 \pi i} \int_{(1)} y^{-s} Z(1+s, U) \gamma_{V^2}(1+s) ds.
\end{equation}
Next the basic idea is to move the contour of integration to $\text{Re}(s) = -\half$ (this is the optimal location because a completed $L$-function is smallest on the critical line, at least if one assumes the Lindel\"{o}f hypothesis).  By Lemma \ref{lemma:RankinSelbergIntegral}, 
the residue at $s=0$ gives to $I_H(U,1)$
\begin{equation}
 \int_{\Gamma \backslash \mathbb{H}} |U(z)|^2 \text{Res}_{s=0} E(z, 1+s) \frac{dx dy}{y^2}.
\end{equation}
This equals $1$, since the residue of the Eisenstein series is $3/\pi$, and we normalize according to \eqref{eq:L2normalization}.  

Next we explain that the Lindel\"{o}f Hypothesis (on average) applied to the right hand side of \eqref{eq:horocycleIntegralLfunction} but at $\text{Re}(s) = -\half$ would show show that $I_H(U, 1) = 1 + O(T^{\varepsilon})$ (for $U$ a Maass form--in the holomorphic case we replace $T^{\varepsilon}$ by $k^{\varepsilon}$).  Of course, this error term is larger than the main term, but since the integrand is oscillatory (the phase of the $L$-function should not correlate with any simple function such as a ratio of gamma functions), it seems reasonable to suppose that there is some cancellation in the integral showing that $I_H(U, 1) =1 + o(1)$.
To this end, we consider the two cases of $U = u_j$ and $U = y^{k/2} f(z)$ separately.  Recall the evaluation \eqref{eq:gammaVsquaredBessel}.
Stirling's approximation applied to \eqref{eq:gammaVsquaredBessel} then shows in the Maass case that
\begin{equation}
 e^{\pi T} \gamma_{V_T^2}(1/2 + it) \ll \exp(\tfrac{\pi}{4} Q(t,T)) (1 + |t+2T|)^{-\frac14} (1 + |t-2T|)^{-\frac14}  (1 + |t|)^{-\frac12},
\end{equation}
where
\begin{equation}
\label{eq:Qdef}
 Q(t,T) = 4T - |t+2T| - |t-2T|.
\end{equation}
By a simple calculation, $Q(t,T) = 0$ for $|t| \leq 2T$, and $Q(t,T) = 2(2T-|t|)$ for $|t| > 2T$, so the contribution to $I_H(U, 1)$ from $|t| \geq 2T + T^{\varepsilon}$ is very small.  A short calculation shows that the Lindel\"{o}f Hypothesis is sufficient to prove that the new integral contributes $O(T^{\varepsilon})$ to $I_H(U, 1)$.  

For the holomorphic case, by \eqref{eq:rhosquaredholomorphic} and \eqref{eq:gammaVsquaredExponential} we have
\begin{equation}
 |\rho(1)|^2 \gamma_{V_k^2}(1/2 + it) = k^{o(1)} \frac{\gamma_{V_k^2}(1/2 + it)}{\gamma_{V_k^2}(1)} \ll k^{\varepsilon} \frac{\Gamma(k-\thalf + it)}{\Gamma(k)}.
\end{equation}
By Stirling's formula, this is exponentially small for $t \gg k^{1/2 + \varepsilon}$, and for $t \ll k^{1/2 + \varepsilon}$ we have
\begin{equation}
 |\rho(1)|^2 \gamma_{V_k^2}(1/2 + it) \ll k^{-1/2 + \varepsilon} \exp(-t^2/k).
\end{equation}
Just like in the Maass case, we see that the Lindel\"{o}f Hypothesis gives a bound of $k^{\varepsilon}$ for the new integral, and since the integral is presumably oscillatory it is reasonable to suppose there is cancellation.  This concludes the derivation of Conjecture \ref{conj:horocycle}.

\section{The QUE conjecture for shrinking sets}
\label{section:shrinking}
This section concerns a discussion around Proposition \ref{prop:thinQUE}.  Recall that our notation is such that QUE is an asymptotic for $\langle U^2, \phi \rangle$ where $\phi$ is a family of functions that may vary with $U$ (but the $\phi$'s are generally not as oscillatory as $U$ is).  In this section we take $U = u_j$ or $U = E_T$, but it would be interesting to study the holomorphic case too.

It is well-known that the QUE conjecture can be approached either via bounds for triple product $L$-functions via Watson's formula \cite{Watson}, or alternatively by shifted convolution sums via Poincare series.  These two methods are apparently not equivalent and indeed Holowinsky and Soundararajan \cite{HolowinskySound} exploit both approaches in their proof of the mass equidistribution conjecture.
We shall present both of these two standard approaches below.

First we impose some conditions on $\phi$ (or more accurately, the \emph{sequence} of $\phi$'s depending on $U$).  One natural choice is to pick a sequence of numbers $C(k)$ and a constant $A$, and consider $\phi$ satisfying
\begin{equation}
\label{eq:phiderivatives}
 \| \Delta^k \phi \|_{1} \leq C(k) A^{2k},
\end{equation}
for all $k=0,1,2, \dots$.
As $A$ gets larger, this allows for more functions $\phi$, which can then be chosen to approximate the characteristic function of a disc of radius $A^{-1}$, for example.
This formulation is good for the triple product approach.  If $\phi$ is an approximation to the characteristic function of a disc with fixed center and of radius $A^{-1}$ then $\| \phi \|_1 = \langle 1, \phi \rangle \asymp A^{-2} \asymp \langle \phi, \phi \rangle$ so $\|\phi \|_2 \asymp A^{-1}$.
Another interesting choice of $\phi$ is $\phi(x+iy) = w(x) \psi(y)$ or $\phi(x+iy) = \psi(x) w(y)$ where $\psi$ is a fixed smooth compactly-supported function on either $\mathbb{R}^+$ or $\mathbb{Z}\backslash \mathbb{R}$ as in Conjecture \ref{conj:geodesicQUE} or Conjecture \ref{conj:horocycle}, respectively, and $w$ satisfies $w^{(k)} \ll A^k$.  With such choices of $\phi$ we can approximate a segment of a vertical geodesic or a horocycle, for instance.  For such $\phi$, we have $\langle 1, \phi \rangle \asymp A^{-1} \asymp \langle \phi, \phi \rangle$ so $\|\phi \|_2 \asymp A^{-1/2}$.  

The discussion in this section has some connections to the recent work of \cite{GRS}, especially their Appendix A on ``quantitative QUE,'' but the overlap is minimal because here we focus on understanding precise rates of convergence.

\subsection{The triple product approach}
\label{section:tripleproductapproach}
\begin{myprop}
\label{prop:QUEshrinkingWatsonapproach}
 Let $U$ be a Hecke-Maass cusp form, and suppose $\phi$ (possibly depending on $U$) satisfies \eqref{eq:phiderivatives} for some $A \leq T^{1-\delta}$.  Assuming the Lindel\"{o}f hypothesis for triple product $L$-functions, we have
\begin{equation}
\label{eq:QUEthin1}
 \langle U^2, \phi \rangle = \langle 1, \phi \rangle + O(\|\phi\|_2 T^{-1/2 + \varepsilon} A^{1/2}).
\end{equation}
Alternatively, if in addition we assume the bound
$|u_j(z) | \ll T^{\varepsilon}$ uniformly for $z$ in the support of $\phi$, then we have
\begin{equation}
\label{eq:QUEthin2}
\langle U^2, \phi \rangle = \langle 1, \phi \rangle + O(\|\phi\|_1 T^{-1/2 + \varepsilon} A^{3/2}).
\end{equation}
The implied constants depend on $\varepsilon > 0$ and the choice of constants $C(k)$ in \eqref{eq:phiderivatives}.
\end{myprop}
% \begin{equation}
%  \langle U^2, \phi \rangle = \langle 1, \phi \rangle + O(A^{1/2} T^{-1/2 + \varepsilon} \| \phi \|_2).
% \end{equation}
We conclude from \eqref{eq:QUEthin2} that QUE should hold for any such sequence of $\phi$'s provided $A \leq T^{1/3 - \delta}$ for some fixed $\delta > 0$.

Sometimes the former bound \eqref{eq:QUEthin1} is superior to \eqref{eq:QUEthin2}, even though \eqref{eq:QUEthin2} requires additional assumptions; for instance, if $\langle 1, \phi \rangle \asymp A^{-1} \asymp \langle \phi, \phi \rangle$
then \eqref{eq:QUEthin1} says
\begin{equation}
 \langle U^2, \phi \rangle = \langle 1, \phi \rangle + O(T^{-1/2 + \varepsilon} ),
\end{equation}
so as long as $A \leq T^{1/2-\delta}$ for some fixed $\delta > 0$ we can conclude that QUE holds.

It is implicit that even when QUE does not hold, the above work gives upper bounds on the $L^2$ norm of $U$ restricted to a shrinking family of sets, conditionally on the Lindel\"{o}f hypothesis.

We also emphasize that we do not expect these results to be optimal, even though they rely on the Lindel\"{o}f hypothesis.  The reason is that in the derivation, there appears a sum over the spectrum, and it is possible that there is cancellation when combining the spectral coefficients.  This situation is similar to the prime geodesic theorem where the analog of the Riemann Hypothesis holds (this means there are no exceptional eigenvalues for $PSL_2(\mz)$), yet from this one does not immediately deduce the presumably optimal ``square-root'' error term; see \cite{IwaniecGeodesic}.
Instead see for example p.139 of \cite{IwaniecGeodesic} where a natural conjecture on sums of Kloosterman sums is stated, which would then give the optimal error term.  In Section \ref{section:PoincareQUE} we show how the strongest possible error terms could follow from robust shifted convolution sum bounds.

\begin{proof}
The Plancherel formula gives
\begin{equation}
\label{eq:PlancherelUsquaredphi}
 \langle U^2, \phi \rangle = \langle U^2, \tfrac{3}{\pi}  \rangle \langle 1, \phi \rangle + \sum_{j \geq 1} \langle U^2, u_j \rangle \langle u_j, \phi \rangle + \frac{1}{4\pi} \intR \langle U^2, E(\cdot,1/2 + it) \rangle \langle E(\cdot,1/2 + it), \phi \rangle dt,
\end{equation}
where recall
the inner product is with respect to $\frac{dx dy}{y^2}$ (not probability measure).  Note that if $U$ is normalized with \eqref{eq:L2normalization}, then $\langle U^2, \frac{3}{\pi} \rangle = 1$ and the constant eigenfunction gives the expected main term in the QUE conjecture.  If $u_j$ is even and $U$ is a Maass form then Watson's formula reads
\begin{equation}
\label{eq:Watsonformula}
 |\langle U^2, u_j \rangle |^2 = \frac{\pi}{8} \frac{|\Gamma(\frac{\frac12 + 2iT + i t_j}{2})|^2 |\Gamma(\frac{\frac12 + 2iT - i t_j}{2})|^2 |\Gamma(\frac{\frac12 +  i t_j}{2})|^4}{|\Gamma(\frac{1 + 2iT }{2})|^4 |\Gamma(\frac{1 + 2it_j }{2})|^2} \frac{L(U \times U \times u_j, 1/2)}{L(\mathrm{sym}^2 U, 1)^2 L(\mathrm{sym}^2 u_j, 1)}.
\end{equation}
A similar formula holds for the Eisenstein series by an unfolding argument along the lines of Lemma \ref{lemma:RankinSelbergIntegral}.  
We shall simply quote the work of Section 2 of \cite{LuoSarnakQUE} for the following:
\begin{multline}
\label{eq:tripleproductTwoEisensteinOneMaass}
 \langle |E(z, 1/2 + iT)|^2 , u_j \rangle 
\\
= c(T) \rho_j(1) \frac{|\Gamma(\frac{\frac12 + it_j}{2} )|^2}{|\zeta(1 + 2iT)|^2} \frac{\Gamma(\frac{\frac12 - it_j - 2iT}{2}) \Gamma(\frac{\frac12 + it_j - 2iT}{2})}{|\Gamma(\frac12 + iT)|^2} L(u_j, 1/2) L(u_j, 1/2 - 2iT),
\end{multline}
where $c(T)$ is such that $|c(T)|$ is an absolute constant independent of $T$.
Note that $|\langle |E(z, 1/2 + iT)|^2 , u_j \rangle|^2$ is a close cousin to 
\eqref{eq:Watsonformula}, recalling that $|\rho_j(1)|^2 = c |\Gamma(\tfrac12 + it_j)|^{-2} (L(1, \mathrm{sym}^2 u_j))^{-1} $.  In particular, the gamma factors are the same as in \eqref{eq:Watsonformula}, and the triple product $L$-function is replaced by $|L(u_j, 1/2) L(u_j, 1/2 - 2iT)|^2$.

By Stirling's formula, the ratio of gamma factors in \eqref{eq:Watsonformula} is
\begin{equation}
\label{eq:gammafactorboundinWatson}
 %|\langle U^2, u_j \rangle |^2 
\ll P(t_j, T) \exp(\tfrac{\pi}{2} Q(t_j,T)), \quad Q(t_j, T)  = 4T - |2T + t_j| - |2T - t_j|,
\end{equation}
where we encountered $Q(t,T)$ earlier in \eqref{eq:Qdef}, and where
$P$ is given by
\begin{equation}
\label{eq:Pdef}
 P(t_j, T) = (1 + |2T+t_j|)^{-1/2} (1 + |2T - t_j|)^{-1/2} t_j^{-1}.
\end{equation}
Since $Q(t_j, T) = 2(2T-t_j)$ for $t_j \geq 2T$ (and $=0$ otherwise) we can bound the terms in \eqref{eq:PlancherelUsquaredphi} with $t_j \geq 2T + C \log{T}$ with
\begin{equation}
\| \phi \|_{1} \sum_{t_j \geq 2T + C \log{T}} \text{Polynomial}(t_j, T)  \exp(-\pi (t_j - 2T)),
\end{equation}
using the trivial bound $|\langle u_j, \phi \rangle | \leq \| u_j \|_{\infty} \| \phi \|_1 \ll (1/4 + t_j^2)^{1/4} \| \phi \|_1$.
Taking $C$ large enough compared to the degree of the unspecified polynomial, we can bound this error term by $\ll \| \phi \|_1 T^{-100}$, with an absolute implied constant.  In summary, we have shown
\begin{equation}
\label{eq:spectraldecomposition}
 \langle U^2, \phi \rangle - \langle 1, \phi \rangle = \sum_{t_j \leq 2T + C \log{T}} \langle U^2, u_j \rangle \langle u_j, \phi \rangle + (\text{Eisenstein}) + O(\|\phi\|_1 T^{-100}).
\end{equation}
Without some additional assumptions on $\phi$, it is not reasonable to expect that the terms with $t_j \leq 2T + C \log{T}$ also constitute an error term.  For instance, if $\phi = U^2$ then in \cite{BKY} it is conjectured that $\frac{3}{\pi} \langle U^2, U^2 \rangle \sim 3$.  In an even more extreme direction, we could take a sequence of $\phi$'s tending to a delta function in which case one would not expect an asymptotic law for $\langle U^2, \phi \rangle$.  %One natural case to consider is where $\phi$ is a smooth and compactly-supported function approximating the characteristic function of a small disc of radius $A^{-1}$ with $A = A(T)$ tending to infinity with $T$ at some  as yet unspecified rate.  

Assuming \eqref{eq:phiderivatives} holds, then by the self-adjointness of the Laplacian, we can bound the spectral coefficients by
\begin{equation}
\label{eq:boundingujphi}
 (1/4 + t_j^2)^k \langle u_j, \phi \rangle = \langle \Delta^k u_j, \phi \rangle = \langle  u_j, \Delta^k \phi \rangle  \ll A^{2k} \| u_j \|_{\infty},
\end{equation}
whence
\begin{equation}
 \langle u_j, \phi \rangle \ll (1/4 + t_j^2)^{1/4} \Big(\frac{A^2}{1/4 + t_j^2}\Big)^{k}.
\end{equation}
Thus if $A\leq T^{1-\delta}$ for some fixed $\delta > 0$ (meaning in some sense that $u_j$ is more oscillatory than $\phi$) then in \eqref{eq:spectraldecomposition} we can truncate the sum at $t_j \leq A T^{\varepsilon}$.  We conclude that %by Cauchy's inequality and Bessel's inequality that
\begin{equation}
 |\langle U^2, \phi \rangle - \langle 1, \phi \rangle | \ll \|\phi \|_1 \sum_{t_j \leq A T^{\varepsilon}} \| u_j \|_{\infty} | \langle U^2, u_j \rangle| + (\text{Eisenstein}) + O(\|\phi \|_1 T^{-100}).
\end{equation}
% \begin{equation}
%  |\langle U^2, \phi \rangle - \langle 1, \phi \rangle |^2 \leq \langle \phi, \phi \rangle  \sum_{t_j \leq A T^{\varepsilon}} | \langle U^2, u_j \rangle|^2 + \text{Eisenstein} + O(T^{-100}).
% \end{equation}
Assuming the Lindel\"{o}f Hypothesis, we have by Watson's formula that
\begin{equation}
 \sum_{t_j \leq A T^{\varepsilon}} | \langle U^2, u_j \rangle|  \ll\sum_{t_j \leq A T^{\varepsilon}}  T^{-1/2+\varepsilon} t_j^{-1/2+\varepsilon} \ll A^{3/2} T^{-1/2 + \varepsilon}.
\end{equation}
% \begin{equation}
%  \sum_{t_j \leq A T^{\varepsilon}} | \langle U^2, u_j \rangle|^2 \ll\sum_{t_j \leq A T^{\varepsilon}}  T^{-1+\varepsilon} t_j^{-1+\varepsilon} \ll A T^{-1 + \varepsilon}.
% \end{equation}
Similar estimates hold for the Eisenstein series so we suppress those arguments.  Using the assumed bound $\|u_j \|_{\infty} \ll T^{\varepsilon}$ (we only need this for $u_j$ restricted to the support of $\phi$), we derive \eqref{eq:QUEthin2}.

To derive \eqref{eq:QUEthin1}, we use the arrangement
\begin{equation}
 |\langle U^2, \phi \rangle - \langle 1, \phi \rangle | \ll \Big(\sum_{t_j \leq A T^{\varepsilon}} |\langle U^2, u_j \rangle |^2 \Big)^{1/2} \Big(\sum_{t_j } |\langle u_j, \phi \rangle |^2 \Big)^{1/2} + (\text{Eisenstein}) + O(\|\phi \|_1 T^{-100}).
\end{equation}
In this case, Bessel's inequality implies $\sum_{t_j} |\langle u_j, \phi \rangle |^2 \leq \langle \phi, \phi \rangle$. Thus by Watson's formula and the Lindel\"{o}f hypothesis, we derive \eqref{eq:QUEthin1}.
\end{proof}

\subsection{Poincare series approach}
\label{section:PoincareQUE}
Now we consider the approach to QUE by Poincare series as in Section 4 of \cite{LuoSarnakQUE}.  This method gives the following
\begin{myprop}
 Suppose that a family of functions $\phi: \Gamma \backslash \mh \rightarrow \mr$ satisfy for all $k,l = 0, 1,2, \dots$
\begin{equation}
\label{eq:phiderivativesCartesian}
 \frac{\partial^{k+l}}{\partial x^k \partial y^l} \phi(x + iy) \ll_{k,l} A^k B^l,
\end{equation}
and that each $\phi$ in the family has support contained in a fixed compact set $K$.  Let $U$ be a Hecke-Maass cusp form.  Suppose the Lindel\"{o}f hypothesis holds for $L(\mathrm{sym}^2 U, s)$, and assume Conjecture \ref{conj:SCH} holds with $\tau = 1/2$, $R=0$.  Then
\begin{equation}
\label{eq:QUEthinPoincare}
 \langle U^2, \phi \rangle = \langle 1, \phi \rangle + O(\|\phi\|_1 T^{-1/2 + \varepsilon}(A + B^{1/2})).
\end{equation}
If in addition Conjecture \ref{conj:SCH2} holds then
\begin{equation}
\label{eq:QUEthinPoincare2}
 \langle U^2, \phi \rangle = \langle 1, \phi \rangle + O(\|\phi\|_1 T^{-1/2 + \varepsilon}(A^{1/2} + B^{1/2})).
\end{equation}
\end{myprop}
The assumption that each $\phi$ has support in $K$ is to avoid unusual behavior of the functions high in the cusp.  Note that \eqref{eq:phiderivativesCartesian} implies \eqref{eq:phiderivatives} with $A$ replaced by $A + B$, since $y$ is restricted to a compact set.

\begin{proof}
For notational simplicity, suppose that $\overline{K} \subset \mh$ is a connected component of the inverse image of $K$ under the natural projection, and that $\overline{K}$ is contained in the interior of the usual fundamental domain for $\Gamma \backslash\mh$.  The general case can be treated as in Section 4 of \cite{LuoSarnakQUE}; one needs to modify the formula slightly in neighborhoods of the elliptic points $i, \rho$.  Define $\overline{\phi}(z): \mathbb{H} \rightarrow \mr$ via $\overline{\phi}(z) = \phi(z)$ for $z \in \overline{K}$, and $0$ otherwise.  Define $\Phi$ to be the extension of $\overline{\phi}$ to $\mh$ 
by $\Gamma_{\infty}$-periodicity.  

The usual Fourier expansion for $\Phi$ takes the form
\begin{equation}
 \Phi(x+iy) = \sum_{m \in \mz} e(mx) \Phi_m(y), \qquad \Phi_m(y) = \int_0^{1} \Phi(x+iy) e(-mx) dx.
\end{equation}
Furthermore, $\phi(z) = \sum_{\gamma \in \Gamma_{\infty} \backslash \Gamma} \Phi(\gamma z)$.  
For $\psi$ a compactly-supported function on the positive reals, define the incomplete Poincare series
\begin{equation}
 P_n(z, \psi) = \sum_{\gamma \in \Gamma_{\infty} \backslash \Gamma} e(n \text{Re}(\gamma z)) \psi(\text{Im}(\gamma z)).
\end{equation}
Thus
\begin{equation}
\label{eq:PoincareSeriesExpansion}
 \phi(x+iy) = \sum_{\gamma \in \Gamma_{\infty} \backslash \Gamma} \sum_{m \in \mz} e(m \text{Re}(\gamma z)) \Phi_m(\text{Im}(\gamma z)) = \sum_{m \in \mz} P_m(z, \Phi_m).
\end{equation}
As in the triple product method of Section \ref{section:tripleproductapproach}, we consider a sequence of functions $\phi$ and we wish to impose conditions that allow us to specify a practical place to truncate the sum over $m$.  To this end, we note that by integration by parts,
\begin{equation}
 \Phi_m(y) = \Big(\frac{1}{2 \pi i m} \Big)^k \int_0^{1} \frac{\partial^k \Phi(x + iy)}{\partial x^k}  e(-mx) dx.
\end{equation}
So if \eqref{eq:phiderivativesCartesian} holds, then
 $\Phi_m(y) \ll (A/|m|)^k$.

By the Poincare series expansion \eqref{eq:PoincareSeriesExpansion}, we have
\begin{equation}
\label{eq:PoincareSeriesExpansionInnerProduct}
 \langle U^2,   \phi \rangle = \sum_{m \in \mz} \langle U^2, P_m(\cdot,  \Phi_m) \rangle.
\end{equation}
We can already see a potential improvement over \eqref{eq:PlancherelUsquaredphi}--there are roughly $A$ terms in \eqref{eq:PoincareSeriesExpansionInnerProduct} while \eqref{eq:PlancherelUsquaredphi} has roughly $A^2$ Maass forms with $t_j \leq A$.
Next we calculate each of these inner products by unfolding (see the proof of Lemma \ref{lemma:RankinSelbergIntegral} for a similar calculation):
\begin{equation}
\label{eq:UsquaredPmunfolded}
 \langle U^2, P_m(z,  \Phi_m) \rangle = \int_0^{\infty} \int_0^{1} |U(x+iy)|^2 e(-mx)  \Phi_m(y) \frac{dx dy}{y^2}.
\end{equation}
Hence by Conjecture \ref{conj:SCH}, with $\tau = 1/2$, $R=0$, we have
\begin{equation}
 \langle U^2, P_m(z,  \Phi_m) \rangle  \ll T^{-1/2 + \varepsilon} \int_0^{\infty} |\Phi_m(y)| \frac{dy}{y^2} \leq T^{-1/2 + \varepsilon} \int_0^{1} \int_0^{\infty} |\Phi(x+iy)| \frac{dx dy}{y^2}.
\end{equation}
By truncating the sum at $|m| \leq AT^{\varepsilon}$ with a very small error, we conclude
\begin{equation}
 \sum_{m \neq 0} \langle U^2, P_m(z,  \Phi_m) \rangle  \ll \|\phi \|_1 T^{-1/2 + \varepsilon} A,
\end{equation}
which is the first of two error terms claimed in \eqref{eq:QUEthinPoincare}.  If one is willing to accept Conjecture \ref{conj:SCH2}, then we can 
show
\begin{equation}
\label{eq:UsquaredPmboundsummedoverm}
 \sum_{m \neq 0} \langle U^2, P_m(z,\Phi_m) \ll \|\phi \|_1  T^{-1/2 + \varepsilon} A^{1/2},
\end{equation}
as follows.
Using \eqref{eq:UsquaredPmunfolded}, the definition of $\Phi_m$, and rearranging the orders of integration and summation appropriately, we have
\begin{multline}
 \sum_{1 \leq |m| \leq A T^{\varepsilon}} \langle U^2, P_m(z,  \Phi_m) \rangle
\\
= \int_0^{1} \int_0^{\infty} \Phi(t+iy) \Big( \sum_{1 \leq |m| \leq A T^{\varepsilon}} e(mt) \int_0^{1} |U(x+iy)|^2 e(-mx) dx \Big) \frac{dy}{y^2} dt.
\end{multline}
The inner expression inside the parentheses is $O(A^{-1/2} T^{1/2 + \varepsilon})$, assuming Conjecture \ref{conj:SCH2}, which immediately leads to \eqref{eq:UsquaredPmboundsummedoverm}, the first of two error terms stated in \eqref{eq:QUEthinPoincare2}.

% Thus
% \begin{equation}
%  \sum_{m \neq 0} \langle U^2, P_m(z,  \phi_m) \rangle = \int_0^{\infty} \Big(\sum_{1 \leq |m| \leq AT^{\varepsilon}} \phi_m(y) \int_0^{1} |U(x+iy)|^2 e(mx)  dx \Big) \frac{ dy}{y^2},
% \end{equation}
% which by Conjecture \ref{conj:SCH} gives

% Inserting the Fourier expansion, \eqref{eq:UFourier}, we obtain
% \begin{equation}
%  \langle U^2, P_m(z,  \phi_m(y)) \rangle = |\rho(1)|^2 \sum_{n \neq 0, -m} \frac{\lambda(n) \lambda(n+m)}{\sqrt{|n(n+m)|}} \int_0^{\infty} \phi_m(y) V(2 \pi |n| y) V(2 \pi |n+m| y) \frac{dy}{y^2}.
% \end{equation}

% In any event, we derive the nice formula
% \begin{equation}
%  \langle U^2,  \phi \rangle = |\rho(1)|^2 \sum_{m \in \mz} \sum_{n \neq 0, -m} \frac{\lambda(n) \lambda(n+m)}{\sqrt{|n(n+m)|}} \int_0^{\infty} y^{-1} \phi_m(y) V(2 \pi |n| y) V(2 \pi |n+m| y) \frac{dy}{y}.
% \end{equation}
% For the record, the term with $m=0$ gives
% \begin{equation}
%  |\rho(1)|^2 \sum_{n \neq 0} \frac{\lambda(n)^2}{|n|} \int_0^{\infty} y^{-1} \phi_0(y) V(2 \pi |n|y)^2 \frac{dy}{y}.
% \end{equation}
% As in Section 3 of \cite{LuoSarnakQUE}, this term with $m=0$ should, in many cases, be closely approximated by $\langle 1, \phi \rangle$, but proving this (with $\phi$ fixed) relies on subconvexity for $L(\mathrm{sym}^2 U, s)$ in the $U$-aspect.

Next we examine the term $m=0$, and show that
\begin{equation}
 \langle U^2, P_0(z, \Phi_0) \rangle = \langle 1, \phi \rangle + O(\|\phi\|_1 T^{-1/2 + \varepsilon} B^{1/2}),
\end{equation}
assuming the generalized Lindel\"{o}f hypothesis.
Here we have $P_0(z, \Phi_0) = E(z, \Phi_0)$ and so \eqref{eq:UsquaredinnerproductincompleteEisenstein} gives
\begin{equation}
 \langle U^2, P_0(\cdot, \Phi_0) \rangle = \frac{1}{2 \pi i} \int_{(\varepsilon)} \widetilde{\Phi_0}(-1-s) 2|\rho(1)|^2 Z(1+s, U) \gamma_{V_T^2}(1+s)
ds.
\end{equation}
Recall $Z(s, U)$ is defined by \eqref{eq:Zdef} and $\gamma_{V_T^2}(s)$ is given by \eqref{eq:gammaVsquaredBessel}.  We shift the contour of integration to $\text{Re}(s) = -1/2$, crossing a pole at $s=0$ which gives
\begin{equation}
 \widetilde{\Phi_0}(-1) \langle U^2, \tfrac{3}{\pi} \rangle = \int_0^{\infty} \Phi_0(y) \frac{dy}{y^2} = \int_0^{\infty} \int_0^{1} \Phi(x+iy) \frac{dx dy}{y^2} = \langle 1, \phi \rangle,
\end{equation}
the stated main term in \eqref{eq:QUEthinPoincare} and \eqref{eq:QUEthinPoincare2}.  By Stirling's formula and Lindel\"{o}f, the new contour integral is of size
\begin{equation}
\label{eq:PoincareErrorterm}
\ll T^{\varepsilon} \intR |\widetilde{\Phi_0}(-1/2 + it)| (1 +|t - 2T|)^{-1/4} (1 +|t + 2T|)^{-1/4} (1 + |t|)^{-1/2+\varepsilon} dt.
\end{equation}
In fact there is extra exponential decay in the integrand for $|t| \geq 2T$, but we do not need this.  It is a simple matter of integration by parts to show
% We derive
% \begin{equation}
%  \frac{\partial^{l}}{ \partial y^l} \phi_m(y) = \Big(\frac{1}{2 \pi i m} \Big)^k \int_0^{1} \frac{\partial^{k+l} \phi(x + iy)}{\partial x^k \partial y^l}  e(-mx) dx \ll_{k,l} (A/m)^k B^l.
% \end{equation}
% We also deduce
\begin{multline}
\label{eq:phitildebound}
 \widetilde{\Phi_m}(s) := \int_0^{\infty} \Phi_m(y) y^s \frac{dy}{y} = \frac{(-1)^l}{s(s+1)\dots(s+l-1)} \int_0^{\infty} \frac{\partial^l \Phi_m(y)}{\partial y^l}  y^{s+j-1} dy 
\\
\ll_{k,l} \frac{B^l}{|s(s+1)\dots(s+l-1)|}.
\end{multline}
%Be aware that $\widetilde{\phi}_m(s)$ does not have poles since $\phi_m$ has compact support on the positive reals.
Thus in \eqref{eq:PoincareErrorterm} we can truncate the integral at $|t| \leq B T^{\varepsilon}$ at no cost.  Hence \eqref{eq:PoincareErrorterm} is
$\ll T^{-1/2 + \varepsilon} B^{1/2}$, which is the other error term stated in \eqref{eq:QUEthinPoincare} and \eqref{eq:QUEthinPoincare2}.
\end{proof}

\subsection{QUE with shrinking sets for Eisenstein series}
\label{section:QUEEIsensteinsubsection}
In this section we prove Proposition \ref{prop:thinEisensteinQUE}.
Our approach most naturally shows a more precise form
\begin{equation}
\label{eq:thinQUEEisenstein2}
 \langle |E(z, 1/2 + iT)|^2, \phi \rangle = \lim_{\alpha \rightarrow 0} \langle D_{\alpha}, \phi \rangle  + O(A^{1/2} T^{-1/6 + \varepsilon} \| \phi \|_2),
\end{equation}
where
\begin{equation}
\label{eq:Dalphadef}
 D_{\alpha}(z) = E(z,1+\alpha) + \Phi_{T}(\alpha)
 E(z, 1-\alpha), \quad \Phi_{T}(\alpha) = \frac{\theta(1/2 - iT - \alpha)}{\theta(1/2 + iT + \alpha)} \frac{\theta(1/2 + iT )}{\theta(1/2 - iT)}.
\end{equation}
Write $E(z, 1 + \alpha) = \frac{3/\pi}{\alpha} + a(z) + O(\alpha)$.  
The constant term in the Taylor expansion for $D_{\alpha}(z)$ around $\alpha = 0$ is
\begin{equation}
 2 a(z) - \frac{3}{\pi} \Phi_T'(0).
\end{equation}
One can find expressions for $a(z)$ via (22.69) of \cite{IK}, and by a calculation,
\begin{equation}
 \Phi_T'(0) = \frac{\Phi_T'}{\Phi_T}(0) = - \sum_{\pm   } \Big(\frac{\Gamma'}{\Gamma}(1/2 \pm iT) + 2 \frac{\zeta'}{\zeta}(1 \pm 2  i T) - \log{\pi} \Big).
\end{equation}
By Stirling's formula, $\frac{\Gamma'}{\Gamma}(1/2 + iT) + \frac{\Gamma'}{\Gamma}(1/2 - iT) = \log(1/4 + T^2) + O(T^{-2})$, while $\frac{\zeta'}{\zeta}(1 + \pm 2iT) = O(\frac{\log T}{\log \log{T}})$, which explains how \eqref{eq:thinQUEEisenstein} follows from \eqref{eq:thinQUEEisenstein2}.

One can also compare Proposition \ref{prop:thinEisensteinQUE} with the Maass-Selberg relation (cf. Proposition 6.8 of \cite{IwaniecSpectral}).

The Luo-Sarnak \cite{LuoSarnakQUE} approach proceeds by showing \eqref{eq:QUEEisenstein} for $\phi$ a Maass form, or $\phi$ an incomplete Eisenstein series.  The incomplete Eisenstein series span the space $\mathcal{E}(\Gamma \backslash \mh)$ which is the orthogonal complement of the span of the Maass forms $\mathcal{C}(\Gamma \backslash \mh)$ (here we use notation as in Iwaniec's book \cite{IwaniecSpectral}), so this suffices to show \eqref{eq:QUEEisenstein}, though with an inexplicit error term.  On the other hand, the error term in \eqref{eq:QUEEisenstein} when $\phi$ is a fixed Maass form gives a power saving in $T$ (and similarly for a fixed incomplete Eisenstein series, but in this case there is a lower-order term that must be included as in \eqref{eq:thinQUEEisenstein2}).  However, it seems difficult to constructively approximate the projection of $\phi$ onto $\mathcal{E}(\Gamma \backslash \mh)$ in terms of incomplete Eisenstein series.  Of course, the spectral expansion does give such a decomposition of the projection of $\phi$, but in terms of Eisenstein series themselves which have moderate growth at the cusp, even though the projection of $\phi$ has rapid decay at the cusp.  
This is the main technical difficulty in analyzing \eqref{eq:QUEEisenstein} using the spectral decomposition and Parseval's formula, because $|E(z,1/2 + iT)|^2$ grows too fast at the cusp. 
However, there is a way around this problem of convergence that was discovered by Zagier \cite{Zagier}, namely, to work with renormalized integrals.  Michel and Venkatesh \cite{MichelVenkateshIHES} have recently given an interpretation of this renormalization in the language of representation theory.  In addition they give a regularized Plancherel formula that we shall use here; see the Proposition on p.243 of \cite{MichelVenkateshIHES}.  See also \cite{ZelditchEisenstein} for an application of regularization to quantum ergodicity with Eisenstein series.

First we define the regularized inner product.  Suppose that $F$ is a function of moderate growth which by definition means that 
\begin{equation}
 F(z) = \varphi(y) + O(y^{-N})
\end{equation}
as $y \rightarrow \infty$, for any $N > 0$, where
\begin{equation}
 \varphi(y) = \sum_{i=1}^{l} \frac{c_i}{n_i!} y^{\alpha_i} \log^{n_i} y,
\end{equation}
for $c_i, \alpha_i \in \mc$, and $n_i \geq 0$ an integer.  Suppose that no $\alpha_i = 1$.  Let $\mathcal{E}(z)$ denote a linear combination of  Eisenstein series $E(z,\alpha)$ (possibly including derivatives with respect to $\alpha$) with $\text{Re}(\alpha) > 1/2$ such that $F(z) - \mathcal{E}(z) = O(y^{1/2})$.  The regularized integral of $F$ is then defined to be
\begin{equation}
\label{eq:regularizedintegraldefinition}
\int_{\Gamma \backslash \mh}^{\text{reg}} F(z) \frac{dx dy}{y^2} := \int_{\Gamma \backslash \mh} (F(z) - \mathcal{E}(z)) \frac{dx dy}{y^2}.
\end{equation}
Zagier computed some special cases of these regularized integrals that we require.  It is obvious from the definition that $\langle E(z,s), 1 \rangle_{\text{reg}} = 0$.  For $0 < \text{Re}(s) < 1$, $E(z,s) \in L^1(\Gamma \backslash \mh)$, and it is easy to see that the integral (without regularization) vanishes using the fact that the Laplacian is self-adjoint, yet $E(z,s)$ and $1$ have different Laplace eigenvalues.  For the case of two Eisenstein series, we have $\langle E(z, s_1), E(z, \overline{s_2}) \rangle_{\text{reg}} = 0$; see p.428 of \cite{Zagier}.  This formula should be heuristically natural because again $E(z, s_1)$ and $E(z,s_2)$ have distinct Laplace eigenvalues (for $s_1(1-s_1) \neq s_2(1-s_2)$).
Technically, to even define the regularized integral we require that $s_1 \neq s_2$ and $s_1 \neq 1-s_2$. %, but then one naturally extends the definition to all $s_1, s_2$ by meromorphic continuation.  
Finally, for the case of the product of three Eisenstein series, Zagier computes (see p.430 of \cite{Zagier})
that
\begin{multline}
\label{eq:threeEisensteinseries}
 \langle E(z, \tfrac12+s_1) E(z, \tfrac12+s_2), E(z,\tfrac12+\overline{s}) \rangle_{\text{reg}} 
=
\\
 %\frac{4 \pi^{-1/2-s}}{{ \zeta(1+2s) \Gamma(1/2+s)} }  
\frac{c}{\theta(\frac12 + s) \theta(\frac12 + s_1) \theta(\frac12 + s_2)}
%\\
%\times
\prod_{\delta_1, \delta_2 \in \{\pm 1\}} \zeta(\tfrac12 + s +\delta_1 s_1 +\delta_2 s_2) \Gamma\Big(\frac{\frac12 + s +\delta_1 s_1 +\delta_2 s_2}{2}\Big).
\end{multline}

\begin{mylemma}[\cite{MichelVenkateshIHES}]
\label{lemma:Plancherel}
 Suppose that $F$ and $G$ are smooth functions on $\Gamma \backslash \mh$, $G$ with compact support, $F$ of moderate growth with $\text{Re}(\alpha_i) \neq 1/2$ for all $i$ and no $\alpha_i = 1$.  Then the following regularized version of Parseval's formula holds:
\begin{equation}
\label{eq:RegularizedPlancherel}
 \langle F, G \rangle_{} = \sum_{j} \langle F, u_j \rangle \langle u_j, G \rangle + \langle F, \tfrac{3}{\pi} \rangle_{\text{reg}} \langle  1,  G\rangle + \frac{1}{4 \pi} \intR \langle F, E_t \rangle_{\text{reg}} \langle E_t, G \rangle dt
%\\
+ \langle \mathcal{E}, G \rangle,
\end{equation}
where $E_t$ denotes $E(z, 1/2 + it)$.
\end{mylemma}
\begin{proof}
We give a minor variation of the proof of Michel and Venkatesh \cite{MichelVenkateshIHES}.  Suppose that $F_1(z) := F(z) - \mathcal{E}(z) = O(y^{1/2-\delta})$ for some $\delta > 0$, whence $F_1 \in L^2(\Gamma \backslash \mh)$.  Since $G$ has rapid decay, we have $\langle F, G \rangle = \langle F_1, G \rangle + \langle \mathcal{E}, G \rangle$ as absolutely convergent integrals.  Then by the usual Plancherel formula, we have
\begin{equation}
 \langle F_1, G \rangle = \sum_j \langle F_1,u_j \rangle \langle u_j, G \rangle + \langle F_1, \tfrac{3}{\pi} \rangle \langle  1,  G\rangle + \frac{1}{4 \pi} \intR \langle F_1, E_t \rangle \langle E_t, G \rangle dt.
\end{equation}
Note that $\langle F_1, u_j \rangle = \langle F, u_j \rangle$ since $\langle \mathcal{E}, u_j \rangle = 0$, and $\langle F_1, \frac{3}{\pi} \rangle = \langle F, \frac{3}{\pi} \rangle_{\text{reg}}$ by definition.  As mentioned above, $\langle \mathcal{E}, E_t \rangle_{\text{reg}} = 0$, so
$\langle F_1, E_t \rangle = \langle F, E_t \rangle_{\text{reg}}$.  Gathering the terms finishes the proof.
\end{proof}

\begin{proof}[Proof of Proposition \ref{prop:thinEisensteinQUE}]
We shall apply Lemma \ref{lemma:Plancherel} with $G = \phi$ and $F(z) = E(z, s_1) E(z, s_2)$.  We shall eventually let $s_1 = 1/2 + iT$ and $s_2 = 1/2 - iT$ by analytic continuation.  In fact, we shall be able to do this for each of the four terms appearing in \eqref{eq:RegularizedPlancherel}.

By following the arguments in Section \ref{section:tripleproductapproach}, we have that
\begin{equation}
 \sum_j \langle |E(z, 1/2 + iT)|^2, u_j \rangle \langle u_j, \phi \rangle \leq \|\phi \|_2 \Big(\sum_{t_j \leq A T^{\varepsilon}} |\langle |E(z, 1/2 + iT)|^2 , u_j \rangle|^2 \Big)^{1/2} + O(T^{-100}).
\end{equation}
Consulting \eqref{eq:gammafactorboundinWatson} and \eqref{eq:Pdef}, we then have
\begin{multline}
 \sum_j \langle |E(z, 1/2 + iT)|^2, u_j \rangle \langle u_j, \phi \rangle 
\\
\ll \|\phi \|_2 T^{-1/2+\varepsilon} \Big(\sum_{t_j \leq A T^{\varepsilon}} t_j^{-1} L(u_j, 1/2)^2 |L(u_j, 1/2-2iT)|^2 \Big)^{1/2} 
+ O(T^{-100}).
\end{multline}
Next we apply the uniform subconvexity bound $L(u_j, 1/2 - 2iT) \ll T^{1/3 + \varepsilon}$ of Jutila-Motohashi \cite{JutilaMotohashi}, and the following bound which follows from the spectral large sieve inequality
\begin{equation}
\label{eq:secondmoment}
\sum_{t_j \leq AT^{\varepsilon}} L(u_j, 1/2)^2 \ll A^2 T^{\varepsilon}.
\end{equation}
In this way, we obtain
\begin{equation}
\label{eq:EsquaredprojectionontoMaass}
 \sum_j \langle |E(z, 1/2 + iT)|^2, u_j \rangle \langle u_j, \phi \rangle \ll T^{-1/6 + \varepsilon} A^{1/2} \|\phi \|_2,
\end{equation}
which is the error term stated in \eqref{eq:thinQUEEisenstein2}.

Next we examine the regularized projections of $|E|^2$ onto the constant eigenfunction and the Eisenstein series.  The inner product with the constant eigenfunction vanishes as remarked following \eqref{eq:regularizedintegraldefinition} (after taking a limit to treat the case $s_1 = -s_2 = iT$).  For the Eisenstein contribution, we use \eqref{eq:threeEisensteinseries}.  Note that $|\langle |E(z, 1/2 + iT)|^2, E_t \rangle_{\text{reg}} |^2$ takes the form
\begin{equation}
 c \frac{|\Gamma(\frac{\frac12 + 2iT + i t}{2})|^2 |\Gamma(\frac{\frac12 + 2iT - i t}{2})|^2 |\Gamma(\frac{\frac12 +  i t}{2})|^4}{|\Gamma(\frac{1 + 2iT }{2})|^4 |\Gamma(\frac{1 + 2it }{2})|^2} 
\frac{|\zeta(\tfrac12 + it + 2iT)|^2 |\zeta(\tfrac12 + it)|^4 |\zeta(\tfrac12 + it - 2iT)|^2}{|\zeta(1 + 2it)|^2 |\zeta(1+2iT)|^4},
\end{equation}
for some absolute constant $c$.  This, as expected, has the exact shape as \eqref{eq:Watsonformula}.  By a similar argument as \eqref{eq:boundingujphi}, we can truncate the integral at $|t| \leq A T^{\varepsilon}$ with a very small error.  In this way, we obtain the bound
\begin{multline}
 \frac{1}{4 \pi} \intR \langle |E(z, 1/2 + iT)|^2, E_t \rangle_{\text{reg}} \langle E_t, \phi \rangle dt \ll  \|\phi\|_2 T^{-100}
\\
+ \|\phi \|_2 T^{-1/2+\varepsilon} \Big(\int_{|t| \leq A T^{\varepsilon}} (1 + |t|)^{-1} |\zeta(\tfrac12 + it)|^4 |\zeta(\tfrac12 + it + 2iT)|^2 |\zeta(\tfrac12 + it - 2iT)|^2 \Big)^{1/2} .
\end{multline}
Using Weyl's bound $\zeta(1/2 + it \pm 2iT) \ll T^{1/6 + \varepsilon}$, and a bound for the fourth moment of zeta (on the level of the mean value theorem for Dirichlet polynomials), we obtain
\begin{equation}
 \frac{1}{4 \pi} \intR \langle |E(z, 1/2 + iT)|^2, E_t \rangle_{\text{reg}} \langle E_t, \phi \rangle dt
\ll T^{-1/6 + \varepsilon} \|\phi\|_2.
\end{equation}
This bound is better than \eqref{eq:EsquaredprojectionontoMaass} due to the smaller spectral measure of the Eisenstein series.

Finally, we evaluate $\langle \mathcal{E}, \phi \rangle$.  First we need to identify $\mathcal{E}(z)$.  The constant term of $E(z, 1/2 + s_1) E(z, 1/2 + s_2)$ is
\begin{equation}
  y^{1+ s_1 + s_2} 
+ %\frac{\theta(1/2-s_1)}{\theta(1/2+s_1)} 
c_1 y^{1-s_1+s_2} 
+ %\frac{\theta(1/2-s_2)}{\theta(1/2+s_2)} 
c_2 y^{1+s_1-s_2} 
+ %\frac{\theta(1/2-s_1)}{\theta(1/2+s_1)} \frac{\theta(1/2-s_2)}{\theta(1/2+s_2)} 
c_1 c_2 y^{1-s_1-s_2}, \quad c_1 = \frac{\theta(1/2-s_1)}{\theta(1/2+s_1)}, \quad c_2 = \frac{\theta(1/2-s_2)}{\theta(1/2+s_2)},
\end{equation}
so
\begin{equation}
 \mathcal{E}(z) = E(z, 1+s_1 + s_2) 
+ %\frac{\theta(1/2-s_1)}{\theta(1/2+s_1)} 
c_1 E(z, 1-s_1 + s_2)
+ %\frac{\theta(1/2-s_2)}{\theta(1/2+s_2)} 
c_2 E(z, 1+s_1 - s_2)
+ %\frac{\theta(1/2-s_1)}{\theta(1/2+s_1)} \frac{\theta(1/2-s_2)}{\theta(1/2+s_2)} 
c_1 c_2 E(z, 1-s_1-s_2).
\end{equation}
Note that by unfolding (and analytic continuation), $\langle E(z, 1-s_1 + s_2), \phi \rangle = \widetilde{\phi_0}(-s_1 + s_2)$, where $\phi_0(y) = \int_0^{1} \phi(x+iy) dx$.  In our application, we have $s_2 = -iT$ and $s_1 = iT+ \alpha$ with $\alpha \rightarrow 0$, and we may use the rapid decay of $\widetilde{\phi_0}(-2iT)$, recalling \eqref{eq:phitildebound}, to absorb this term into the error term (and similarly for $E(z, 1+s_1 - s_2)$).  Thus, we have
\begin{equation}
 \langle \mathcal{E}, \phi \rangle = \langle D_{\alpha}, \phi \rangle + O(T^{-100}),
\end{equation}
where $D_{\alpha}$ is defined by \eqref{eq:Dalphadef}.  This completes the proof of Proposition \ref{prop:thinEisensteinQUE} in the form of \eqref{eq:thinQUEEisenstein2}.
\end{proof}

\subsection{Unconditional upper bounds}
\begin{myprop}
Let $U$ be a Hecke-Maass cusp form, and suppose the family of $\phi$'s have support in a fixed compact set, and  %is a nonnegative function 
satisfy \eqref{eq:phiderivatives} for some $A \leq T^{1-\delta}$.  Then
\begin{equation}
 \label{eq:L2normalongshrinkingdiscsL2}
\langle U^2, \phi \rangle \ll T^{\varepsilon} A^{1/4} \|\phi \|_2.
\end{equation}
If instead of \eqref{eq:phiderivatives} we assume $\frac{\partial^k}{\partial x^k} \phi(x+iy) \ll A^k$, uniformly for $y$ in the fixed compacet set, then
\begin{equation}
\label{eq:L2normalongshrinkingdiscsL1}
 \langle U^2, \phi \rangle \ll T^{\varepsilon} A \|\phi \|_1.
\end{equation}
\end{myprop}
Here we can take $\phi$ nonnegative such that $\chi_{A}(z) \leq \phi(z) \leq \chi_{2A}(z)$ 
where $\chi_{r}(z)$ is the characteristic function of the disc $D_{z_0}(r)$ centered at a fixed point $z_0$ having radius $r$, and \eqref{eq:L2normalongshrinkingdiscsL1} gives an upper bound on the $L^2$ norm of $U$ restricted to such a disc, namely
\begin{equation}
\label{eq:L2normalongshrinkingdiscs2}
\int_{D_{z_0}(A^{-1})} |U(z)|^2 \frac{dx dy}{y^2} \ll_{z_0} A^{-1} T^{\varepsilon}.
\end{equation}
This is curiously just as strong as what follows from Cauchy's inequality and the Sarnak-Watson bound \cite{SarnakBAMS} $\|U\|_4 \ll T^{\varepsilon}$ (conditional on Ramanujan).  Here \eqref{eq:L2normalongshrinkingdiscsL2} comes from the triple product approach while \eqref{eq:L2normalongshrinkingdiscsL1} arises from the Poincare series approach.

One may also wonder about bounds on average.  It is easy to produce the following strong average bound:
\begin{equation}
\sum_{T \leq t_j \leq T+1} \langle u_j^2, \phi \rangle \ll T \|\phi\|_1.
\end{equation}
The proof follows immediately upon using (13.8) of \cite{IwaniecSpectral} which says $\sum_{T \leq t_j \leq T+1} |u_j(z)|^2 \ll T$, uniformly for $z$ in some fixed compact set.

\begin{proof}
We begin with the Poincare series method.  By the computations in Section \ref{section:PoincareQUE}, we have
\begin{align}
\langle U^2, \phi \rangle &= \sum_{|m| \leq AT^{\varepsilon}} \langle U^2, P_m(z, \Phi_m) \rangle + O(T^{-100} \| \phi \|_1) \\
 &= \sum_{|m| \leq AT^{\varepsilon}} \int_0^{\infty} \int_0^{1} |U(x+iy)|^2 e(-mx) \Phi_m(y) \frac{dx dy}{y^2}.
\end{align}
By the unconditional horocycle bound $\int_0^{1} |U(x+iy)|^2 dx \ll T^{\varepsilon}$, uniform for $y$ in a fixed compact set (see \cite{GRS}, Theorem 5.1 (2)), we have
\begin{equation}
\langle U^2, \phi \rangle \ll T^{\varepsilon} \sum_{|m| \leq AT^{\varepsilon}}  \int_0^{\infty} \Big| \int_0^{1} \Phi(x+iy) e(-mx) dx \Big| \frac{dy}{y^2} \leq 3A T^{2\varepsilon} \| \phi \|_1,
\end{equation}
giving \eqref{eq:L2normalongshrinkingdiscsL1}.

Next we use the triple product approach.
As in the proof of Proposition \ref{prop:QUEshrinkingWatsonapproach}, we have
\begin{equation}
\langle U^2, \phi \rangle = \sum_j \langle U^2, u_j \rangle \langle u_j, \phi \rangle + \dots,
\end{equation}
with the dots indicating the constant eigenfunction and the Eisenstein series contributions.  We can truncate the spectral sum at $t_j \leq A T^{\varepsilon}$ with an error of size $O(T^{-100} \| \phi \|_1)$.  By Watson's formula \eqref{eq:Watsonformula}, for $t_j = o(T)$, 
\begin{equation}
\langle U^2, u_j \rangle = \theta_{U,j} T^{-1/2} t_j^{-1/2}  \sqrt{L(1/2, \mathrm{sym}^2 U \times u_j)} \sqrt{L(1/2, u_j)},
\end{equation}
where $\theta_{U,j}$ is a real number satisfying $|\theta_{U,j}| \ll T^{\varepsilon}$.  The conductor of $L(1/2, U \times U \times u_j) = L(1/2, \mathrm{sym}^2 U \times u_j) L(1/2, u_j)$ is, for $t_j = o(T)$, $T^4 t_j^4$.  The conductor of $L(1/2, u_j)$ is $t_j^2$, so the conductor of the degree $6$ factor is $T^4 t_j^2$.  

By H\"{o}lder's inequality (with exponents $4, 4, 2$), we obtain
\begin{multline}
\langle U^2, \phi \rangle \ll T^{-1/2 + \varepsilon} \Big(\sum_{t_j \leq AT^{\varepsilon}} t_j^{-2} L(1/2, u_j)^2 \Big)^{1/4} \Big(\sum_{t_j \leq AT^{\varepsilon}}  L(1/2, \mathrm{sym}^2 U \times u_j)^2 \Big)^{1/4}
\\
\times \Big( \sum_{t_j} |\langle u_j, \phi \rangle|^2 \Big)^{1/2} + \dots.
\end{multline}
By Bessel's inequality, $\sum_{t_j} |\langle u_j, \phi \rangle|^2 \leq \langle \phi, \phi \rangle$.
As noted earlier, \eqref{eq:secondmoment} holds, and in addition we claim that the spectral large sieve inequality proves
\begin{equation}
\label{eq:largesievebound}
\sum_{t_j \leq AT^{\varepsilon}} L(1/2, \mathrm{sym}^2 U \times u_j)^2 \ll A T^{2+\varepsilon},
\end{equation}
which taken together leads to \eqref{eq:L2normalongshrinkingdiscsL2}.
A case similar to \eqref{eq:largesievebound} (varying the level) appeared in \cite{LMY}, Proposition 6.3, so we omit a detailed explanation and instead give a sketch that displays the main ideas.  The rough idea of the proof is to apply the approximate functional equation to see
\begin{equation}
\sum_{t_j \leq AT^{\varepsilon}} L(1/2, \mathrm{sym}^2 U \times u_j)^2 \approx \sum_{t_j \leq AT^{\varepsilon}} \Big| \sum_{n \leq A T^{2+\varepsilon}} \frac{a_n \lambda_j(n)}{\sqrt{n}} \Big|^2,
\end{equation}
where $a_n = A(1,n)$ occur as Fourier coefficients of $\mathrm{sym}^2 U$.
Then the spectral large sieve inequality shows
\begin{equation}
\sum_{t_j \leq V} \Big| \sum_{n \leq N} a_n \lambda_j(n) \Big|^2 \ll (VN)^{\varepsilon}(V^2 +N) \sum_{n \leq N} |a_n|^2.
\end{equation}
In our application, we have $\sum_{n \leq N} \frac{|A(1,n)|^2}{n} \ll (NT)^{\varepsilon}$ by the convexity bound for Rankin-Selberg $L$-functions \cite{XiannanLi}.  
\end{proof}

\section{The geodesic QUE theorem for Eisenstein series}
Here we present the proof of our main result, Theorem \ref{thm:geodesicEisenstein}.  
We shall prove the following more precise version of Theorem \ref{thm:geodesicEisenstein} which has a power saving in the error term.
\begin{mytheo}
\label{thm:geodesicEisenstein2} 
Suppose $\psi$ is a fixed smooth, compactly-supported function on $\mathbb{R}^+$.  Then
 \begin{equation}
\label{eq:EisensteinGeodesicIntegral}
 \int_0^{\infty} \psi(y) |E(iy, 1/2 + iT)|^2 \frac{dy}{y} = 2\langle |E(z, 1/2 + iT)|^2, E(z, y \psi(y))\rangle + a+b(T) + c + O_{\delta}(T^{-\delta}),
 \end{equation}
where $a$ is defined by \eqref{eq:adef}, $b(T) = O(1)$ is defined by \eqref{eq:bdef}, $c$ is defined by \eqref{eq:cdef} and $\delta < 1/33$.
\end{mytheo}
Luo and Sarnak \cite{LuoSarnakQUE} evaluated asymptotically the inner product appearing on the right hand side of \eqref{eq:EisensteinGeodesicIntegral}, showing (see their Proposition 2.2)
\begin{equation}
\label{eq:LuoSarnakQUEresult}
 \langle |E(z,1/2 + iT)|^2, E(z, y \psi(y)) \rangle = \frac{3}{\pi} \log(1/4 + T^2) \int_0^{\infty} \psi(y) \frac{dy}{y} + O\Big(\frac{\log{T}}{\log \log T}\Big).
\end{equation}
Indeed, they show (see their Proposition 2.3) that $\frac{3}{\pi} \log(1/4 + T^2)$ is the average size of $|E(z, 1/2 + iT)|^2$ restricted to any fixed compact Jordan measurable subset of $\Gamma \backslash \mh$ having positive measure.
Thus Theorem \ref{thm:geodesicEisenstein} follows from Theorem \ref{thm:geodesicEisenstein2}.  See also Proposition \ref{prop:thinEisensteinQUE} or the more precise version appearing in Section \ref{section:QUEEIsensteinsubsection}.

The proof is long and we have presented some auxiliary results in Sections \ref{section:maintermstwistedfourthmoment} and \ref{section:SCSunconditional}, so here we map out the strategy.  The first step, appearing in Section \ref{section:reductiontofourthmoment}, is to use harmonic analysis on the positive reals to relate the integral of the Eisenstein series to a shifted fourth moment of the Riemann zeta function, similarly to the method used in Section \ref{section:derivation}.  The main difference here is that this requires some regularization.
In Section \ref{section:EisensteingeodesicQUEMainterm}, working on the assumption that the asymptotic formula for the shifted fourth moment with large shifts takes the same form as for small shifts, we relate the main term in the fourth moment to an inner product of the Eisenstein series with an incomplete Eisenstein series, again similarly to the approach in Section \ref{section:derivation}; this appears as Theorem \ref{thm:J(s)squaredintegral}.  
This asymptotic for the fourth moment of zeta requires, as one ingredient, the asymptotic behavior of a shifted divisor sum.  We present this part separately as Theorem \ref{thm:SCS}; in fact all of Section \ref{section:SCSunconditional} is devoted to the proof of this estimate.  In Section \ref{section:maintermstwistedfourthmoment}, which is also self-contained, we show (loosely speaking) that if one has an asymptotic for a shifted divisor sum with a power saving, then one can evaluate the fourth moment of zeta including all the lower-order main terms.  This amounts to a calculation of various main terms and is a generalization of work of Hughes-Young \cite{HughesYoung} which assumed that the shifts are all small.  Finally, in Section \ref{section:fourthmomentofzeta} we collect these auxiliary results and show that indeed the shifted fourth moment with large shifts does take the expected form, thus completing the proof of Theorem \ref{thm:geodesicEisenstein2}.

\subsection{Reduction to a fourth moment}
\label{section:reductiontofourthmoment}
The Fourier expansion for Eisenstein series, i.e., \eqref{eq:UFourier}, may be alternatively expressed in the symmetric form $\theta(s) E(z,s) = \theta(s)y^{s} + \theta(1-s) y^{1-s} + \dots$.  Analogously to the definition of the Hardy $Z$-function, set $E^*(z,s) = \frac{\theta(s)}{|\theta(s)|} E(z,s)$.   
The Fourier expansion of $E^*$ takes the form
\begin{equation}
E^*(z,1/2+iT) = c_0^*(y) + \rho^*(1) \sum_{n \neq 0} \frac{\tau_{iT}(n)}{\sqrt{|n|}} e(nx) V_T(2\pi |n| y),
\end{equation}
where now $c_0^*(y) = \mu y^{1/2 + iT} + \overline{\mu} y^{1/2-iT}$, with
\begin{equation}
 \label{eq:Adef}
\mu = \frac{\theta(1/2+iT)}{|\theta(1/2+iT)|}.
\end{equation}
Furthermore, by comparison with \eqref{eq:rho1def}, we have
\begin{equation}
\label{eq:rho*}
 \rho^*(1)  = |\rho(1)| = (2/\pi)^{1/2} |\theta(1/2 + iT)|^{-1},
\end{equation}
 and we see that $E^*(z, 1/2 + iT)$ is real-valued.

Our main goal in this subsection is to prove the following
\begin{myprop}
\label{prop:reductiontofourthmoment}
Let $I(T;\psi)$ denote the left hand side of \eqref{eq:EisensteinGeodesicIntegral}.  Then
\begin{equation}
\label{eq:ITpsiasymptotic}
 I(T;\psi) = 4 \frac{\rho^*(1)^2}{\cosh(\pi T)} \frac{1}{2\pi i} \int_{(0)} \widetilde{\psi}(-v) M(v) dv  +  b(T)
 + O(T^{-1/12 + \varepsilon}),
\end{equation}
where
\begin{equation}
 M(v) = \frac{\cosh(\pi T)}{2\pi i} \int_{(0)} L(1/2 + s + v, E_T) L(1/2 - s , E_T) \gamma_{V_T}(1/2 + s + v) \gamma_{V_T}(1/2 - s) ds.
\end{equation}
and with $\mu$ defined above by \eqref{eq:Adef}, we set
\begin{equation}
 \label{eq:bdef}
b(T) = 2 \int_0^{\infty} \psi(y) (y + y^{-1} + \mu^2 + \overline{\mu}^2) \frac{dy}{y}.
\end{equation}
\end{myprop}
Remark. We can interpret $b(T)$ alternatively via
\begin{equation}
 b(T) = \int_0^{\infty} (c_0^*(y) + c_0^*(1/y))^2 \psi(y) \frac{dy}{y} + O(T^{-100}).
\end{equation}
Note also that $M(v) = M(-v)$, which is a symmetry corresponding to the fact that the integrands in $I(T, \psi)$ and $b(T)$ are invariant under $y \rightarrow 1/y$.

\begin{proof}
As a simple method of regularization, write
\begin{equation}
 I(T;\psi) = \int_0^{\infty} E_T^*(iy)^2 \psi(y) \frac{dy}{y} = \int (E_T^* - c_0^*)^2 + 2 \int c_0^*(E_T^* - c_0^*) + \int (c_0^*)^2,
\end{equation}
which we denote as $I_1 + I_2 + I_3$, respectively.  We compute the three terms in turn.

A very short calculation shows
\begin{equation}
 I_3 = 2 \int_0^{\infty} \psi(y) y \frac{dy}{y} + O(T^{-100}),
\end{equation}
which gives one of the four terms making up $b(T)$.

Using the Fourier expansion of $E_T^*$, we have
\begin{equation}
 I_2 = 4 \rho^*(1) \sum_{n \geq 1} \frac{\tau_{iT}(n)}{\sqrt{n}} \int_0^{\infty} V_T(2 \pi n y) (\mu y^{1/2 +iT} + \overline{\mu} y^{1/2 -iT})
\psi(y) \frac{dy}{y}.
\end{equation}
Applying the Mellin inversion formula to $\psi$, changing variables $y \rightarrow y/n$, and reversing orders of integration and summation, we obtain (the reader may wish to recall the definitions \eqref{eq:gammaVformula} and \eqref{eq:EisensteinSeriesLfunction})
\begin{multline}
 I_2 = \frac{4 \rho^*(1)}{2\pi i} \int_{(1)} \widetilde{\psi}(-s) \Big( \mu L(1 + s + iT, E_T) \gamma_{V_T}(1 + s + iT) 
\\
+ \overline{\mu} L(1 + s - iT, E_T) \gamma_{V_T}(1 + s - iT) \Big) ds.
\end{multline}
Next we move the contour of integration to $\text{Re}(s) = -1/2$, crossing poles at $s= 0$, and $s= \pm 2iT$.  The residues at $s = \pm 2 iT$ contribute $O(T^{-100})$ to $I_2$ since $\widetilde{\psi}(-s) \ll (1 + |s|)^{-200}$.  It will be useful to record that
\begin{equation}
\label{eq:EisensteinLfunctionResidue}
 \text{Res}_{s=0} L(1 + s + iT, E_T) \gamma_{V_T}(1 + s + iT) = \half \sqrt{\frac{\pi}{2}} \theta(1/2 + iT).
\end{equation}
We also have that the size of the new contour integral is
\begin{equation}
\label{eq:residueboundforI2}
 \ll T^{\varepsilon} \int_{|t| \ll T^{\varepsilon}} |\zeta(1/2 + it) \zeta(1/2 + it +2iT)| T^{-1/4} dt + O(T^{-100}),
\end{equation}
and using Weyl's bound shows this term is $O(T^{-1/12 + \varepsilon})$.  Gathering these estimates, and using \eqref{eq:rho*}, we have
\begin{equation}
 I_2 = 2 \int_0^{\infty} \psi(y) (\mu^2  + \overline{\mu}^2 )  \frac{dy}{y} + O(T^{-1/12 + \varepsilon}),
\end{equation}
which contributes two of the four terms making up $b(T)$. 

Finally we examine $I_1$.  By Lemma \ref{lemma:Mellin} (the Mellin convolution theorem) and \eqref{eq:completedLFunction2},
\begin{multline}
 I_1 = \frac{4 \rho^*(1)^2}{(2\pi i)^2} \int_{(2)} \widetilde{\psi}(-v)  
\\
 \int_{(1)} L(1/2 + s, E_T) L(1/2 + v-s, E_T) \gamma_{V_T}(1/2 + s) \gamma_{V_T}(1/2 + v -s) ds dv.
\end{multline}
As a first step, we move the $s$-integral to the line $\text{Re}(s) = 0$, crossing poles at $s=1/2 \pm iT$ (and only those, since $\text{Re}(v) = 2$ at this time).  We defer treatment of these residues for a moment.  For the new integral, we reverse the orders of integration and move the $v$-integral to the $0$-line, crossing poles at $v = 1/2 + s \pm iT$.  The new double integral equals the integral main term in Proposition \ref{prop:reductiontofourthmoment} after changing variables $s \rightarrow -s$.  The residues of these poles at $v = 1/2 + s \pm iT$ contribute $O(T^{-1/12 + \varepsilon})$ to $I_1$, by a calculation very similar to \eqref{eq:residueboundforI2}.

By \eqref{eq:EisensteinLfunctionResidue}, the residues at $s = 1/2 \pm iT$ give
\begin{equation}
 \frac{2 \rho^*(1)^2}{2\pi i} \sqrt{\frac{\pi}{2}} \sum_{\pm}  \theta(1/2 \pm iT) \int_{(2)} \widetilde{\psi}(-v)  L(v\mp iT, E_T) \gamma_{V_T}(v\mp iT)  dv.
\end{equation}
We shift this contour to $\text{Re}(v) = 1/2$, crossing poles at $v=1$ and $v = 1 \pm 2iT$.  The residues with $v = 1 \pm 2iT$ are very small by the rapid deay of $\widetilde{\psi}$, and using \eqref{eq:EisensteinLfunctionResidue} again as well as \eqref{eq:rho*}, the residues at $v=1$ give
\begin{equation}
  \rho^*(1)^2 \frac{\pi}{2}  \sum_{\pm} \theta(1/2 \pm iT) \theta(1/2 \mp iT) \widetilde{\psi}(-1) = 2 \int_0^{\infty} \psi(y) y^{-1} \frac{dy}{y},
\end{equation}
which is the remaining term of $b(T)$.  The new integral along $\text{Re}(v) = 1/2$ is again $O(T^{-1/12 + \varepsilon})$ by Weyl's bound.
 
We have accounted for all the terms of \eqref{eq:ITpsiasymptotic}, so the proof is complete.
\end{proof}

\subsection{The main term}
\label{section:EisensteingeodesicQUEMainterm}
Recall \eqref{eq:rho1squared}, that is, $\rho^*(1)^2 = \cosh( \pi T) T^{o(1)}$, and let
\begin{equation}
\label{eq:IJdef}
 %\frac{1}{2 \pi i} \int_{(0)} J(s) J(-s) ds 
I_J(T;\psi)
= 4 \frac{\rho^*(1)^2}{\cosh( \pi T)} \frac{1}{2\pi i} \int_{(0)} \widetilde{\psi}(-v) M(v) dv.
\end{equation}
Since $\widetilde{\psi}(-v)$ has rapid decay, understanding $I_J(T;\psi)$ (and hence $I(T;\psi)$) is, to first approximation, the same as understanding $M(v)$ for a fixed $v$.  
\begin{mytheo}
\label{thm:J(s)squaredintegral}
 We have the asymptotic
\begin{equation}
\label{eq:IJasymptotic}
I_J(T, \psi) =  %\frac{1}{2 \pi i} \int_{(0)} J(s) J(-s) ds = 
2 \langle |E(z, 1/2 + iT)|^2, E(z, y \psi(y)) \rangle + a+c + O(T^{-1/33 + \varepsilon}),
\end{equation}
where
\begin{equation}
 \label{eq:adef}
 a = -2 \int_0^{\infty} (y + y^{-1}) \psi(y) \frac{dy}{y},
\end{equation}
and
where with $f(z) = y^{1/2} |\eta(z)|^2$ (so in particular this $f$ is $\Gamma$-invariant), we define
\begin{equation}
\label{eq:cdef}
 c = -4 \int_0^{\infty} \psi(y) (y + y^{-1} + \tfrac{3}{\pi} \log f(iy)) \frac{dy}{y}.
\end{equation}
\end{mytheo}
We deduce Theorem \ref{thm:geodesicEisenstein2} from Proposition \ref{prop:reductiontofourthmoment} and Theorem \ref{thm:J(s)squaredintegral}, and Theorem \ref{thm:geodesicEisenstein} follows from Theorem \ref{thm:geodesicEisenstein2} via \eqref{eq:LuoSarnakQUEresult}.  See also Section \ref{section:QUEEIsensteinsubsection} for a more precise form of the main term (compared to \eqref{eq:LuoSarnakQUEresult}).

We begin by indicating how this main term emerges by a rigorous analog of the work in Section \ref{section:derivation}.
First observe that $M(v)$ is a weighted shifted fourth moment of zeta.  That is, it takes the form
\begin{equation}
\label{eq:fourthmoment}
M(v)= \frac{1}{2\pi} \intR \zeta(1/2 + \alpha + it) \zeta(1/2 + \beta + it) \zeta(1/2 + \gamma - it) \zeta(1/2 + \delta - it) w(t) dt,
\end{equation}
where 
\begin{gather}
\label{eq:alphabetagammadeltadefinition}
 \alpha = v + iT, \quad \beta = v- iT, \quad \gamma =  iT, \quad \delta =  - iT,
\end{gather}
$\text{Re}(v) = 0$, and we may assume $\text{Im}(v) \ll T^{\varepsilon}$.  Furthermore, via \eqref{eq:gammaVformula},
\begin{equation}
\label{eq:gammaweightfunction}
 w(t) = \frac{\cosh(\pi T)}{ 2^{3} %\pi^{-\alpha-\beta-\gamma-\delta}
\pi^{\frac{\alpha+\beta+\gamma+\delta}{2}} }
\Gamma\Big(\frac{\half + \alpha + it}{2}\Big) \Gamma\Big(\frac{\half + \beta+ it}{2}\Big) \Gamma\Big(\frac{\half + \gamma - it}{2}\Big) \Gamma\Big(\frac{\half + \delta - it}{2}\Big).
\end{equation}
% Motohashi \cite{Motohashi} proved an exact formula for shifted moments with a Gaussian weight, but it is not immediately obvious that his result leads to an asymptotic formula here; it is necessary to estimate various integral transforms  to bound an error term.  In \cite{HughesYoung}, we also developed an asymptotic formula for shifted moments but we assumed that the shifts are close to $0$.  No matter the approach taken, it seems that there is some extra work necessary to treat these large shifts as well as the particular weight function.  However, the main term is already predicted ahead of time so we shall first work with this main term, and later bound the error terms required to prove that this is the asymptotic.

As stated in Theorem \ref{thm:HughesYounglargeshifts} below, (which is the large shift analog of Theorem 1.1 of \cite{HughesYoung}), the main term\footnote{Of course we need to show that the error term is indeed smaller than this ``main term''.} of \eqref{eq:fourthmoment} is the sum of six terms. In this paper we shall extend this result to hold for shifts of the form \eqref{eq:alphabetagammadeltadefinition} and $T$ very large; see Theorem \ref{thm:fourthmomentlargeshifts} below.  The full main term in \eqref{eq:Ihkasymptotic} is holomorphic in terms of the shift parameters as long as they are all in the strip $-1/2 < \text{Re}(z) < 1/2$, even though each of the terms in the sum has poles.
One of the six terms is
\begin{equation}
\label{eq:onemainterm}
\mathcal{M}_0(v) :=  \frac{\zeta(1+ \alpha + \gamma) \zeta(1 + \alpha + \delta) \zeta(1 + \beta + \gamma) \zeta(1 + \beta + \delta)}{\zeta(2 + \alpha + \beta + \gamma + \delta)} \frac{1}{2\pi} \intR w(t) dt.
\end{equation}
To match more closely with the arguments in Section \ref{section:derivation}, it is helpful to know that the term \eqref{eq:onemainterm} 
 arises from the ``first part'' of the approximate functional equation and taking the diagonal analogously to \eqref{eq:I2alphadiagonalterm}.  One can see this by following the proof of Theorem \ref{thm:HughesYounglargeshifts}, specifically by looking at the residue at $s=0$ of \eqref{eq:IDhk1integral}.
In this way, we have for $\text{Re}(v) > 0$,
\begin{equation}
\label{eq:M0formula}
 \mathcal{M}_0(v) =  \sum_{n=1}^{\infty} \frac{\tau_{iT}(n)^2}{n^{1+v}} \frac{\cosh(\pi T)}{2 \pi i} \int_{(0)} \gamma_{V_T}(1/2 + s + v) \gamma_{V_T}(1/2-s) ds,
\end{equation}
and the formula extends by meromorphic continuation.  Note the similarity to \eqref{eq:I2alphadiagonalterm}.  
To work with only $\mathcal{M}_0(v)$ we cannot integrate along the line $\text{Re}(v) = 0$ because it passes through poles.  Therefore in \eqref{eq:IJdef} we first shift the contour slightly to the right to $\text{Re}(v) = \varepsilon$ with $\varepsilon < 1/2$, and then
insert \eqref{eq:M0formula} into \eqref{eq:IJdef}.  Using \eqref{eq:Zdef} and \eqref{eq:gammaVsquaredDef}, we then obtain
\begin{equation}
 4 \frac{\rho^*(1)^2}{\cosh( \pi T)} \frac{1}{2 \pi i} \int_{(\varepsilon)} \widetilde{\psi}(-v) \mathcal{M}_0(v) dv 
= 4 |\rho(1)|^2 \frac{1}{2 \pi i} \int_{(\varepsilon)} \widetilde{\psi}(-v) Z(1+v) \gamma_{V_T^2}(1+v) dv.
\end{equation}
By \eqref{eq:UsquaredinnerproductincompleteEisenstein}, with $h(y) = y \psi(y)$ so that $\widetilde{h}(-1-s) = \widetilde{\psi}(-s)$, we can recognize this as the pretty formula
\begin{equation}
\label{eq:M0vintegralintermsofEisenstein}
  \frac{4\rho^*(1)^2}{\cosh(\pi T)} \frac{1}{2 \pi i} \int_{(1)} \widetilde{\psi}(-v) \mathcal{M}_0(v) dv = 2 \langle |E(z, 1/2 + iT)|^2, E(z,y \psi(y)) \rangle - 2 \int_0^{\infty} c_0^*(y)^2 \psi(y) \frac{dy}{y}.
\end{equation}
This second integral we can write more symmetrically, using the assumption $\psi(y) = \psi(y^{-1})$, as
\begin{equation}
- \int_0^{\infty} (c_0^*(y)^2 + c_0^*(y^{-1})^2) \psi(y) \frac{dy}{y} = -2 \int_0^{\infty} (y + y^{-1}) \psi(y) \frac{dy}{y} + O(T^{-100}).
\end{equation}
One immediately reads off the inner product appearing in \eqref{eq:IJasymptotic} as well as the constant $a$ defined by \eqref{eq:adef}.

Next we shall analyze the other five terms making up the main term of \eqref{eq:fourthmoment}, which are given in the general form by Theorem \ref{thm:HughesYounglargeshifts}.  We need to show that these give the constant $c$ up to a satisfactory error term.
The underlying principle is that the original fourth moment has symmetries arising from applying the functional equation to one of $\zeta(1/2 + \alpha + it)$ or $\zeta(1/2 + \beta + it)$, and one of $\zeta(1/2 + \gamma - it)$, $\zeta(1/2 + \delta - it)$ (or both).  There are five such symmetries.
If we apply the functional equation to $\zeta(1/2 + \alpha + it)$ and to $\zeta(1/2 + \gamma - it)$, for instance, then 
this amounts to switching $\alpha$ with $-\gamma$ and multiplying by $X_{\alpha,\gamma,t}$, in the notation of Section \ref{section:maintermstwistedfourthmoment}.  However, our calculations are simplified here by noting that the weight function $w(t)$ also depends on the shifts and in fact $w(t)$ times the product of four zetas gives the completed zeta functions which are then \emph{invariant} under changes of variable $\alpha \leftrightarrow - \gamma$, etc.  That is, the other five terms are obtained by changing variables (1) $\alpha \leftrightarrow -\gamma$ (meaning $\alpha$ is replaced by $-\gamma$ and $\gamma$ is replaced by $-\alpha$) or (2) $\alpha \leftrightarrow -\delta$ or (3) $\beta \leftrightarrow -\gamma$ or (4) $\beta \leftrightarrow -\delta$ or (5) $\alpha \leftrightarrow -\gamma$ and $\beta \leftrightarrow -\delta$.

We need to explicitly evaluate the $t$-integral in terms of $\alpha, \beta, \gamma, \delta$.
By (6.412) of \cite{GR}, 
\begin{equation}
\label{eq:wintegralevaluation}
 \frac{1}{2\pi} \intR w(t) dt = \frac{\cosh(\pi T)}{2^{2} 
\pi^{\frac{\alpha+\beta+\gamma+\delta}{2}} }  \frac{\Gamma(\frac{1+ \alpha + \gamma}{2})\Gamma(\frac{1 + \alpha + \delta}{2}) \Gamma(\frac{1 + \beta + \gamma}{2}) \Gamma(\frac{1 + \beta + \delta}{2})}{\Gamma(\frac{2 + \alpha + \beta + \gamma + \delta}{2})}.
\end{equation}
Letting $\Lambda(s) = \pi^{-s/2} \Gamma(s/2) \zeta(s)$, and inserting \eqref{eq:wintegralevaluation} into \eqref{eq:onemainterm}, we derive
\begin{equation}
\label{eq:M0mainterm}
\frac{\mathcal{M}_0(v)}{\cosh( \pi T)} = \frac{\pi}{4}  \frac{\Lambda(1+ \alpha + \gamma) \Lambda(1 + \alpha + \delta) \Lambda(1 + \beta + \gamma) \Lambda(1 + \beta + \delta)}{\Lambda(2 + \alpha + \beta + \gamma + \delta)}.
\end{equation}

Thus we arrive at
\begin{equation}
\label{eq:Jsquaredasymptotic}
%\frac{1}{2 \pi i} \int_{(0)} J(s) J(-s) ds = 
I_J(T;\psi) = 
4   \frac{1}{2\pi i} \int_{(\varepsilon)} \widetilde{\psi}(-v) \frac{\rho^*(1)^2}{\cosh(\pi T)} \sum_{k=0}^{5} \mathcal{M}_k(v) dv + \mathcal{E}(T)
\end{equation}
where $\mathcal{E}(T)$ is an error term that we shall estimate with Theorem \ref{thm:fourthmomentlargeshifts} (showing $\mathcal{E}(T) \ll T^{-1/33 + \varepsilon}$), and $\sum_{k=0}^{5}\frac{4}{\pi \cosh(\pi T)}  \mathcal{M}_k(v)$ is tediously calculated to be (applying the changes of variables $(1)$-$(5)$ in terms of the shift parameters $\alpha, \beta, \gamma, \delta$, and then substituting back into $v$ and $T$ with \eqref{eq:alphabetagammadeltadefinition})
\begin{multline}
\label{eq:Mvmainterms}
  \frac{\Lambda(1+ v + 2iT) \Lambda(1 +v)^2  \Lambda(1 + v-2iT)}{\Lambda(2 + 2v)} + \frac{\Lambda(1- v - 2iT) \Lambda(1 -2iT)^2  \Lambda(1 + v-2iT)}{\Lambda(2 -4iT)}
 \\
+\frac{\Lambda(1+ 2iT) \Lambda(1 +v)\Lambda(1 -v)  \Lambda(1 -2iT)}{\Lambda(2)}
 + \frac{\Lambda(1+ 2iT) \Lambda(1 -2iT)  \Lambda(1 + v)\Lambda(1 - v)}{\Lambda(2)}
\\
+\frac{\Lambda(1+ v+2iT) \Lambda(1 +2iT)^2 \Lambda(1 -v+2iT)}{\Lambda(2+4iT)}
+ \frac{\Lambda(1- v-2iT) \Lambda(1 -v)^2 \Lambda(1 -v+2iT)}{\Lambda(2-2v)}.
\end{multline}
Here the displayed terms are respective to the index $k$ in $\mathcal{M}_k(v)$ which in turn is with respect to the labelling of the changes of variable in the paragraph immediately preceding \eqref{eq:wintegralevaluation}.  As a consistency check, we remark that \eqref{eq:Mvmainterms} is symmetric under $v \rightarrow -v$, as it should be, recalling that $M(v) = M(-v)$.  It is holomorphic for $-1/2 < \text{Re}(v) < 1/2$, and Stirling's formula shows that $\mathcal{M}_k(v)$ is bounded by a polynomial in $v$ and $T$.  

Let $\mathcal{R}(T)$ denote the integral on the right hand side of \eqref{eq:Jsquaredasymptotic}.  We will presently show
\begin{equation}
\label{eq:RTasymptotic}
 \mathcal{R}(T) = 2 \langle |E(z, 1/2 + iT)|^2, E(z, y \psi(y)) \rangle + a+ c + O(T^{-1/6 + \varepsilon}).
\end{equation}
Write $\mathcal{R}(T) = \sum_{k=0}^{5} \mathcal{R}_k(T)$ according to the sum in \eqref{eq:Jsquaredasymptotic}.
We have already seen that $\mathcal{R}_0(T)$ gives the inner product and the constant $a$ on the right hand side of \eqref{eq:RTasymptotic}.  We will presently show that $\mathcal{R}_k(T)$ for $k=1,4, 5$ is bounded by $O(T^{-1/6 + \varepsilon})$, while $\mathcal{R}_2(T) + \mathcal{R}_3(T)$ give the constant $c$ defined by \eqref{eq:cdef}.

By a trivial estimation, for $k=1, 4$ we have $\mathcal{R}_k(T) \ll T^{-1/2 + \varepsilon}$. 
This bound arises from the fact that $\Lambda(2 \pm 4 iT) \gg T^{1/2} \exp(-\pi T)$.  For $\mathcal{R}_5(T)$, we move the contour to $\text{Re}(v) = \half$, and use $\Lambda(\half + i t) \ll t^{-1/12 + \varepsilon}$ by Weyl's bound $\zeta(1/2 + it) \ll t^{1/6 + \varepsilon}$.  In this way we have
\begin{equation}
\int_{(\varepsilon)} \widetilde{\psi}(-v) \frac{\rho^*(1)^2}{\cosh(\pi T)} \mathcal{M}_5(v) dv   = \int_{(\half)} \widetilde{\psi}(-v) \frac{\rho^*(1)^2}{\cosh(\pi T)} \mathcal{M}_5(v) dv \ll T^{-1/6 + \varepsilon}.
\end{equation}
For $\mathcal{R}_2(T) = \mathcal{R}_3(T)$, by direct substitution we have
\begin{equation}
\label{eq:R2R3initialform}
 \mathcal{R}_2(T) + \mathcal{R}_3(T) = 2\pi |\rho^*(1)|^2 \frac{\Lambda(1+ 2iT) \Lambda(1-2iT)}{\Lambda(2)} \frac{1}{2 \pi i} \int_{(\varepsilon)} \widetilde{\psi}(-v) \Lambda(1+v)\Lambda(1-v) dv.
\end{equation}
We work with this inner integral.
By the functional equation,
\begin{equation}
\label{eq:M2integral}
 \frac{1}{2\pi i} \int_{(\varepsilon)} \widetilde{\psi}(-v) \Lambda(1+v) \Lambda(1-v) dv = \frac{1}{2\pi i} \int_{(\varepsilon)} \widetilde{\psi}(-v) \Lambda(1+v) \Lambda(v) dv,
\end{equation}
which equals
\begin{multline}
 \int_0^{\infty} \psi(y) \frac{1}{2\pi i} \int_{(\varepsilon)} \Lambda(1+v) \Lambda(v) y^{-v} dv \frac{dy}{y}
\\
 = \int_0^{\infty} \psi(y) \Big(- \Lambda(2) y^{-1} + \frac{1}{2\pi i} \int_{(2)} \Lambda(1+v) \Lambda(v) y^{-v} dv \Big) \frac{dy}{y}.
\end{multline}
Then we write $\Lambda(1 + v) \Lambda(v) = \Gamma_{\mr}(1+v) \Gamma_{\mr}(v) \zeta(1+v) \zeta(v)$,  use the identity $\Gamma_{\mr}(1+v) \Gamma_{\mr}(v) = \Gamma_{\mc}(v) = 2 (2\pi)^{-v} \Gamma(v)$, and reverse the orders of summation and integration, so we see that \eqref{eq:M2integral} equals
\begin{equation}
 \int_0^{\infty} \psi(y) \big(-\Lambda(2) y^{-1} + 2\sum_{m,n} m^{-1}  \exp(-2\pi mn y) \big) \frac{dy}{y}.
\end{equation}
The inner sum equals $2 \sum_{n \geq 1} \sigma_{-1}(n) \exp(-2\pi ny)$ which is $D(iy)$ defined by (22.64) of \cite{IK}.  This function $D(iy)$ makes up the tail of the Fourier expansion in the constant term of the Laurent expansion of the Eisenstein series at $s=1$.  Inserting this calculation into \eqref{eq:R2R3initialform}, we have
\begin{equation}
\label{eq:M2M3simplified}
\mathcal{R}_2(T) + \mathcal{R}_3(T) = 2\pi \rho^*(1)^2 \frac{\Lambda(1+2iT) \Lambda(1-2iT)}{\Lambda(2)} \int_0^{\infty} \psi(y) (-\Lambda(2) y^{-1} + D(iy) ) \frac{dy}{y}.
\end{equation}
By (22.68) of \cite{IK}, $D(iy) = -\frac{\pi}{6} y + \half \log y -\half \log f(iy)$ where $f(z) = y^{1/2} |\eta(z)|^2$ is $\Gamma$-invariant.  Also since $\log$ is odd under $y \rightarrow y^{-1}$, and we assumed $\psi$ is even, \eqref{eq:M2M3simplified} then becomes
\begin{equation}
-2\pi \rho^*(1)^2 \Lambda(1+2iT) \Lambda(1-2iT) \int_0^{\infty} \psi(y) (y^{-1} + y + \frac{1}{2 \Lambda(2)} \log f(iy)) \frac{dy}{y}.
\end{equation}
Using \eqref{eq:rho*} and $\theta(1/2 + iT) = \Lambda(1 + 2iT)$, we have
\begin{equation}
\mathcal{R}_2(T) + \mathcal{R}_3(T)  =  -4 \int_0^{\infty} \psi(y) (y + y^{-1} + \frac{3}{\pi} \log f(iy)) \frac{dy}{y},
\end{equation}
which is the constant $c$ defined by \eqref{eq:cdef}.  This finishes the proof of \eqref{eq:RTasymptotic}.

\subsection{The fourth moment of the zeta function with large shifts}
\label{section:fourthmomentofzeta}
% As shown in Section \ref{section:EisensteingeodesicQUEMainterm}, deriving an asymptotic formula for $I(T;\psi)$ is reduced to finding an asymptotic for $M(v)$ defined by \eqref{eq:fourthmoment}, with parameters defined by \eqref{eq:alphabetagammadeltadefinition} and
% \eqref{eq:gammaweightfunction}.  
\begin{mytheo}
\label{thm:fourthmomentlargeshifts}
Let $M(v)$ be defined by \eqref{eq:fourthmoment}, with shifts as in \eqref{eq:alphabetagammadeltadefinition}.
Then
\begin{equation}
 M(v) = M.T. + O(T^{-\frac{1}{33} + \varepsilon}),
\end{equation}
where $M.T. = \sum_{k=0}^{5} \mathcal{M}_k(v)$  denotes the main term defined by \eqref{eq:M0mainterm} and following discussion. 
\end{mytheo}
% The main term here is consistent with Theorem 1.1 of \cite{HughesYoung} described in further detail in Section \ref{section:EisensteingeodesicQUEMainterm} above.

We begin the analysis of $M(v)$ by describing the properties of the weight function $w(t)$ defined by \eqref{eq:gammaweightfunction}.  By Stirling's approximation,
\begin{equation}
 w(t) \ll_v \exp(\tfrac{\pi}{4} Q'(t,v, T)) (1 + |t-T|)^{-1/2} (1 + |t+T|)^{-1/2},
\end{equation}
where the implied constant depends polynomially on $v$, and
\begin{equation}
 Q'(t,v, T) = 4T - |v+ t + T| - |v+t-T| - |t-T| - |t+T|.
\end{equation}
We claim that $Q'(t,v,T) \leq 0$ for all $t, v, T \in \mr$.  For a given $t$ and $T$, the maximum of $Q'$ in terms of $v$ must occur at $v=-t-T$ or $v=-t+T$.  At $v= - t - T$ (which it suffices to check by symmetry), $Q'$ specializes as
\begin{equation}
 2T - |t-T| - |t+T|,
\end{equation}
which happens to equal $Q(t,T/2)$ where recall we originally encountered $Q(t,T)$ in \eqref{eq:Qdef}.  We know $Q(t,T/2) \leq 0$ for all $t,T$ and $Q(t,T/2) \leq 2(T-|t|)$ for $|t| \geq T$.  Thus for any $v \in \mr$ and $|t| \geq T$, we have
\begin{equation}
\label{eq:Q'bound}
 Q'(t,v,T) \leq 2(T - |t|).
\end{equation}

By logarithmic differentiation and the asymptotic expansion of $\frac{\Gamma'}{\Gamma}(z) = \log(z) + \frac{a_1}{z} + \dots$, one can derive the bounds
\begin{equation}
\label{eq:wbound}
 w^{(j)}(t) \ll_{v} (1 + |t-T|)^{-1/2} (1 + |t+T|)^{-1/2} \Big(\frac{1}{1 + |t-T|} + \frac{1}{1 + |t+T|} \Big)^j.
\end{equation}

Let $M(v) = M_{+}(v) + M_{-}(v)$, according to $t \geq 0$ or $t \leq 0$.  By symmetry, we focus on $M_{+}(v)$.
Next we apply a smooth partition of unity supported on sets of the form $T- 2\Delta \leq t \leq T - \Delta$ with $1 \leq \Delta \leq T/2$, as well as the region $t \geq T-1$. Accordingly, write $M_{+}(v) = \sum_{\Delta} M_{\Delta}(v) + (T^{-1/6 + \varepsilon})$, where we now explain the origin of this error term.  It 
accounts for $0 \leq t \leq 1$, say, and $t \geq T-1$.  For $t = O(1)$, the contribution to $M_{+}(v)$ is $\ll (T^{1/6 + \varepsilon})^4 T^{-1} \ll T^{-1/3 + \varepsilon}$, by Weyl's bound and the fact that $w(t) \ll T^{-1}$ for such $t$.  For $T-1 \leq t \leq T + T^{\varepsilon}$, the error is $\ll (T^{1/6 + \varepsilon})^2 T^{-1/2} \ll T^{-1/6 + \varepsilon}$, by the same type of reasoning, except here $w(t) \ll T^{-1/2}$, and only two of the zeta functions are evaluated at height $T$, the other two having height $T^{\varepsilon}$.  The exponential decay of $w$ overwhelms the polynomial growth of the zeta functions for $t \geq T + T^{\varepsilon}$, using \eqref{eq:Q'bound}.

We need to treat $M_{\Delta}(v)$ in two different ways depending on the size of $\Delta$.
\begin{mylemma}
\label{lemma:MDeltaUB}
 We have
\begin{equation}
\label{eq:MDeltaUB}
 M_{\Delta}(v) \ll T^{-1/6 + \varepsilon} + \Delta^{1/2} T^{-1/2 + \varepsilon}.
\end{equation}
\end{mylemma}
This bound is satisfactory for $\Delta \ll T^{1-\delta}$.  For $\Delta$ large, we need to extract the main term, and to this end we have
\begin{mylemma}
\label{lemma:MDeltaAsymptotic}
We have
\begin{equation}
\label{eq:MDeltaAsymptotic}
 M_{\Delta}(v) = M.T.^{(\Delta)} + O(T^{\varepsilon} (T^{25/12} \Delta^{-9/4} + T^{13/24} \Delta^{-5/8}).
\end{equation}
Here the main term is a sum of the form $\sum_{k=0}^{5} \mathcal{M}_k^{(\Delta)}(v)$ 
where for example $\mathcal{M}_0^{(\Delta)}$ is defined by \eqref{eq:onemainterm}, but with $w$ multiplied by the appropriate constituent of the partition of unity.
\end{mylemma}
Here the main term depends on $\Delta$, but is bounded from above by $\Delta^{1/2} T^{-1/2 + \varepsilon}$.  Thus in Lemma \ref{lemma:MDeltaUB} we can freely claim that $M_{\Delta}(v) = M.T.^{(\Delta)} + (\text{error})$.  Also, $\sum_{\Delta} M.T.^{(\Delta)}$ forms the main term in Theorem \ref{thm:fourthmomentlargeshifts}, as summing over $\Delta$ simply eliminates the partition of unity.  Taking the optimal choice of using Lemma 
\ref{lemma:MDeltaUB} for $\Delta \leq T^{31/33}$, and Lemma \ref{lemma:MDeltaAsymptotic} for $\Delta \geq T^{31/33}$, we obtain the error term stated in Theorem \ref{thm:fourthmomentlargeshifts}.

\begin{proof}[Proof of Lemma \ref{lemma:MDeltaUB}]
 The key tool here is Iwaniec's \cite{IwaniecFourthMoment} upper bound on the fourth moment:
\begin{equation}
\label{eq:IwaniecFourthMoment}
 \int_X^{X + Y} |\zeta(1/2 + it)|^4 dt \ll (X^{2/3} + Y) (XY)^{\varepsilon}.
\end{equation}
By applying H\"{o}lder's inequality to $M_{\Delta}(v)$ with exponents $4,4,4,4$, 
and using $w(t) \ll (T\Delta)^{-1/2}$ on this range of $t$, 
we obtain
\begin{equation}
 M_{\Delta}(v) \ll (\Delta T)^{-\frac12} \Big(\int_{T-2\Delta}^{T-\Delta} |\zeta(\tfrac12 + v + it + iT)|^4 dt \Big)^{\frac14} \Big(\int_{T-2\Delta}^{T-\Delta} |\zeta(\tfrac12 + v + it - iT)|^4 dt \Big)^{\frac14} (\dots),
\end{equation}
with the dots representing two more factors in the product which are identical to the two displayed terms, but with $v=0$.  Thus Iwaniec's bound gives
\begin{equation}
 M_{\Delta}(v) \ll (\Delta T)^{-1/2}  T^{\varepsilon} (\Delta + T^{2/3})^{1/2} \Delta^{1/2},
\end{equation}
which is \eqref{eq:MDeltaUB}.
\end{proof}

\begin{proof}[Proof of Lemma \ref{lemma:MDeltaAsymptotic}]
We may assume $\Delta \geq T^{31/33}$, since otherwise the result is already included in Lemma \ref{lemma:MDeltaUB}.

We apply Theorem \ref{thm:HughesYounglargeshifts} below with $h=k=1$, $w(t)$ replaced by $w_{\Delta}(t)$ and shifts defined by \eqref{eq:alphabetagammadeltadefinition}.  There are two error terms in Theorem \ref{thm:HughesYounglargeshifts}, denoted $E_1$ and $E_2$.  Here $E_1$ comes from taking the diagonal terms and shifting the contour to the left; a residue produces one of the main terms, and the new contour integral is this error term.
Precisely, $E_1$ is of the form \eqref{eq:E1errorbound} which with our current notation is
\begin{multline}
 \ll \int_{T-2\Delta}^{T-\Delta} (\Delta T)^{-1/2} T^{-1/2} \Delta^{-1/2} dt
\\
\max_{|u| \ll T^{\varepsilon}} |\zeta(1/2 + v + 2iT+iu) \zeta(1/2 + v + iu)^2 \zeta(1/2 + v-2iT + iu)|.
\end{multline}
By Weyl's bound, $E_1 \ll T^{-2/3 + \varepsilon}$.

The other error term $E_2$ arises as the error term in the shifted divisor problem. 
The off-diagonal terms are given by \eqref{eq:IO1} below, which we write here as
\begin{equation}
\label{eq:OD}
M_{\Delta}^{OD}(v) := \sum_{r \neq 0} \sum_{m-n=r}  \frac{\sigma_{\alpha, \beta}(m) \sigma_{\gamma, \delta}(n)}{(mn)^{1/2}} f(m,n),  
\end{equation}
with
\begin{equation}
f(x,y) = \intR \Big(\frac{x}{y}\Big)^{-it} V_{\alpha, \beta, \gamma, \delta, t}(\pi^2 xy) w_{\Delta}(t) dt,
\end{equation}
and
where the notation is given below by \eqref{eq:Vweightfunction}, \eqref{eq:galphabetagammadeltadef}.
By a short calculation, using \eqref{eq:alphabetagammadeltadefinition}, we have that
\begin{equation}
 \sigma_{\alpha,\beta}(n) = n^{-v} \tau_{iT}(n), \qquad \sigma_{\gamma, \delta}(n) = \tau_{iT}(n).
\end{equation}
We wish to apply Theorem \ref{thm:SCS}, which is an asymptotic formula for a shifted divisor sum.  For this, we need to know the sizes of the derivatives of $f$.
First we note that Stirling's formula implies
\begin{equation}
\label{eq:gshiftsbound}
g_{\alpha,\beta,\gamma,\delta}(s,t) \ll (\Delta T)^{\sigma},
\end{equation}
for $\text{Im}(s) \ll T^{\varepsilon}$, for $t$ in the support of $w_{\Delta}$.
The weight function $V$ decays quickly for $xy \geq (\Delta T)^{1+\varepsilon}$ (as this is the square-root of the conductor), by shifting the contour far to the right if necessary, using \eqref{eq:gshiftsbound}.
In terms of $t$, we have
\begin{equation}
 \frac{\partial}{\partial t} V_{\alpha,\beta,\gamma,\delta,t}(xy) \ll_{j,s} (xy)^{\varepsilon j} \Delta^{-j},
\end{equation}
where the implied constant depends polynomially on $s$; this is a slight generalization of \eqref{eq:wbound}.
Hence $H(t)$ defined by $H(t) = V_{\alpha,\beta,\gamma,\delta,t}(xy) w_{\Delta}(t)$ satisfies a bound similar to \eqref{eq:wbound}, except bigger by a small factor $(xy)^{\varepsilon j}$ (which has no practical effect for our work).
Furthermore, by integration by parts, $f(x,y)$ is small unless $\Delta |\log(x/y)| \leq  T^{\varepsilon}$.  
We can conclude that $f$ is very small unless $|x-y| \ll \Delta^{-1/2} T^{1/2 + \varepsilon}$.  Furthermore, since
$|x-y| \geq 1$, we have that $f$ is very small unless $x,y \gg \Delta T^{-\varepsilon}$.

By these observations, we have
\begin{equation}
\label{eq:offdiagonal}
M^{OD}_{\Delta}(v) = \sum_{r \neq 0}  \sum_{n \geq 1} \tau_{iT}(n) \tau_{iT}(n+r) F(n),
\end{equation}
where $F(n) = F_{r,v}(n) = n^{-v} f(n, n+r)$ is a smooth function satisfying the following bound
\begin{equation}
 \frac{d^j}{dx^j} F(x) \ll (1 + |v|)^j x^{-1-j} T^{-1/2} \Delta^{1/2}.
\end{equation}
In order to apply Theorem \ref{thm:SCS} to the shifted convolution sum appearing in \eqref{eq:offdiagonal},  we need two minor modifications to meet the conditions of the theorem.  Firstly, we need to apply a dyadic partition of unity to say $[N, 2N]$, summing over $1 \ll N \ll (\Delta T)^{1/2 + \varepsilon}$, and then we need to extract the constant factor $N^{-1} T^{-1/2} \Delta^{1/2}$ from the weight function (a simple re-scaling of the weight function).  
Having done this, we obtain
\begin{equation}
\label{eq:MODasymptotic}
 M^{OD}_{\Delta}(v) = M.T. + E.T.,
\end{equation}
where, in the notation of Theorem \ref{thm:SCS}, we have $Y = N$, $P = T^{\varepsilon}$,  $R \ll \frac{|r|T}{N} T^{\varepsilon}$.  Thus 
\begin{equation}
\label{eq:MDeltaOffdiagonalET}
 E.T. \ll \sum_{\substack{1 \ll N \ll \Delta^{1/2} T^{1/2+\varepsilon} \\ N \text{ dyadic}}} \frac{\Delta^{1/2} T^{\varepsilon} }{T^{1/2} N}  \sum_{1 \leq |r| \ll \frac{N}{\Delta} T^{\varepsilon}} 
\Big(T^{1/3} N^{1/2} \Big(\frac{|r|T}{N}\Big)^2 +  T^{1/6} N^{3/4}  \big(\frac{|r|T}{N} \big)^{1/2} 
 \Big).
\end{equation}
We do not need to work with the main term at this point because that is the purview of Theorem \ref{thm:HughesYounglargeshifts}.
To simplify \eqref{eq:MDeltaOffdiagonalET}, it is useful to notice that the ``worst'' value of $N$ is at $N = \Delta^{1/2} T^{1/2+\varepsilon}$ (the maximal value we need to consider), since after summing over $r$, $N$ does not appear in a denominator.  Thus, after simplifications, we obtain
\begin{equation}
 E.T. \ll T^{\varepsilon} (T^{25/12} \Delta^{-9/4} + T^{13/24} \Delta^{-5/8}),
\end{equation}
which is the error claimed in Lemma \ref{lemma:MDeltaAsymptotic}.  This term $E.T.$ is precisely one of the two terms making up $E_2$ in Theorem \ref{thm:HughesYounglargeshifts}, while the other term in $E_2$ has a bound of the same size by symmetry.
\end{proof}

\section{The main terms in the shifted fourth moment of the zeta function}
\label{section:maintermstwistedfourthmoment}
In \cite{HughesYoung}, we proved an asymptotic formula for the twisted fourth moment of the Riemann zeta function, with small shifts.  One standard approach to this type of problem is to reduce it to an asymptotic evaluation of a shifted divisor problem.  The main term in the fourth moment comes about through a complicated process of matching of terms.  The method used in \cite{HughesYoung} to simplify the main terms unfortunately does not carry over to large shift parameters because we used Stirling's approximation at different stages of the calculations.  
Here we modify the calculations of the main terms of \cite{HughesYoung} to carry over to arbitrary shifts.  This section is independent from the rest of this paper but heavily relies on the work of \cite{HughesYoung}.  For the purposes of this paper, we do not require the full ``twisted'' version, so the reader may wish to consider the simpler special case $h = k =1$ below.

Begin with some smooth weight function $w$ having compact support, and let
\begin{equation}
 I(h,k) = \int_{-\infty}^{\infty} \Big(\frac{h}{k}\Big)^{-it} \zeta(\tfrac12 + \alpha + it) \zeta(\tfrac12 + \beta + it) \zeta(\tfrac12 + \gamma - it) \zeta(\tfrac12 + \delta - it) w(t) dt,
\end{equation}
where $(h,k) = 1$.
We set some more notation.
Let $\sigma_{\alpha,\beta}(n) = \sum_{ab = n} a^{-\alpha} b^{-\beta}$.
Define
\begin{equation}
X_{\alpha,\beta,\gamma,\delta,t} = \pi^{\alpha + \beta + \gamma + \delta} 
\frac{\Gamma(\frac{\half -\alpha - it}{2})}{\Gamma(\frac{\half + \alpha + it}{2})}
\frac{\Gamma(\frac{\half -\beta - it}{2})}{\Gamma(\frac{\half + \beta + it}{2})}
\frac{\Gamma(\frac{\half -\gamma + it}{2})}{\Gamma(\frac{\half + \gamma - it}{2})}
\frac{\Gamma(\frac{\half -\delta + it}{2})}{\Gamma(\frac{\half + \delta - it}{2})},
\end{equation}
and similarly write $X_{\alpha,\beta,\gamma,\delta,t} = X_{\alpha,\gamma, t} X_{\beta,\delta, t}$ (with hopefully obvious meaning).  Let
\begin{equation}
\label{eq:galphabetagammadeltadef}
g_{\alpha, \beta, \gamma, \delta}(s,t) = 
\frac{\Gamma\left(\frac{\half + \alpha + s +it}{2} \right)}{\Gamma\left(\frac{\half + \alpha +it }{2} \right)} 
\frac{\Gamma\left(\frac{\half + \beta + s +it}{2} \right)}{\Gamma\left(\frac{\half + \beta +it }{2} \right)} 
\frac{\Gamma\left(\frac{\half + \gamma + s -it}{2} \right)}{\Gamma\left(\frac{\half + \gamma -it }{2} \right)}
\frac{\Gamma\left(\frac{\half + \delta + s -it }{2} \right)}{\Gamma\left(\frac{\half + \delta -it }{2} \right)},
\end{equation}
and with $G(s) = e^{s^2}$ (or any even holomorphic function with rapid decay in vertical strips), set
\begin{equation}
\label{eq:Vweightfunction}
V_{\alpha, \beta, \gamma, \delta,t}(x) = \frac{1}{2 \pi i} \int_{(1)} \frac{G(s)}{s} g_{\alpha, \beta, \gamma, \delta}(s,t) x^{-s} ds.
\end{equation}
Let
\begin{equation}
\label{eq:Aarithmeticfactor}
A_{\alpha,\beta,\gamma,\delta}(s) = \frac{\zeta(1+s + \alpha + \gamma)\zeta(1+s + \alpha + \delta)\zeta(1+s + \beta + \gamma)\zeta(1+s + \beta + \delta)}{\zeta(2+ 2s + \alpha + \beta + \gamma + \delta)},
\end{equation}
and supposing %$(h,k) = 1$, 
$p^{h_p} || h$ and $p^{k_p} || k$, we define
\begin{multline}
\label{eq:B}
B_{\alpha,\beta,\gamma,\delta,h,k}(s) = \prod_{p | h} \left(\frac{\sum_{j=0}^{\infty} \sigma_{\alpha,\beta}(p^j) \sigma_{\gamma,\delta}(p^{j + h_p}) p^{-j(s+1)}}{\sum_{j=0}^{\infty} \sigma_{\alpha,\beta}(p^j) \sigma_{\gamma,\delta}(p^{j}) p^{-j(s+1)}} 
\right)
\\
\times \prod_{p | k} 
\left(
\frac{\sum_{j=0}^{\infty} \sigma_{\alpha,\beta}(p^{j + k_p}) \sigma_{\gamma,\delta}(p^{j}) p^{-j(s+1)}}{\sum_{j=0}^{\infty} \sigma_{\alpha,\beta}(p^j) \sigma_{\gamma,\delta}(p^{j}) p^{-j(s+1)} }
\right).
\end{multline}
Let
\begin{equation}
Z_{\alpha,\beta,\gamma,\delta,h,k}(s) =  A_{\alpha,\beta,\gamma,\delta}(s) B_{\alpha,\beta,\gamma,\delta,h,k}(s).
\end{equation}
Our result is the following.
\begin{mytheo}
\label{thm:HughesYounglargeshifts}
Let $T > 0$, and suppose that $w(t)$ is supported on $|t| \leq T$.  Assume that the  
shifts $\alpha, \beta, \gamma, \delta$ have real part $O(1/\log{T})$, that $|\alpha + it|, |\beta+it|, |\gamma-it|, |\delta - it| \gg T^{\varepsilon}$ for all $t$ in the support of $w$, and that all the shifts have imaginary parts bounded by some fixed polynomial in $T$.  Then
\begin{multline}
\label{eq:Ihkasymptotic}
I(h,k) = \frac{1}{\sqrt{hk}} \int_{-\infty}^{\infty} w(t) \left( Z_{\alpha,\beta,\gamma,\delta,h,k}(0)  
+ X_{\alpha,\beta,\gamma,\delta, t} Z_{-\gamma, - \delta, -\alpha, -\beta,h,k}(0) \right. \\
+ X_{\alpha,\gamma,t} Z_{-\gamma,\beta,-\alpha,\delta,h,k}(0) 
+ X_{\alpha,\delta,t} Z_{-\delta,\beta,\gamma,-\alpha,h,k}(0) \\
\left.
+ X_{\beta,\gamma,t} Z_{\alpha,-\gamma, -\beta,\delta,h,k}(0) 
+ X_{\beta,\delta,t} Z_{\alpha,-\delta,\gamma,-\beta,h,k}(0)  
\right) dt 
+ E_1 + E_2 + O_A(T^{-A}),
\end{multline}
where $E_1$ is bounded by \eqref{eq:E1errorbound} below, and $E_2$ is defined by \eqref{eq:E2} below.
\end{mytheo}

We do not attempt to bound the error terms with large shifts since our purpose here is solely the calculation of the main terms.  
\begin{proof}
As in (37) of \cite{HughesYoung} (the approximate functional equation), we have
\begin{multline}
 I(h,k) = \sum_{m, n} \frac{\sigma_{\alpha,\beta}(m) \sigma_{\gamma,\delta}(n)}{\sqrt{mn}} \intR \Big(\frac{hm}{kn} \Big)^{-it} V_{\alpha,\beta,\gamma,\delta,t}(\pi^2 mn) w(t) dt 
\\
+
\sum_{m, n} \frac{\sigma_{-\gamma,-\delta}(m) \sigma_{-\alpha,-\beta}(n) }{\sqrt{mn}} \intR \Big(\frac{hm}{kn} \Big)^{-it} X_{\alpha,\beta,\gamma,\delta, t} V_{-\gamma,-\delta,-\alpha,-\beta, t}(\pi^2 mn) w(t) dt,
\end{multline}
and we write $I(h,k) = I^{(1)}(h,k) + I^{(2)}(h,k)$ correspondingly.  We will work mainly with $I^{(1)}$ because we can derive analogous formulas for $I^{(2)}$ by first switching $\alpha$ with $-\gamma$ and $\beta$ with $-\delta$, and then replacing $w(t)$ by $X_{\alpha,\beta,\gamma,\delta,t} w(t)$.

Let $I_D^{(1)}(h,k)$ denote the contribution to $I^{(1)}(h,k)$ from $hm=kn$.  By (46) of \cite{HughesYoung}, we have
\begin{equation}
\label{eq:IDhk1integral}
 I_D^{(1)}(h,k) = \frac{1}{\sqrt{hk}}  \int_{-\infty}^{\infty} w(t) 
\frac{1}{2 \pi i}  \int_{(\varepsilon)} \frac{G(s)}{s} (\pi^2 hk)^{-s}  g_{\alpha, \beta, \gamma, \delta}(s,t) 
Z_{\alpha,\beta,\gamma,\delta,h,k}(2s) 
ds dt.
\end{equation}
We shift the contour to $\text{Re}(s) = -\frac14 + \varepsilon$, crossing a pole at $s=0$ as well as four poles at $2s = -\alpha - \gamma$, etc.  The pole at $s=0$ gives the first main term in \eqref{eq:Ihkasymptotic}.  The other four poles give ``junk terms'' which surely cannot persist in the final answer since they depend on $G(s)$ which is chosen from a wide class of functions.  In fact it is possible to choose $G$ to vanish at these four points.  However, we do not need to impose this condition.  The new contour, which we view as an ``error term," is
\begin{equation}
\ll (hk)^{-1/4+\varepsilon} \intR |w(t)| \int_{(-1/4 + \varepsilon)} \frac{|G(s)|}{|s|} |g_{\alpha,\beta,\gamma,\delta}(s,t)| |A_{\alpha,\beta,\gamma,\delta}(2s)| ds dt,
\end{equation}
using a divisor-type bound on $B(s)$.  By Stirling's formula applied to $g(s,t)$, we have that this is
\begin{multline}
\ll (hk)^{-1/4+\varepsilon} T^{\varepsilon} \intR |w(t)| (1 + |\alpha+it|)^{-\frac14} 
(1 + |\beta+it|)^{-\frac14} (1 + |\gamma -it|)^{-\frac14} (1 + |\delta -it|)^{-\frac14}
\\
\int_{(-\frac14 + \varepsilon)} \frac{|G(s)|}{|s|}  |A_{\alpha,\beta,\gamma,\delta}(2s)| ds dt.
\end{multline}
By the rapid decay of $G(s)$, we bound the second line above by its maximum value (plus an error of size $O(T^{-A})$, which amounts to
\begin{multline}
\label{eq:E1errorbound}
\ll (hk)^{-1/4+\varepsilon} T^{\varepsilon} \intR |w(t)| (1 + |\alpha+it|)^{-\frac14} 
(1 + |\beta+it|)^{-\frac14} (1 + |\gamma -it|)^{-\frac14} (1 + |\delta -it|)^{-\frac14} dt
\\
\times \max_{|u| \ll T^{\varepsilon}} |\zeta(\tfrac12 + \alpha + \gamma + iu) \zeta(\tfrac12 + \alpha + \delta + iu) \zeta(\tfrac12 + \beta + \gamma + iu) \zeta(\tfrac12 + \beta + \delta + iu)| .
\end{multline}
Observe that switching the parameters $\alpha \leftrightarrow -\gamma$, $\beta \leftrightarrow -\delta$, and multiplying by $X_{\alpha,\beta,\gamma,\delta,t}$ (which has absolute value $O(1)$) does not alter the form of this bound, so it is valid also for the ``second part" of the approximate functional equation.

Next we look at the off-diagonal terms, which take the form
\begin{equation}
\label{eq:IO1}
I_O^{(1)}(h,k) = \sum_{r \neq 0} \sum_{hm - kn = r} \frac{\sigma_{\alpha,\beta}(m) \sigma_{\gamma,\delta}(n)}{\sqrt{mn}} f(hm,kn),
\end{equation}
with
\begin{equation}
\label{eq:fdefinition}
f(x,y) = \intR \Big(\frac{x}{y} \Big)^{-it} V_{\alpha,\beta,\gamma,\delta,t}\Big(\frac{\pi^2 xy}{hk} \Big) w(t) dt.
\end{equation}
As a working hypothesis, we suppose that we have an asymptotic formula for the shifted divisor sum
\begin{multline}
\label{eq:shifteddivisorsumasymptotic}
\sum_{hm - kn = r} \sigma_{\alpha,\beta}(m) \sigma_{\gamma,\delta}(n) F(hm,kn) 
= N_{\alpha,\beta,\gamma,\delta}(h,k;r;F) + N_{\beta,\alpha,\gamma,\delta}(h,k;r;F)
\\
+ N_{\alpha,\beta,\delta ,\gamma}(h,k;r;F) + N_{\beta,\alpha,\delta,\gamma}(h,k;r;F)
+ \mathcal{E}_{\alpha,\beta,\gamma,\delta}(h,k;r,F),
\end{multline}
where
\begin{multline}
\label{eq:Nalphabetagammadeltadef}
N_{\alpha,\beta,\gamma,\delta}(h,k;r;F) = \frac{\zeta(1-\alpha+\beta) \zeta(1-\gamma+\delta)}{h^{1-\alpha} k^{1-\gamma}} \int_{\max(0,r)}^{\infty} x^{-\alpha} (x-r)^{-\gamma} F(x,x-r) dx 
\\
\sum_{l=1}^{\infty} \frac{S(r,0;l) (h,l)^{1-\alpha+\beta} (k,l)^{1-\gamma+\delta}}{l^{2-\alpha+\beta-\gamma+\delta}},
\end{multline}
and $\mathcal{E}$ is a presumed error term, and it is not our purpose here to prove a bound on this $\mathcal{E}$.  Here $F(x,y) = (xy)^{-1/2} f(x,y)$.  Formally speaking, we have then
\begin{equation}
\label{eq:E2}
E_2 = \sum_{r \neq 0} \mathcal{E}_{\alpha,\beta,\gamma,\delta}(h,k;r; (xy)^{-1/2} f(x,y)) + \mathcal{E}_{-\gamma,-\delta,-\alpha,-\beta}(h,k;r; (xy)^{-1/2} f^*(x,y)),
\end{equation}
where $f(x,y) = f_{\alpha,\beta,\gamma,\delta}(x,y)$ is defined by \eqref{eq:fdefinition} and $f^*$ is identical to $f$ but with $\alpha \leftrightarrow -\gamma$, $\beta \leftrightarrow -\delta$, and multiplied by $X_{\alpha,\beta,\gamma,\delta,t}$.

Now define $I^{(1\pm)}_{\alpha,\beta,\gamma,\delta} = \sum_{\pm r > 0} N_{\alpha,\beta,\gamma,\delta}(h,k;r, F)$.  The arguments of (83)-(87) from \cite{HughesYoung} carry over almost without change.  
There is a small error in (89) which should read
\begin{multline}
\label{eq:Kminus}
K^- = r^{-\alpha - \gamma}
\int_{0}^{\infty} x^{-\half -\alpha}  (x+1)^{ -\half- \gamma} 
\\
\frac{1}{2 \pi i}  \int_{(\varepsilon)} \frac{G(s)}{s} \left(\frac{hk}{\pi^2 r^2 x(x+1)}\right)^{s}  \int_{-\infty}^{\infty} x^{-it}(1+x)^{it} g(s,t)   w(t) dt ds dx,
\end{multline}
where the difference is that we previously had $x^{-\half-\gamma} (1+x)^{-\half -\alpha}$ appearing in $K^{-}$.   In addition, there is an error in (91); the $\pm$ signs need to be switched on the right hand side.  It turns out that these two errors compensate for each other and the resulting calculations of \cite{HughesYoung} remain valid.

We cannot use (94) (which is Stirling's formula) because it assumes the shifts are small.  However, we can evaluate the $x$-integrals in terms of gamma functions, and use 
 (98) to evaluate the arithmetical sum.  In this way, we arrive at
\begin{multline}
I^{(1+)}_{\alpha,\beta,\gamma,\delta} = \frac{\zeta(1-\alpha+\beta) \zeta(1-\gamma + \delta)}{\zeta(2-\alpha+\beta-\gamma+\delta)} \frac{1}{h^{1/2-\alpha} k^{1/2-\gamma}}
\intR  \frac{w(t)}{2 \pi i} \int_{(\varepsilon)} \frac{G(s)}{s} g_{\alpha,\beta,\gamma,\delta}(s,t) 
\\
\Big(\frac{hk}{\pi^2} \Big)^s 
\frac{\Gamma(\tfrac12 - \gamma - s + it) \Gamma(\alpha + \gamma + 2s) }{\Gamma(\tfrac12 + \alpha + s + it)} C_{\alpha,\beta,\gamma,\delta,h,k}(s) \zeta(\alpha+\gamma+2s) \zeta(1+\beta+\delta+2s) ds dt.
\end{multline}
 Here $C(s)$ is a finite Euler product defined by (99)--(103) of \cite{HughesYoung}.  Similarly, for the case of $r < 0$, we find
\begin{multline}
I^{(1-)}_{\alpha,\beta,\gamma,\delta} = \frac{\zeta(1-\alpha+\beta) \zeta(1-\gamma + \delta)}{\zeta(2-\alpha+\beta-\gamma+\delta)} \frac{1}{h^{1/2-\alpha} k^{1/2-\gamma}}
\intR  \frac{w(t)}{2 \pi i} \int_{(\varepsilon)} \frac{G(s)}{s} g_{\alpha,\beta,\gamma,\delta}(s,t) 
\\
\Big(\frac{hk}{\pi^2} \Big)^s 
\frac{\Gamma(\tfrac12 - \alpha - s - it) \Gamma(\alpha + \gamma + 2s) }{\Gamma(\tfrac12 + \gamma + s - it)} C_{\alpha,\beta,\gamma,\delta,h,k}(s) \zeta(\alpha+\gamma+2s) \zeta(1+\beta+\delta+2s) ds dt.
\end{multline}
Note the following simple gamma function identity
\begin{equation}
 \frac{\Gamma(\tfrac12 - z)}{\Gamma(\tfrac12 + w)} + \frac{\Gamma(\tfrac12 - w)}{\Gamma(\tfrac12 + z)} = \pi^{-1} \Gamma(\tfrac12 - z) \Gamma(\tfrac12 - w) 2 \cos(\tfrac{\pi}{2}(z+w)) \cos(\tfrac{\pi}{2}(z-w)).
\end{equation}
We apply this with $z = \alpha + s +it$, $w  = \gamma + s - it$, whence
\begin{multline}
\sum_{\pm} I^{(1\pm)}_{\alpha,\beta,\gamma,\delta} = \frac{\zeta(1-\alpha+\beta) \zeta(1-\gamma + \delta)}{\zeta(2-\alpha+\beta-\gamma+\delta)} \frac{\pi^{-1}}{h^{1/2-\alpha} k^{1/2-\gamma}}
\intR  \frac{w(t)}{2 \pi i} \int_{(\varepsilon)} \frac{G(s)}{s} g_{\alpha,\beta,\gamma,\delta}(s,t) 
\\
\Big(\frac{hk}{\pi^2} \Big)^s 
\Gamma(\tfrac12 - \alpha - s - it) \Gamma(\tfrac12 - \gamma - s + it)  C_{\alpha,\beta,\gamma,\delta,h,k}(s) \cos(\tfrac{\pi}{2}(\alpha-\gamma+2it))
\\
  2 \cos(\tfrac{\pi}{2}(\alpha + \gamma + 2s)) \Gamma(\alpha + \gamma + 2s) \zeta(\alpha+\gamma+2s) \zeta(1+\beta+\delta+2s) ds dt.
\end{multline}

Next we use the functional equation of the zeta function in the form
\begin{equation}
\pi^{-2s} \Gamma(\alpha + \gamma + 2s) \zeta(\alpha + \gamma + 2s) = \frac{\pi^{\alpha + \gamma} 2^{\alpha + \gamma + 2s}}{2 \cos(\tfrac{\pi}{2}(\alpha + \gamma + 2s)} \zeta(1-\alpha-\gamma-2s).
\end{equation}
Thus we obtain (we suppress the subscripts on $C(s)$ as they match $g(s,t)$)
\begin{multline}
\label{eq:I1pm}
\sum_{\pm} I^{(1\pm)}_{\alpha,\beta,\gamma,\delta} = \frac{\zeta(1-\alpha+\beta) \zeta(1-\gamma + \delta)}{\zeta(2-\alpha+\beta-\gamma+\delta) \sqrt{hk}} 
\intR  \frac{w(t)}{2 \pi i} \int_{(\varepsilon)} \frac{G(s)}{s} g_{\alpha,\beta,\gamma,\delta}(s,t) 
\\
C(s) h^{s+\alpha} k^{s+\gamma}
\Gamma(\tfrac12 - \alpha - s - it) \Gamma(\tfrac12 -\gamma - s + it) 
 \pi^{\alpha + \gamma-1} 2^{\alpha + \gamma + 2s}
\\
\cos(\tfrac{\pi}{2}(\alpha-\gamma+2it))
\zeta(1-\alpha-\gamma-2s) \zeta(1+\beta+\delta+2s) ds dt.
\end{multline}
Anticipating some future simplifications, we apply another gamma function identity.  We claim
\begin{equation}
\pi^{-1} 2^{1-z-w} \Gamma(z) \Gamma(w) \cos(\tfrac{\pi}{2}(w-z))  = \half
\frac{\Gamma(\frac{z}{2})}{\Gamma(\frac{1-z}{2})} \frac{\Gamma(\frac{w}{2})}{\Gamma(\frac{1-w}{2})}
(1 + \tan(\tfrac{\pi}{2} z) \tan(\tfrac{\pi}{2} w)),
\end{equation}
which can quickly be checked using the duplication and reflection formulas for the gamma function, and the addition formula for cosine.  We apply this
with $z = 1/2 - \alpha - s - it$ and $w = 1/2 - \gamma - s + it$.   Our basic assumption $|\alpha + it| \gg T^{\varepsilon}$ (and similarly for the other shifts) implies that $\tan(\tfrac{\pi}{2}z) = \pm i + O_A(T^{-A})$ (the choice of sign depending on the sign of $t$) and similarly $\tan(\tfrac{\pi}{2}w) = \mp i + O_A(T^{-A})$.  Thus
\begin{multline}
\label{eq:I1pmsimplified}
\sum_{\pm} I^{(1\pm)}_{\alpha,\beta,\gamma,\delta} = \frac{\zeta(1-\alpha+\beta) \zeta(1-\gamma + \delta)}{\zeta(2-\alpha+\beta-\gamma+\delta) \sqrt{hk}} 
\intR  \frac{w(t)}{2 \pi i} \int_{(\varepsilon)} \frac{G(s)}{s} g_{\alpha,\beta,\gamma,\delta}(s,t) C(s) h^{s+\alpha} k^{s+\gamma}
\\
\frac{\Gamma(\frac{\half -\alpha -s- it}{2})}{\Gamma(\frac{\half + \alpha +s+ it}{2})}
\frac{\Gamma(\frac{\half -\gamma -s + it}{2})}{\Gamma(\frac{\half + \gamma +s- it}{2})}
 \pi^{\alpha + \gamma} 
\zeta(1-\alpha-\gamma-2s) \zeta(1+\beta+\delta+2s) ds dt + O(T^{-A}).
\end{multline}

We need to match this term with another one arising from the second part of the approximate functional equation.  Define $I^{(2 \pm)}_{\alpha,\beta,\gamma,\delta}$ to be the term arising from taking $I^{(1 \pm)}_{\beta,\alpha,\delta,\gamma}$, and then applying the symmetries $\alpha \leftrightarrow -\gamma$, $\beta \leftrightarrow -\delta$, and multiplying by $X_{\alpha,\beta,\gamma,\delta_t}$ (so in all we applied $\alpha \leftrightarrow - \delta$, $\beta \leftrightarrow -\gamma$, and multiplied by $X$).  Inserting these changes into \eqref{eq:I1pm} and changing variables $s \rightarrow -s$, we 
obtain
\begin{multline}
\label{eq:I2pmsimplified}
\sum_{\pm} I^{(2\pm)}_{\alpha,\beta,\gamma,\delta}  = - \frac{\zeta(1-\alpha+\beta) \zeta(1-\gamma + \delta)}{\zeta(2-\alpha+\beta-\gamma+\delta) \sqrt{hk}} 
\intR  \frac{w(t)}{2 \pi i} \int_{(-\varepsilon)} \frac{G(s)}{s} g_{-\delta, -\gamma, -\beta, -\alpha}(-s,t) \frac{C(-s)}{h^{s+\delta} k^{s+\beta}}
\\
X_{\alpha,\beta,\gamma,\delta,t}
\frac{\Gamma(\frac{\half +\delta +s- it}{2})}{\Gamma(\frac{\half -\delta -s+ it}{2})}
\frac{\Gamma(\frac{\half +\beta +s + it}{2})}{\Gamma(\frac{\half +-\beta -s- it}{2})}
 \pi^{-\beta-\delta} 
\zeta(1-\alpha-\gamma-2s) \zeta(1+\beta+\delta+2s) ds dt + O(T^{-A}).
\end{multline}
By Corollary 6.5 of \cite{HughesYoung}, we have
\begin{equation}
 h^{s+\alpha} k^{s+\gamma} C_{\alpha,\beta,\gamma,\delta,h,k}(s) = h^{-s-\delta} k^{-s-\beta} C_{-\delta, -\gamma, -\beta, -\alpha, h, k}(-s).
\end{equation}
We also claim that
\begin{multline}
g_{\alpha,\beta,\gamma,\delta}(s,t) \frac{\Gamma(\frac{\half -\alpha -s- it}{2})}{\Gamma(\frac{\half + \alpha +s+ it}{2})}
\frac{\Gamma(\frac{\half -\gamma -s + it}{2})}{\Gamma(\frac{\half + \gamma +s- it}{2})}
 \pi^{\alpha + \gamma} 
 \\
 =
 g_{-\delta, -\gamma, -\beta, -\alpha}(-s,t)
 X_{\alpha,\beta,\gamma,\delta,t}
\frac{\Gamma(\frac{\half +\delta +s- it}{2})}{\Gamma(\frac{\half -\delta -s+ it}{2})}
\frac{\Gamma(\frac{\half +\beta +s + it}{2})}{\Gamma(\frac{\half-\beta -s- it}{2})}
 \pi^{-\beta -\delta},
\end{multline}
which immediately follows upon writing out the definitions of both sides.

We can recognize now that $\sum_{\pm} I_{\alpha,\beta,\gamma,\delta}^{(1\pm)}$ and $\sum_{\pm} I_{\alpha,\beta,\gamma,\delta}^{(2\pm)}$ have identical integrands (up to the tiny error $O(T^{-A})$), but different contours that when combined give a closed contour around the finitely many poles with real part near $0$.  Thus by Cauchy's theorem this inner $s$-integral
can be expressed as the sum of residues with small real parts, plus a very small error term from $O_A(T^{-A})$.
The residue at $s=0$ gives
\begin{multline}
 \frac{\zeta(1-\alpha+\beta) \zeta(1-\gamma + \delta)\zeta(1-\alpha-\gamma) \zeta(1+\beta+\delta)}{\zeta(2-\alpha+\beta-\gamma+\delta) \sqrt{hk}} 
\intR  w(t) X_{\alpha,\gamma,t} C(0) h^{\alpha} k^{\gamma}
 dt + O(T^{-A}).
\end{multline}
From Lemma 6.10 of \cite{HughesYoung}, we have $h^{\alpha} k^{\gamma} C_{\alpha,\beta,\gamma,\delta,h,k}(0) = B_{-\gamma,\beta,-\alpha,\delta, h, k}(0)$.
Hence this residue gives the first term on the second line of \eqref{eq:Ihkasymptotic}.  The other three permutations of the off-diagonal terms give the remaining three terms of \eqref{eq:Ihkasymptotic}.

Finally, we briefly indicate how the ``junk terms" cancel.  This is luckily almost identical to the calculations of \cite{HughesYoung}.  Looking at \eqref{eq:I1pmsimplified}, the residue  at $2s = -\alpha - \gamma$ has the pleasant feature that the ratio of two gamma factors on the second line evaluates to $1$.  This term exactly matches the analogous residue of the diagonal term at $2s = -\alpha - \gamma$, using Lemma 6.11 of \cite{HughesYoung} to match the arithmetical factors.  Similarly, Lemma 6.12 of \cite{HughesYoung} matches the other residue at $2s = -\beta -\delta$ (in this case it is easier to compare via \eqref{eq:I2pmsimplified}).
\end{proof}

\section{Unconditional shifted divisor sum}
\label{section:SCSunconditional}
\begin{mytheo}
\label{thm:SCS}
 Suppose that $w(x)$ is a smooth function on the positive reals supported on $Y \leq x \leq  2Y$ and satisfying $w^{(j)}(x) \ll_j (P/Y)^j$ for some parameters $1 \leq P \leq Y$.  Let $\theta = 7/64$, and set $R = P + \frac{T|m|}{Y}$.  Then for $m \neq 0$, $ R \ll T/(TY)^{\delta}$, we have
\begin{equation}
\label{EisensteinShiftedConvolutionSum}
 \sum_{n \in \mz} \tau_{iT}(n) \tau_{iT}(n+m) w(n) = M.T. +  E.T.,
\end{equation}
where
\begin{multline}
\label{eq:MTdef}
 M.T. = \sum_{\pm}  \frac{|\zeta(1 + 2iT)|^2}{\zeta(2)} \sigma_{-1}(m) \int_{\max(0, -m)}^{\infty} (x + m)^{\mp iT} x^{\pm iT} w(x) dx
\\
+  \sum_{\pm} \frac{\zeta(1 \mp  2iT)^2}{\zeta(2 \mp 4iT)}  \sigma_{-1 \pm 4iT}(m) \int_{\max(0, -m)}^{\infty} (x + m)^{\mp iT} x^{\mp iT} w(x) dx.
\end{multline}
and
\begin{equation}
\label{eq:ETdef}
 E.T.\ll %T^{1/6} Y^{3/4} (TY)^{\varepsilon} +  
 %\begin{cases} 
(|m|^{\theta} T^{\frac13} Y^{\frac12} R^2 %+ m^{-\frac12} Y^{\frac12} R 
+ 
T^{\frac16} Y^{\frac34}  R^{\frac12} ) (TY)^{\varepsilon} .
%\\
 %         m^{\theta-\half} T^{-\frac16} Y R^{\frac52} %+ m^{-\frac12} Y^{\frac12} R 
%+ T^{-\frac{1}{12}} Y R^{\frac34}
%        \end{cases}
\end{equation}
Furthermore, with $R = P + \frac{TM}{Y}$, we have
\begin{equation}
\label{eq:ETsummedoverm}
 \sum_{1 \leq |m| \leq M} |E.T.| \ll (M T^{\frac13} Y^{\frac12} R^2 %+ m^{-\frac12} Y^{\frac12} R 
+ 
M T^{\frac16} Y^{\frac34}  R^{\frac12} ) (TY)^{\varepsilon}
\end{equation}
\end{mytheo}
The crucial qualitative feature of the error term is that when $X = T$, $m\asymp P \asymp 1$ (so $R \asymp 1$), it is a power saving in $X$.

The main term here is identical to the main term in \eqref{eq:shifteddivisorsumasymptotic}, because when $h=k=1$ the sum over $l$ in \eqref{eq:Nalphabetagammadeltadef} can be explicitly evaluated in the above form.

In his work on the $L^4$ norm of Eisenstein series, Spinu \cite{SpinuL4norm} was also led to estimating a shifted divisor sum with large spectral parameter, but of the form $\sum_{n \approx T^2} \tau_{iT}(n) \tau_0(n+m)$; see his (1.11).  

For the proof, we proceed in an analogous way to Motohashi \cite{MotohashiBinaryDivisor} who found a spectral expansion for shifted convolution sums.  We would not save a lot of effort by directly appealing to the work in \cite{MotohashiBinaryDivisor} because the exact form we need does not appear there, and anyway, with any approach
it is necessary to make estimations which forms the bulk of our work (by comparison, examine Chapter 5 of \cite{Motohashi}).  Therefore we chose to give a mostly self-contained exposition.

The proof requires the full spectral theory of automorphic forms.  The older Estermann-type approach using Weil's bound for Kloosterman sums is far from satisfactory; see \eqref{eq:Weilbound} for this estimate.

% First note that the stated error term is nontrivial for
% \begin{equation}
% \label{eq:TupperboundintermsofY}
% T \leq Y^{\frac{15}{14} - \varepsilon},
% \end{equation}
% at least for $m=1$,
% so we may impose this condition without loss of generality.  

By a symmetry argument, we may assume that $m \geq 1$; this uses the fact that $w_m(x) : = w(x-m)$ has support on $x \asymp Y$ too.

\subsection{Separation of variables}
As a starting point, we find it convenient to appeal to an approximate functional equation for the divisor function, as in Lemma 5.4 of \cite{YoungFourthMomentofDirichlet}.  This is a particularly simple alternative to the circle method or the delta symbol method for the divisor function. We have $\tau_{iT}(m) = m^{-iT} \sigma_{2iT}(m)$.  The formula is
\begin{equation}
\tau_{iT}(m) = m^{-iT} \sum_{l = 1}^{\infty} \frac{S(m,0;l)}{l^{1-2iT}} f_{2iT}\Big(\frac{l}{\sqrt{m}}\Big) + m^{iT} \sum_{l = 1}^{\infty} \frac{S(m,0;l)}{l^{1+2iT}} f_{-2iT}\Big(\frac{l}{\sqrt{m}}\Big),
\end{equation}
where
\begin{equation}
 f_{\lambda}(x) = \frac{1}{2 \pi i} \int_{(\varepsilon)} x^{-w} \zeta(1-\lambda + w) \frac{G(w)}{w} dw,
\end{equation}
and $G(w)$ is any fixed even holomorphic function with rapid decay in vertical strips (e.g., $\exp(w^2)$).  The terms are symmetric under $T \rightarrow -T$.  By shifting the contour to the right, we see that $f_{\pm 2 iT}(x) \ll x^{-A}$ with $A > 0$ arbitrary (uniformly in $T$).  One can verify this identity by reversing the orders of integration and summation, evaluating the sum over $l$, and using symmetry to express the sum of the two contour integrals as the residue.
% Moving the contour slightly past the $0$-line, we see that $f_{\pm 2 iT}(x) \ll \log xT$.  Generalizing this to arbitrary derivatives, and putting the two types of bounds together, we obtain
% \begin{equation}
%  \frac{d^j}{dx^j} f_{\pm 2iT}(x) \ll x^{-j} (\log{xT}) (1 + x)^{-A}.
% \end{equation}

Letting $B_w(T)$ denote the left hand side of \eqref{EisensteinShiftedConvolutionSum},
we then have
\begin{equation}
\label{eq:offdiagonal2}
B_w(T) = \sum_{\pm} \sum_{l=1}^{\infty} \frac{1}{l^{1\mp 2iT}} \sum_{n=1}^{\infty} \frac{ \tau_{iT}(n) S(n+m,0;l)}{n^{\pm iT}} \Big(1 + \frac{m}{n}\Big)^{\mp iT} f_{\pm 2iT}\Big(\frac{l}{\sqrt{n+m}}\Big) w(n).
\end{equation}
For later use, we wish to have the variables $l$ and $n$ separated, so we rearrange this as follows:
\begin{equation}
\label{eq:BwintermsofBwu}
 B_w(T) = \sum_{\pm} \frac{1}{2 \pi i} \int_{(1)} \frac{G(u)}{u} \zeta(1 \mp 2iT + u) B_{w, \pm, u, m}(T) du,
\end{equation}
with
\begin{equation}
\label{eq:offdiagonal3}
 B_{w, \pm, u, m}(T) =  \sum_{l=1}^{\infty}  \sum_{n=1}^{\infty} \frac{\tau_{iT}(n) S(n+m,0;l)}{l^{1\mp 2iT+u} n^{\pm iT - u/2}} g_{m,u,T}(n), \quad g_{m,u,T}(x) = \Big(1 + \frac{m}{x}\Big)^{\mp iT + \frac{u}{2}} w(x).
\end{equation}
The changing of orders of integration and summation is justified by absolute convergence.
This function $g$ has support for $x \asymp Y$ and satisfies
\begin{equation}
\label{eq:gproperty}
 \frac{d^j}{dx^j} g_{m,u,T}(x) \ll_{u,j}  \Big(\frac{T m}{Y^2} + \frac{P}{Y} \Big)^{j} = Y^{-j}  R^{j}, 
\end{equation}
with an implied constant depending polynomially on $u$.  In practice, we think of $u$ as almost bounded because the rapid decay of $G(u)$ overcomes any $u$-dependence of $B$.  Notice that in estimating $B_w(T)$, we are free to move the $u$-contour to any fixed line with $\text{Re}(u) > 0$ since the integrand is holomorphic in this region.

We record a bound on the Mellin transform of $g$.  Integration by parts shows
\begin{equation}
\label{eq:gformulaintegrationbyparts}
 \widetilde{g}(s) = \frac{(-1)^j}{s(s+1)\dots (s+j-1)} \int_0^{\infty} g^{(j)}(x) x^{s+j} \frac{dx}{x},
\end{equation}
so that using \eqref{eq:gproperty} we derive the bound
\begin{equation}
\label{eq:gbound}
 \widetilde{g}(s) \ll_{\sigma,j} Y^{\sigma} \frac{R^j}{|s(s+1)\dots(s+j-1)|}, %\frac{(1 + \frac{T|m|}{Y})^j}{|s(s+1)\dots(s+j-1)|},
\end{equation}
so that $\widetilde{g}(s)$ is very small for $|\text{Im}(s)| \gg (TY)^{\varepsilon} R$.  Since $g$ has compact support, $\widetilde{g}$ is entire, and the apparent singularities at $s=0, -1, -2, \dots$ in \eqref{eq:gformulaintegrationbyparts} are necessarily removable.

\subsection{Relationships with the Estermann function}
Next we open the Kloosterman sum, and use the Mellin transform of $g = g_{m,u,T}$ to write \eqref{eq:offdiagonal3} as
\begin{equation}
\label{eq:divisorsum}
B_{w, \pm, u, m}(T) =  \sum_{l=1}^{\infty} \sumstar_{a \shortmod{l}} \frac{\e{am}{l}}{l^{1\mp 2iT+u}} \frac{1}{2 \pi i} \int_{(2)} \widetilde{g}(s+\tfrac{u}{2})
D(s , \mp 2 i T, \frac{a}{l})  %\sum_{n} \frac{\sigma_{\mp 2iT}(n) \e{an}{l}}{n^{ - \frac{u}{2} +s}} 
ds,
\end{equation}
where we have used $\tau_{iT}(n)/n^{\pm iT} = \sigma_{\mp 2 iT}(n)$, and $D$ is the Estermann function defined for $\text{Re}(s) > \max(1, 1 + \text{Re}(\xi))$ 
\begin{equation}
 D(s, \xi, \frac{a}{l}) = \sum_{n=1}^{\infty} \sigma_{\xi}(n) \e{an}{l} n^{-s}.
\end{equation}
See Lemma 3.7 of \cite{Motohashi} for proofs of the following properties.
The Estermann function has a meromorphic continuation and functional equation
\begin{multline}
 \label{eq:EstermannFE}
D(s, \xi, a/l) = 2 (2\pi)^{2s - \xi - 2} l^{\xi - 2s + 1} \Gamma(1-s) \Gamma(1+\xi - s)
\\
(D(1-s, -\xi, \overline{a}/l) \cos(\tfrac12 \pi \xi) - D(1-s, -\xi, -\overline{a}/l) \cos(\pi (s-\tfrac12 \xi)) ).
\end{multline}
Furthermore, the pole at $s=1$ has residue $l^{\xi-1} \zeta(1-\xi)$ and the one at $1+\xi$ has residue $l^{-\xi-1} \zeta(1 + \xi)$.

We analyze $B$ by moving the contour to the line $-1-\varepsilon$, applying the functional equation, and reversing the order of summation and integration.  There are poles at $s = 1$ and $s  =  1\mp 2iT$ which lead to the main terms that we shall examine shortly in Section \ref{section:SCSmainterm}.  Thus we obtain, with shorthand $\alpha_T = 2 (2\pi)^{\pm 2iT - 2}$,
\begin{multline}
\label{eq:SafterFunctionalequation}
 B_{w,\pm,u,m} - M.T. = \alpha_{T}\sum_{l=1}^{\infty} \sumstar_{a \shortmod{l}} \frac{\e{am}{l}}{l^{u}} \sum_{n =1}^{\infty} \frac{\sigma_{\pm 2iT}(n)}{n^{}} 
\frac{1}{2\pi i} \int_{(-1-\varepsilon)} \widetilde{g}(s + \tfrac{u}{2}) 
\\
(2\pi)^{2s} \Big(\frac{n}{l^2}\Big)^s 
\Gamma(1-s) \Gamma(1\mp 2iT -s) 
\Big(e\Big(\frac{\overline{a} n}{l}\Big) \cosh(\pi T) - e\Big(\frac{-\overline{a} n}{l}\Big) \cos(\pi(s \pm iT)) \Big) ds,
\end{multline}
% Thus we obtain
% \begin{equation}
% \label{eq:SafterFunctionalequation}
%  B_{w,\pm,u,m} = M.T. + \alpha_{T}\sum_{l} \sumstar_{a \shortmod{l}} \frac{\e{am}{l}}{l^{u}} \sum_{n =1}^{\infty} \frac{\sigma_{\pm 2iT}(n)}{n^{}}\Big(e\big(\frac{\overline{a} n}{l}\big) k_{1}\big(\frac{4 \pi \sqrt{n}}{l} \big) - e\big(\frac{-\overline{a} n}{l}\big) k_{2}\big(\frac{4 \pi \sqrt{n}}{l} \big) \Big),
% \end{equation}
where $M.T.$ stands for the main terms.  

\subsection{Computing the main terms}
\label{section:SCSmainterm}
The pole at $s=1$ gives to $B_w(T)$
\begin{equation}
\sum_{\pm} \frac{1}{2 \pi i} \int_{(1)} \frac{G(u)}{u} \zeta(1 \mp 2iT + u) \sum_{l=1}^{\infty} \sumstar_{a \shortmod{l}} \frac{e(\frac{am}{l})}{l^{1\mp 2iT + u}} \widetilde{g}(1+\tfrac{u}{2}) \zeta(1\pm 2iT) l^{-1 \mp 2iT} du.
\end{equation}
Using $S(m,0;l) = \sum_{d|(m,l)} d \mu(l/d)$, this 
simplifies as
\begin{equation}
\sum_{\pm} \frac{1}{2 \pi i} \int_{(1)} \frac{G(u)}{u} \widetilde{g}(1+\tfrac{u}{2}) \zeta(1 \mp 2iT + u) \frac{\sigma_{-1-u}(m)}{\zeta(2+u)}  \zeta(1\pm 2iT)  du.
\end{equation}
By shifting the contour left to $\text{Re}(u) = -1/2$ , we cross a pole at $u=0$ which gives
\begin{equation}
\label{eq:MT1}
\sum_{\pm} \widetilde{g}(1) \frac{|\zeta(1 + 2iT)|^2}{\zeta(2)} \sigma_{-1}(m).
\end{equation}
The pole at $u = \pm 2iT$ is very small but it also will cancel a forthcoming term (this should not be surprising because one could choose a different function $G$ and change the value of $G(\pm 2iT)$).
Using \eqref{eq:gbound} and Weyl's bound, we see that the new integral satisfies the bound
\begin{equation}
\label{eq:boundforoffdiagonalmainterms}
(TY)^{\varepsilon} \intR \Big| \frac{G(-\frac12 + it)}{-\frac12 + it}  \widetilde{g}(3/4 + it/2)  \zeta(1/2 + it \mp 2iT) \Big|     dt \ll  T^{1/6}  Y^{3/4} (YT)^{\varepsilon}, 
\end{equation}
and this is absorbed by the second error term appearing in \eqref{eq:ETdef}.

The other residue at $s=1 \mp 2iT$ equals
\begin{equation}
\frac{1}{2 \pi i} \int_{(1)} \frac{G(u)}{u} \widetilde{g}(1 \mp 2iT + \tfrac{u}{2}) \zeta(1 \mp 2iT + u) \frac{\sigma_{-1\pm 4iT-u}(m)}{\zeta(2\mp 4iT +u)}  \zeta(1\mp 2iT)  du.
\end{equation}
Again shifting contours to the left to $\text{Re}(u) = -1/2$, we cross a pole at $u=0$ which gives
\begin{equation}
\label{eq:MT2}
\widetilde{g}(1 \mp 2iT) \frac{\zeta(1 \mp  2iT)^2}{\zeta(2 \mp 4iT)} \sigma_{-1 \pm 4iT}(m).
\end{equation}
The pole at $u = \pm 2iT$ is small, but could also be checked to cancel a previous pole.  The new contour is much smaller than the bound appearing in \eqref{eq:boundforoffdiagonalmainterms}.

Note that
\begin{equation}
\widetilde{g}(1) = \int_0^{\infty} (x + m)^{\mp iT} x^{\pm iT} w(x) dx, \quad \widetilde{g}(1 \mp 2iT) = \int_0^{\infty} (x + m)^{\mp iT} x^{\mp iT} w(x) dx.
\end{equation}
Thus \eqref{eq:MT1} and \eqref{eq:MT2} form the main term in \eqref{eq:MTdef}.  This combination of main terms precisely agrees with the main terms of \eqref{eq:shifteddivisorsumasymptotic}.

\subsection{Analysis of the weight functions}
We return to \eqref{eq:SafterFunctionalequation}.
By changing variables $s \rightarrow s+ \half - \frac{u}{2}$, it becomes
\begin{equation}
\label{eq:BwumwithKloosterman}
 B_{w,\pm,u,m} - M.T. = \alpha_{T}' \sum_{\delta = \pm 1} \delta  \sum_{n=1}^{\infty} \frac{\sigma_{\pm 2iT}(n)}{n^{\frac12 + \frac{u}{2}}} \sum_{l=1}^{\infty} \frac{S(n, \delta m;l)}{l} \varphi_{\delta} \Big(\frac{4 \pi \sqrt{mn}}{l}\Big),
\end{equation}
where $\alpha_T' = 2 (2\pi)^{\pm 2iT -1 - u}$ and
\begin{equation}
\label{eq:varphideltadef}
 \varphi_{\delta} (x) = \frac{1}{2\pi i} \int_{(\sigma)} \widetilde{g}(s + \tfrac{1}{2}) (x/2)^{2s} m^{-s} \Gamma(\tfrac12-s + \tfrac{u}{2}) \Gamma(\tfrac12-s + \tfrac{u}{2} \mp 2iT) c_{\delta}(s) ds,
\end{equation}
with $c_{+}(s) = \cosh(\pi T)$ and $c_{-}(s) = \cos(\pi(s + \half -\frac{u}{2} \pm iT))$. Here $\sigma < 1/2$ (recall $\text{Re}(u) > 0$).

We shall derive some bounds on $\varphi_{\delta}$.
Let $\text{Re}(s) = \sigma$ and $\text{Re}(u) = \nu > 0$.  We claim that for $\sigma < 1/2$ fixed, and each choice of $\delta = \pm$, that
\begin{equation}
\label{eq:gammabound}
 \Gamma(\tfrac12-\sigma - it) \Gamma(\tfrac12-\sigma - it  \mp 2iT) c_{\delta}(\sigma + it) \ll (1 + |t|)^{-\sigma} (1 + |t\pm 2T|)^{-\sigma}.
\end{equation}
To prove this, initially in the case $\delta = +$, observe that Stirling's formula gives a bound of this form, except multiplied by $\exp(\frac{\pi}{2} Q_1(t))$, where $Q_1(t) = 2T - |t \pm 2T| - |t|$.  Notice $Q_1(t \mp T) = Q(t,T/2)$ defined by \eqref{eq:Qdef}, so $Q_1(t) \leq 0$ for all $t,T \in \mr$ (this can also be checked directly).  The case $\delta = -$ is similar, except with $Q_1(t)$ replaced by $Q_2(t) = 2|t\pm T| - |t| - |t \pm 2T|$, which we again claim is $\leq 0$ for all $t$.  It takes the value $0$ for $|t|$ large, and so its maximum value must occur at $t =0$ or $t \pm 2T = 0$, and $Q_2(t) = 0$ at both of these points.

By \eqref{eq:gbound} and \eqref{eq:gammabound}, we have
\begin{equation}
\label{eq:k1integralbound}
 \varphi_{\delta}(x) \ll  \Big(\frac{x^2}{m}\Big)^{\sigma} \intR 
%\exp(\pi T -\frac{\pi}{2} (|t+2T|+ |t|)) \exp(\tfrac{\pi}{2}
(1 + |t|)^{ - \sigma + \frac{\nu}{2}} (1 + |t \pm 2T|)^{ - \sigma + \frac{\nu}{2}} |\widetilde{g}(\sigma + \tfrac12 + it)| dt,
\end{equation}
the implied constant depends polynomially on $\text{Im}(u)$.
Using the assumption $R \ll T/(TY)^{\delta}$, %which implies $m \ll Y/(YT)^{\delta}$, 
we deduce for fixed $\sigma < 1/2$ that
\begin{equation}
 \varphi_{\delta}(x) \ll_{\sigma, u} x^{2\sigma} Y^{\frac12 + \sigma} m^{-\sigma} T^{-\sigma+\frac{\nu}{2}} R^{1-\sigma+\frac{\nu}{2}}.
\end{equation}
By taking $\sigma = 1/2-\varepsilon$ or $-\sigma > 0$ very large, we obtain 
\begin{equation}
\label{eq:k1bound}
 \varphi_{\delta}(x) \ll x Y^{} m^{-1/2} T^{-1/2} R^{1/2} (TY/x)^{\varepsilon} \Big(1 + \frac{x^2
}{mTR/Y} \Big)^{-A} (TR)^{\nu/2}.
\end{equation}
% One can also differentiate the integral representation of $\varphi_{\delta}(x)$ to show
% \begin{equation}
%  x^j \varphi_{\delta}^{(j)}(x) \ll_{T,Y,m,R,\varepsilon} x^{2-\varepsilon} (1 + x)^{-A}.
% \end{equation}

Since it is straightforward at this point, we record the effect of bounding $B$ trivially with the Weil bound, even though this is not our final objective.  The sum over $l$ can be truncated at $Y^{1/2 + \varepsilon}$.  Taking $\nu = \varepsilon$ and $\sigma = 1/2-\varepsilon$, we then obtain that the error term in $B$ is
\begin{equation}
 \ll (TY)^{\varepsilon}  \sum_{l \ll Y^{1/2 + \varepsilon}} \sum_{n \ll (TY)^{\varepsilon} l^2 RT/Y} |S(n, \pm m, l)|  \frac{Y R^{1/2}}{l^2 T^{1/2}}.
\end{equation}
A short computation shows that this error term is
\begin{equation}
\label{eq:Weilbound}
 \ll Y^{3/4} T^{1/2} R^{3/2} (TY)^{\varepsilon},
\end{equation}
which is only nontrivial for $T \leq Y^{1/2-\delta}$ when $R \asymp 1$.

\subsection{Application of the Kuznetsov formulas}
To get a stronger result, we apply the spectral theory of automorphic forms in the guise of the Kuznetsov formula.  Let
\begin{equation}
 K_{\delta}(m, n, \varphi_{\delta}) = \sum_{l=1}^{\infty} \frac{S(n, \delta m;l)}{l} \varphi_{\delta} \Big(\frac{4 \pi \sqrt{mn}}{l}\Big).
\end{equation}
The Kuznetsov formula (in one particularly useful way for us) states
\begin{multline}
\label{eq:Kuznetsovplus}
 K_{+}(m, n, \varphi_{+}) = \sum_j \frac{|\rho_j(1)|^2}{\cosh(\pi t_j)} \lambda_j(m) \lambda_j(n) \frac{1}{2\pi i} \int_{(\sigma)} \varphi_{+}^{*}(s) h_{+}(s,t_j) ds 
\\
+ \frac{1}{\pi} \intR \frac{\tau_{ir}(m) \tau_{ir}(n)}{|\zeta(1+2ir)|^2} \frac{1}{2\pi i} \int_{(\sigma)} \varphi_{+}^{*}(s) h_{+}(s,r) ds dr
\\
+ \sum_k (2k-1) q_{m,n}(k) \frac{1}{2\pi i} \int_{(\sigma)} \varphi_{+}^{*}(s) \frac{\Gamma(k-1+s)}{\Gamma(k+1-s)} ds - \frac{1}{2\pi} \delta_{m=n} \frac{1}{2\pi i} \int_{(\sigma)} \varphi_{+}^{*}(s) \frac{\Gamma(s)}{\Gamma(1-s)} ds,
\end{multline}
with notation as follows: the $s$-integrals are over a vertical line with $\sigma = 3/4 + \varepsilon$, 
\begin{equation}
 h_{+}(s, r) = \tfrac12 \sin(\pi s) \Gamma(s-\tfrac12 + ir) \Gamma(s - \tfrac12 - ir),
\end{equation}
\begin{equation}
 \varphi_{\delta}^{*}(s) = \int_0^{\infty} \varphi_{\delta}(x) (x/2)^{-2s} dx,
\end{equation}
and $q_{m,n}(k)$ is the sum over a Hecke eigenbasis of weight $2k$ (even) holomorphic forms, with scaling such that
\begin{equation}
 q_{m,n}(k) = k^{-1 + o(1)} \sum_{f \in B_{2k}} \lambda_f(m) \lambda_f(n).
\end{equation}
We also recall that $\rho_j(1)$ is scaled so that $\frac{|\rho_j(1)|^2}{\cosh(\pi t_j)} = t_j^{o(1)}$.  One can find this form of Kuznetsov's formula as (2.4.13) of \cite{Motohashi}.  It is easy to check that $\varphi_{\delta}(x)$ satisfies the required bounds to apply the Kuznetsov formula, using \eqref{eq:k1bound} as well as simple variants for the derivatives of $\varphi_{\delta}$.

By the Mellin inversion formula and \eqref{eq:varphideltadef}, we have
\begin{equation}
 \varphi_{\delta}^{*}(s) =   2^{2s}  \widetilde{\varphi_{\delta}}(1-2s) =   \widetilde{g}(s) m^{\half - s} \Gamma(1-s+\tfrac{u}{2})  \Gamma(1-s+\tfrac{u}{2} \mp 2iT) c_{\delta}(s-\tfrac12).
\end{equation}
It follows from \eqref{eq:gbound} and \eqref{eq:gammabound} that
for $|t| \leq R(TY)^{\varepsilon}$ (recall $R \ll T(TY)^{-\delta}$ so $t = o(T)$):
\begin{equation}
\label{eq:varphistarbound}
 \varphi_{\delta}^*(s) \ll_{\sigma, u} Y^{\sigma} m^{\half - \sigma} (1 + |t|)^{\half - \sigma + \frac{\nu}{2}} T^{\half - \sigma + \frac{\nu}{2}}.
\end{equation}
Also, $\varphi_{\delta}^*(s)$ is very small for $|t| \geq R(TY)^{\varepsilon}$, and the implied constant depends polynomially on $\text{Im}(u)$.

We also need the opposite sign case
where the Kuznetsov formula has a different shape than \eqref{eq:Kuznetsovplus}.  For this, we claim
\begin{multline}
\label{eq:Kuznetsovminus}
 K_{-}(m,n, \varphi_{-}) = \sum_j \frac{|\rho_j(1)|^2}{\cosh(\pi t_j)} \lambda_j(-m) \lambda_j(n)  \frac{1}{2\pi i} \int_{(\sigma)} \varphi_{-}^{*}(s) h_{-}(s,t_j) ds 
\\
+ \frac{1}{\pi} \intR \frac{\tau_{ir}(m) \tau_{ir}(n)}{|\zeta(1+2ir)|^2}  \frac{1}{2\pi i} \int_{(\sigma)} \varphi_{-}^{*}(s) h_{-}(s,r) ds,
\end{multline}
where $1/2 < \sigma < 1$, and
\begin{equation}
 h_{-}(s,r) = \tfrac12 \cosh( \pi r) \Gamma(s-\tfrac12 + ir) \Gamma(s-\tfrac12 - ir).
\end{equation}
Towards this, we first quote Theorem 2.5 of \cite{Motohashi}, stating 
\begin{equation}
 K_{-}(m,n, \varphi) = \sum_j \frac{|\rho_j(1)|^2}{\cosh(\pi t_j)} \lambda_j(-m) \lambda_j(n) \check{\varphi}(t_j) %\frac{1}{2\pi i} \int \varphi_{-}^{*}(s) h_{-}(s,t_j) ds 
+ \frac{1}{\pi} \intR \frac{\tau_{ir}(m) \tau_{ir}(n)}{|\zeta(1+2ir)|^2} \check{\varphi}(r)dr %\frac{1}{2\pi i} \int \varphi_{-}^{*}(s) h_{-}(s,r) ds,
\end{equation}
where
\begin{equation}
 \check{\varphi}(r) = 2 \cosh(\pi r) \int_0^{\infty} \varphi(x) K_{2ir}(x) \frac{dx}{x}.
\end{equation}
We prefer the Mellin transform version of this formula, which we quickly derive as follows:
\begin{equation}
\check{\varphi}(r) = \frac{2 \cosh( \pi r)}{2 \pi i} \int_{(\sigma)} \widetilde{\varphi}(-s) \int_0^{\infty} K_{2ir}(x) x^{s-1} dx ds,
\end{equation}
valid for $0 < \sigma < 1$ (the upper bound arises to ensure the holomorphy of $\widetilde{\varphi}(-s)$), so by (6.561.16) of \cite{GR} and a change of variables, we have for $1/2 < \sigma < 1$ that
\begin{equation}
\check{\varphi}(r) =  \half \frac{\cosh( \pi r) }{2 \pi i} \int_{(\sigma)} \widetilde{\varphi}(1-2s) 2^{2s} \Gamma(s-\tfrac12 + ir) \Gamma(s-\tfrac12 - ir) ds.
\end{equation}
Using $\varphi_{\delta}^{*}(s) =   2^{2s}  \widetilde{\varphi_{\delta}}(1-2s)$,  we have
\begin{equation}
 \check{\varphi}(r) = \half \cosh( \pi r) \frac{1}{2\pi i}\int_{(\sigma)} \varphi^*(s) \Gamma(s-\tfrac12 + ir) \Gamma(s-\tfrac12 - ir) ds.
\end{equation}
Hence we derive \eqref{eq:Kuznetsovminus}.

Define
\begin{multline}
\label{eq:EMaass+}
 E^{\delta}_{\text{Maass}} = \delta \sum_{\pm} \frac{1}{2 \pi i} \int_{(2)} \frac{G(u)}{u} \zeta(1 \mp 2iT + u) 2 (2\pi)^{\pm 2iT -1 - u} \sum_{n=1}^{\infty} \frac{\sigma_{\pm 2iT}(n)}{n^{\frac12 + \frac{u}{2}}} 
\\
\sum_j \frac{|\rho_j(1)|^2}{\cosh(\pi t_j)} \lambda_j(m) \lambda_j(n) \frac{1}{2\pi i} \int_{(\sigma)} \varphi_{\delta}^{*}(s) h_{\delta}(s,t_j) ds du,
\end{multline}
\begin{multline}
\label{eq:EEisdef}
E_{\text{Eis}} = \sum_{\pm} \frac{1}{2 \pi i} \int_{(2)} \frac{G(u)}{u} \zeta(1 \mp 2iT + u) 2 (2\pi)^{\pm 2iT -1 - u} \sum_{n=1}^{\infty} \frac{\sigma_{\pm 2i T}(n)}{n^{\frac12 + \frac{u}{2}}}
\\
\frac{1}{\pi} \intR \frac{\tau_{ir}(m) \tau_{ir}(n)}{|\zeta(1+2ir)|^2}  \frac{1}{2\pi i} \int_{(\sigma)} [\varphi_{+}^{*}(s) h_{+}(s,r) - \varphi_{-}^{*}(s) h_{-}(s,r)] ds dr du.
\end{multline}
Then 
inserting \eqref{eq:Kuznetsovplus} and \eqref{eq:Kuznetsovminus}
into \eqref{eq:BwumwithKloosterman}, and then in turn into \eqref{eq:BwintermsofBwu}, we get, say
\begin{equation}
\label{eq:Bwmaintermerrorterm}
 B_w(T) = M.T. + O(T^{1/6} Y^{3/4} (TY)^{\varepsilon}) + E_{\text{Maass}}^{+} + E_{\text{Maass}}^{-} + E_{\text{holo}} + E_{\text{diag}} + E_{\text{Eis}}.
\end{equation}
Here this error term comes from the calculation of the residues of the Estermann function as in Section \ref{section:SCSmainterm}, and $E_{\text{holo}}$ and $E_{\text{diag}}$ are given by analogous formulas to $E_{\text{Maass}}$ and $E_{\text{Eis}}$.  See \eqref{eq:Eholodef} below for $E_{\text{holo}}$.
In the forthcoming sections, we show
\begin{equation}
\label{eq:EMaassbound1}
\sum_{\delta = \pm} |E^{\delta}_{\text{Maass}}| + |E_{\text{holo}}| + |E_{\text{diag}}| \ll m^{\theta} T^{1/3} Y^{1/2} R^2 (TY)^{\varepsilon},
\end{equation}
and
\begin{equation}
\label{eq:EEisbound}
 E_{\text{Eis}} \ll T^{1/6}  Y^{3/4}   R^{1/2} (TY)^{\varepsilon} + T^{1/3} Y^{1/2} R (TY)^{\varepsilon},
\end{equation}
which together account for \eqref{eq:ETdef}.  Note the second term in \eqref{eq:EEisbound} is smaller than \eqref{eq:EMaassbound1}, as is the error term in \eqref{eq:Bwmaintermerrorterm}.

\subsection{Cusp form contributions}
We pick up with \eqref{eq:EMaass+}.
By absolute convergence, we can bring the sum over $n$ to the inside.  By an exercise with the Hecke relations, one can check that for $\text{Re}(v) > 1$, 
\begin{equation}
 \sum_{n=1}^{\infty} \frac{\sigma_{\pm 2iT}(n) \lambda_j(n)}{n^v} = \frac{L(v\mp 2iT, u_j)L(v, u_j)}{\zeta(2v \mp 2iT)}.
\end{equation}
We evaluate the sum over $n$, and then shift the $u$-contour back to $\nu = \varepsilon$ (crossing no poles), and obtain
\begin{multline}
\label{eq:EMaassdef}
 E^{\delta}_{\text{Maass}} = \sum_{\pm} \frac{1}{2 \pi i} \int_{(\varepsilon)} \frac{G(u)}{u} \alpha_{T}'
\\
\sum_j \frac{|\rho_j(1)|^2}{\cosh(\pi t_j)} \lambda_j(\delta m) L(\tfrac12 + \tfrac{u}{2}\mp 2iT, u_j)L(\tfrac12 + \tfrac{u}{2}, u_j)  \frac{1}{2\pi i} \int_{(3/4 + \varepsilon)} \varphi_{\delta}^{*}(s) h_{\delta}(s,t_j) ds du.
\end{multline}
Next we estimate the $s$-integral above, initially for $\delta = +$.  Compared to \eqref{eq:varphideltadef}, the integral representation is the same except we have $s$ shifted by $1/2$, and we have multiplied by $h_{+}(s,t_j)$.  Stirling's formula shows that 
\begin{equation}
\label{eq:h+bound}
h_+(\sigma + it, r) \ll (1 + |t + r|)^{\sigma - 1} (1 + |t - r|)^{\sigma - 1} \exp(\tfrac{\pi}{2} (2|t| - |t+r| - |t-r|).
\end{equation}
The exponential part above is $1$ for $|r| \leq |t|$, and is $\exp(-\pi |t-r|)$ for $|r| > |t|$ (again we have encountered \eqref{eq:Qdef}).  Thus the $s$-integral in \eqref{eq:EMaassdef} is very small unless $t_j \leq R(TY)^{\varepsilon}$.  Suppose that $1/2 < \sigma < 1$.  Then with $\nu = \varepsilon$, using \eqref{eq:varphistarbound} and \eqref{eq:h+bound}, we have
\begin{multline}
\label{eq:h+varphi+starbound}
 \int_{(\sigma)} \varphi_{+}^*(s) h_{+}(s,r)ds 
\\
\ll Y^{\sigma} m^{\frac12 - \sigma} T^{\frac12 - \sigma} (TY)^{\varepsilon} \int_{|t| 
\ll R(TY)^{\varepsilon}} (1 + |t|)^{\half - \sigma} (1 + |t+r|)^{\sigma - 1} (1 + |t-r|)^{\sigma - 1} dt+ \dots,
\end{multline}
with the dots representing a very small error term arising from the truncation.  
Taking $\sigma = 1/2 + \varepsilon$, we obtain the bound
\begin{equation}
 \int_{(\sigma)} \varphi_{+}^*(s) h_{+}(s,r)ds \ll (TY)^{\varepsilon}
Y^{1/2} \Big(1 + \frac{|r|}{R (TY)^{\varepsilon}}\Big)^{-100} .
\end{equation}
Using the uniform subconvexity bound $L(1/2 + u/2 \pm 2i T, u_j) \ll (t_j + T)^{1/3+\varepsilon}$ of Jutila-Motohashi \cite{JutilaMotohashi}, the bound $|\lambda_j(m)| \ll m^{\theta+\varepsilon}$ (with $\theta = 7/64$), and the mean value result (with polynomial dependence on $u$) $\sum_{t_j \leq R} |L(1/2 + u/2, u_j)|^2 \ll R^{2+\varepsilon}$ (see \cite{Motohashi}, Theorem 3.1, for example), 
we obtain \eqref{eq:EMaassbound1} for $E_{\text{Maass}}^{+}$.

Furthermore, if we sum this error term over $m \leq M$, then we can use $\sum_{m \leq M} |\lambda_j(m)|^2 \ll M (t_j M)^{\varepsilon}$ \cite{IwaniecSpectralgrowth}, which explains why \eqref{eq:ETsummedoverm} effectively has $\theta =0$.

We claim that the bound of \eqref{eq:EMaassbound1} holds for $E_{\text{Maass}}^{-}$ also.
By Stirling's formula, analogously to \eqref{eq:h+bound},
\begin{equation}
 h_{-}(\sigma + ir, r) \ll (1 + |t+r|)^{\sigma -1} (1 + |t-r|)^{\sigma -1} \exp(\tfrac{\pi}{2} (2|r| - |t+r| - |t-r|)
\end{equation}
The polynomial factor here is identical to that in \eqref{eq:h+bound}, while the exponential factor is $1$ for $|r| \geq |t|$, and is $\exp(-\pi |r-t|)$ for $|r| < |t|$, which simply means that we cannot immediately truncate $|r|$ at $R (TY)^{\varepsilon}$ in this case, in contrast to the $E_{\text{Maass}}^{+}$ case.  If $|r| \leq R(TY)^{\varepsilon}$, then in fact \eqref{eq:h+varphi+starbound} holds in the case $\delta = -$ too, since the bound \eqref{eq:varphistarbound} is independent of $\delta$, and the right hand side of \eqref{eq:h+varphi+starbound} drops the exponential part of $h_{+}(s,r)$ anyway, and so we get the same bound for $h_{-}(s,r)$.  Since the estimates on the weight functions are identical, we have that 
the contributions to $E_{\text{Maass}}^{-}$ from $|t_j| \leq R(TY)^{\varepsilon}$ immediately leads to \eqref{eq:EMaassbound1}.

We claim that $\int \varphi_{-}^{*}(s) h_{-}(s, r) ds$ is very small if $|r| \geq R (TY)^{\varepsilon}$.  We note that on the line $\text{Re}(s) = \sigma$, $\text{Re}(u) = \varepsilon$ (which we assume avoids any pole of a gamma function), we have
\begin{equation}
\label{eq:EMaass-bound}
 \int_{(\sigma)} \varphi_{-}^{*}(s) h_{-}(s,r) ds \ll |r|^{2\sigma -2} Y^{\sigma} m^{\half - \sigma} R^{\frac32 - \sigma} T^{\frac12 - \sigma} (TY)^{\varepsilon} =  \frac{m^{\frac12} T^{\frac12} R^{\frac32}}{r^2} (TY)^{\varepsilon} \Big(\frac{r^2 Y}{mRT} \Big)^{\sigma}.
\end{equation}
If $- \sigma > 0$ is very large, then this bound becomes very small unless $r^2 \leq \frac{mRT}{Y} (TY)^{\varepsilon} \leq R^2 (TY)^{\varepsilon}$. %, but we assumed that $|r| \geq R(TY)^{\varepsilon}$.  
Our original integral representation requires $\half < \sigma < 1$, and moving the contour far to the left crosses poles at $s-\half \pm ir = 0, -1, -2, \dots$.  However, these residues are also very small because these occur at $|t| = |r| \geq R(TY)^{\varepsilon}$, but $\widetilde{g}(\sigma + it)$ is very small for such $t$'s, by \eqref{eq:gbound}.  Thus $E_{\text{Maass}}^{-}$ satisfies the same bounds as $E_{\text{Maass}}^{+}$, as desired.

\subsection{Holomorphic forms}
Define  $E_{\text{holo}}$ analogously to $E^{\delta}_{\text{Maass}}$.
Following the argument of $E_{\text{Maass}}$, we arrive at
\begin{multline}
\label{eq:Eholodef}
E_{\text{holo}} = \sum_{\pm} \frac{1}{2 \pi i} \int_{(\varepsilon)} \frac{G(u)}{u} 2 (2\pi)^{\pm 2iT -1 - u} \sum_{k} \sum_{f \in B_{2k}} k^{o(1)}  \lambda_f(m) 
\\
L(\tfrac12 + \tfrac{u}{2},f) L(\tfrac12 + \tfrac{u}{2} \mp 2iT, f) \frac{1}{2 \pi i} \int_{(\sigma)} \varphi_{+}^*(s) \frac{\Gamma(k-1+s)}{\Gamma(k+1-s)} ds  du.
\end{multline}
Now we estimate this $s$-integral.  We first show that it is very small unless $k \ll R (TY)^{\varepsilon}$.  Suppose otherwise.  Then by Stirling and a trivial bound,
\begin{equation}
 \int_{(\sigma)} \varphi_{+}^*(s) \frac{\Gamma(k-1+s)}{\Gamma(k+1-s)} ds \ll_{\sigma} m^{1/2} R^{3/2} T^{1/2} k^{-2} (TY)^{\varepsilon} \Big(\frac{Y k^2}{TRm} \Big)^{\sigma}, 
\end{equation}
so by taking $\sigma < 0$ very far to the left (but with $\sigma  > 1-k$), we see that this bound can be made to be very small since $TRm/Y \leq R^2$.  
Note the similarity with \eqref{eq:EMaass-bound}. 
Then with the truncation $k \ll R (TY)^{\varepsilon}$, setting $\sigma = 1/2$, and using $L(\tfrac12 + iT, f) \ll (T + k)^{1/3 +\varepsilon}$, another result of Jutila-Motohashi \cite{JutilaMotohashi}, the bound on the holomorphic forms becomes equivalent to that of the Maass forms, except we can take $\theta = 0$ in this case since we have Deligne's bound.  That is, \eqref{eq:EMaassbound1} holds for $E_{\text{holo}}$.

\subsection{Diagonal term} The diagonal term, say $E_{\text{diag}}$, is easily checked, with $\sigma = 1/2 + \varepsilon$, to give the following bound, which is much smaller than \eqref{eq:EMaassbound1}:
\begin{equation}
\label{eq:EDiag}
 E_{\text{diag}} \ll m^{-1/2} Y^{1/2} R (TY)^{\varepsilon}.
\end{equation}

\subsection{Eisenstein contribution}
We continue with \eqref{eq:EEisdef}.
By absolute convergence, we can sum over $n$ first, getting now
\begin{multline}
E_{\text{Eis}} = \sum_{\pm} \frac{1}{2 \pi i} \int_{(2)} \frac{G(u)}{u}  2 (2\pi)^{\pm 2iT -1 - u}
\frac{1}{\pi} \intR \frac{\tau_{ir}(m)}{|\zeta(1+2ir)|^2} 
\\
L(\tfrac12 + \tfrac{u}{2}\mp 2iT, E_r)L(\tfrac12 + \tfrac{u}{2}, E_r)  
\frac{1}{2\pi i} \int_{(\sigma)} [\varphi_{+}^{*}(s) h_{+}(s,r) - \varphi_{-}^{*}(s) h_{-}(s,r)] ds dr du,
\end{multline}
where we recall the definition \eqref{eq:EisensteinSeriesLfunction}.  For ease of reference, we recall that
\begin{multline}
\label{eq:varphihdeltaformula}
 \varphi_{\delta}^{*}(s) h_{\delta}(s,r) = \frac12 \widetilde{g}(s) m^{\half - s} \Gamma(1-s+\tfrac{u}{2})  \Gamma(1-s+\tfrac{u}{2} \mp 2iT) \Gamma(s-\tfrac12 + ir) \Gamma(s-\tfrac12 - ir) 
\\
\times
\begin{cases}
 \cosh(\pi T) \sin(\pi s), \quad &\delta = + \\
 \cos(\pi (s-\tfrac{u}{2} \pm iT)) \cosh( \pi r), \quad &\delta = -. 
\end{cases}
\end{multline}
Now we move the $u$-integral to the inside and shift it to the line $\nu = \varepsilon$.  In contrast to the cusp form cases, there are poles at $u = 1 + 2ir$, $u=1-2ir$, $u = 1 +2ir \pm 2iT$, and $u = 1 - 2ir \pm 2iT$.  The bound on the new line is completely analogous to the bound on $E_{\text{Maass}}$ (i.e., \eqref{eq:EMaassbound1}), and in fact the estimates here are slightly better since the spectral measure of the Eisenstein series is smaller (by a factor of $R$), and we have the estimate $|\tau_{ir}(m)| \leq d(m)$ (``Ramanujan'').  This explains the second-listed error term in \eqref{eq:EEisbound}.

We examine the residues now.  The poles at $u = 1+ 2ir \pm 2iT$ and $u = 1- 2ir \pm 2iT$ give a very small contribution because the inner integral over $s$ is small unless $|r| \leq R (TY)^{\varepsilon} = o(T)$, and $G(u)$ is small at height $T$.  We work with the residue at $u=1+2ir$ as the other is similar.  By a calculation, we have that this residue, denoted say $E_{\text{Eis}}^{r}$ is
\begin{multline}
\label{eq:EEisresidue}
E_{\text{Eis}}^{r} = \sum_{\pm} \frac{1}{\pi} \intR \frac{G(1+2ir)}{1+2ir}  2 (2\pi)^{\pm 2iT -2-2ir}
 \frac{\tau_{ir}(m)}{\zeta(1-2ir)} \zeta(1 + 2ir \mp 2iT) \zeta(1  \mp 2iT) 
\\
\frac{1}{2\pi i} \int_{(\sigma)} \frac12 \widetilde{g}(s) m^{\half - s} \Gamma(\tfrac32-s+ir)  \Gamma(\tfrac32-s+ir \mp 2iT) \Gamma(s-\tfrac12 + ir) \Gamma(s-\tfrac12 - ir)  
\\
\times [\cosh(\pi T) \sin(\pi s) - \cos(\pi(s-\tfrac12 -ir \pm iT)) \cosh(\pi r)]
ds.
\end{multline}
With a cursory examination, it appears that the integrand apparently passes through a pole at $r = \pm T$ (from the pole of the zeta function), but note that the third line of \eqref{eq:EEisresidue} vanishes there.  This is the reason we combined the $+$ and $-$ cases.

Now we can fix $\sigma$ with $\frac34 < \sigma \leq 1$ and shift the $r$-contour to $\text{Im}(r) = 1/4$ without crossing any poles.
From the decay of $G$, we can truncate the $r$-integral at $(TY)^{\varepsilon}$, and then we obtain a bound
\begin{equation}
 E_{\text{Eis}}^{r} \ll m^{1-\sigma} T^{1/6} Y^{\sigma} T^{3/4-\sigma} (TY)^{\varepsilon} \int_{|t| \leq R (TY)^{\varepsilon}} (1 + |t|)^{\sigma - \frac54} dt \ll T^{\frac{11}{12} - \sigma} Y^{\sigma} m^{3/4-\sigma} R^{\sigma - \frac14} (TY)^{\varepsilon}.
\end{equation}
Taking $\sigma =3/4 + \varepsilon$, this becomes
\begin{equation}
 E_{\text{Eis}}^{r} \ll T^{1/6}  Y^{3/4}  R^{1/2} (TY)^{\varepsilon},
\end{equation}
which is the first-stated bound appearing in \eqref{eq:EEisbound}.

This completes the proof of Theorem \ref{thm:SCS}.

\end{document}